\documentclass[11pt]{amsart}   	
\usepackage[margin=0.9in]{geometry}             
\usepackage[dvipsnames]{xcolor}

\usepackage{graphicx}				
\usepackage{amssymb,mathrsfs,latexsym}
\usepackage{amsthm,amsmath,stmaryrd}
\usepackage{tikz}
\usepackage{tikz-cd}
\usepackage{tikz-3dplot}
\tdplotsetmaincoords{70}{120}

\usepackage{accents,upgreek,enumerate}
\usepackage{bm}
\usepackage{mathtools}
\usepackage[all]{xy}
\usepackage{caption}
\usepackage[euler-digits]{eulervm}
\usepackage[normalem]{ulem}

\usepackage{ wasysym }

\usetikzlibrary{shadings,fadings}

\usepackage{setspace}
\tikzset{
  commutative diagrams/.cd, 
  arrow style=tikz, 
  diagrams={>=stealth}
}

\newtheorem{innercustomthm}{Theorem}
\newenvironment{customthm}[1]
  {\renewcommand\theinnercustomthm{#1}\innercustomthm}
  {\endinnercustomthm}
  
  \newtheorem{innercustomcor}{Corollary}
\newenvironment{customcor}[1]
  {\renewcommand\theinnercustomcor{#1}\innercustomcor}
  {\endinnercustomcor}

  \makeatletter
\def\@tocline#1#2#3#4#5#6#7{\relax
  \ifnum #1>\c@tocdepth 
  \else
    \par \addpenalty\@secpenalty\addvspace{#2}%
    \begingroup \hyphenpenalty\@M
    \@ifempty{#4}{%
      \@tempdima\csname r@tocindent\number#1\endcsname\relax
    }{%
      \@tempdima#4\relax
    }%
    \parindent\z@ \leftskip#3\relax \advance\leftskip\@tempdima\relax
    \rightskip\@pnumwidth plus4em \parfillskip-\@pnumwidth
    #5\leavevmode\hskip-\@tempdima
      \ifcase #1
       \or\or \hskip 1em \or \hskip 2em \else \hskip 3em \fi%
      #6\nobreak\relax
    \dotfill\hbox to\@pnumwidth{\@tocpagenum{#7}}\par
    \nobreak
    \endgroup
  \fi}
\makeatother

\usetikzlibrary{calc}
\usetikzlibrary{fadings}
\usetikzlibrary{decorations.pathmorphing}
\usetikzlibrary{decorations.pathreplacing}

\newcounter{marginnote}
\DeclareMathAlphabet{\mathpzc}{OT1}{pzc}{m}{it}

\usepackage[backref=page]{hyperref}
\hypersetup{
  colorlinks   = true,          
  urlcolor     = blue,          
  linkcolor    = violet,          
  citecolor   = VioletRed             
}

\newtheorem{theorem}{Theorem}[subsection]
\newtheorem{corollary}[theorem]{Corollary}
\newtheorem{lemma}[theorem]{Lemma}
\newtheorem{proposition}[theorem]{Proposition}

\newtheorem{quasi-theorem}[theorem]{Quasi-Theorem}

\theoremstyle{definition}
\newtheorem{definition}[theorem]{Definition}

\newtheorem{remark}[theorem]{Remark}

\newtheorem{construction}[theorem]{Construction}

\newtheorem{blank remark}[theorem]{}

\newtheorem{not1}[theorem]{Notation}

\newcommand{\A}{{\mathbb{A}}}

\newcommand{\NN} {{\mathbb N}}		
\newcommand{\PP}{\mathbb{P}}         
\newcommand{\QQ} {{\mathbb Q}}		
\newcommand{\RR} {{\mathbb R}}		
\newcommand{\ZZ} {{\mathbb Z}}

\def\setminus{\smallsetminus}

\newcommand{\cal}{\mathcal}

\def\cD{{\cal D}}

\def\cM{{\cal M}}

\def\cX{{\cal X}}

\def\fM{\mathfrak{M}}

\newcommand{\Mbar}{\overline{\mathsf M}\vphantom{\cM}}

\newcommand{\Spec}{\operatorname{Spec}}

\newcommand{\vr}{\vec{r}}

\begin{document}

\title{Gromov--Witten theory, degenerations, and the tautological ring}

\author{Davesh Maulik {\it \&} Dhruv Ranganathan}

\address{Davesh Maulik \\ Department of Mathematics\\
Massachusetts Institute of Technology, Cambridge, MA, USA}
\email{\href{mailto:maulik@mit.edu}{maulik@mit.edu}}

\address{Dhruv Ranganathan \\ Department of Pure Mathematics {\it \&} Mathematical Statistics\\
University of Cambridge, Cambridge, UK}
\email{\href{mailto:dr508@cam.ac.uk}{dr508@cam.ac.uk}}

\begin{abstract}
Gromov--Witten (GW) theory produces Chow and cohomology classes on the moduli of curves, and there are several conjectures/speculations about their relation to the tautological ring. We develop new degeneration techniques to address these.

In Chow, we show that GW cycles of complete intersections in products of projective spaces (and more generally a broad class of toric varieties) with restricted insertions are tautological. This gives significant evidence for a 2010 speculation of Pandharipande that GW cycles of varieties over $\overline{\QQ}$ are tautological. In particular, the 0-cycle for curves on the quintic threefold is proportional to a zero stratum in $\Mbar_g$.

In cohomology, we show that in normal crossings degenerations, GW classes of the general fiber lie in the span of absolute GW classes of the special fiber strata. This confirms a 2006 conjecture of Levine–-Pandharipande for targets that degenerate into elementary pieces, including complete intersections in products of projective spaces and many toric varieties.

Our proofs rely on several reconstruction theorems in logarithmic GW theory, which make the logarithmic degeneration formula an inductive tool to compute GW cycles via snc degenerations. We prove a folklore conjecture that logarithmic GW cycles of a pair are determined by absolute invariants of the strata.  We prove a conjecture of Urundolil Kumaran and the second author that GW cycles of toric pairs are tautological, and analogous results for broken toric bundles. We also develop tools to study GW cycles with vanishing cohomology and strengthen the logarithmic degeneration formula to allow iteration.
	\end{abstract}

\maketitle

\setcounter{tocdepth}{1}
\tableofcontents
\newpage

\section*{Introduction}

\subsection{Cycles from Gromov--Witten theory} Let $X$ be a smooth projective variety. The moduli space $\Mbar_{g,n}(X,\beta)$ of stable maps to $X$ from $n$ pointed genus $g$ curves with class $\beta$ is a proper Deligne--Mumford stack with morphisms
\[
\begin{tikzcd}
\Mbar_{g,n}(X,\beta)\arrow{r}{\mathsf{ev}}\arrow[swap]{d}{\pi} & X^n\\
\Mbar_{g,n}.&
\end{tikzcd}
\]
It carries a virtual fundamental class:
\[
[\Mbar_{g,n}(X,\beta)]^{\sf vir} \in \mathsf{CH}_{\sf vdim}(\Mbar_{g,n}(X,\beta);\QQ). 
\]
The cycle map induces a virtual fundamental class in homology. The virtual dimension is equal to $(\dim X-3)(1-g)+c_1(X)\cdot\beta+n$. 

The virtual class is a correspondence that relates the cohomology and cycle theory of the moduli space of curves with that of any particular $X$. The {\it Gromov--Witten (GW) cycles} of $X$ in Chow (resp. cohomology) are cycles of the form
\[
\pi_\star\left(\mathsf{ev}^\star(\alpha)\cap [\Mbar_{g,n}(X,\beta)]^{\sf vir}\right), \ \ \textrm{for } \alpha\in {\sf CH}^\star(X^n) \ \ (\textnormal{resp. in}\   {\sf H}^\star(X^n)). 
\]
The class $\alpha$ is called an {\it insertion}. These cycles play a basic role in the study of the Chow and cohomology of the moduli space of curves~\cite{FP00,GV05,Jan17b,Pan12}. 

\subsection{Speculations} Our study of GW cycles is guided by well-known speculations/conjectures. The rational Chow rings of the moduli spaces of curves have {\it tautological subalgebras}, denoted $\mathsf R^\star(\Mbar_{g,n})$. These are the smallest subalgebras, defined simultaneously for all $(g,n)$, that are closed under pushforward under forgetful and gluing maps. The tautological ring in cohomology is its image under the cycle map, so we obtain
\[
\mathsf R^\star(\Mbar_{g,n})\subset \mathsf {CH}^\star(\Mbar_{g,n};\QQ) \ \ \textnormal{and} \ \ \mathsf {RH}^\star(\Mbar_{g,n})\subset \mathsf {H}^\star(\Mbar_{g,n};\QQ).
\]
The subrings are small, come with an explicit set of additive generators~\cite{GP03} and a conjecturally complete set of relations~\cite{PPZ}, and are amenable to analysis via computer algebra systems~\cite{DSvZ}. The tautological Chow ring has finite dimension as a $\mathbb Q$-vector space~\cite{GP03}.

However, most tautological groups, either in cohomology or in Chow, are either known or expected to be very far from the full group. The cohomology of $\Mbar_{g,n}$ can carry holomorphic forms, so even its Chow group of $0$-cycles can be uncountable. The space $\Mbar_{g}$ has nonzero odd cohomology for sufficiently large $g$, which implies the rational Chow ring is uncountable. The tautological cohomology of $\Mbar_g$ misses all odd cohomology classes, even all non-Hodge classes. It is also known that the cohomology of $\Mbar_g$ has non-tautological algebraic cycles once $g\geq 12$. See~\cite{CanningHolomorphic,CLP23,FP11,GP03,vZ18} for a review and references. 

Nevertheless, there has long been speculation that GW cycles usually lie in the tautological ring. A precise conjecture in cohomology first appeared in work of Levine--Pandharipande~\cite{LP09}, and was strengthened in~\cite{ABPZ}. 

\noindent
{\bf Cohomology Speculation.} {\it All Gromov--Witten cycles $\pi_\star\left(\mathsf{ev}^\star(\alpha)\cap [\Mbar_{g,n}(X,\beta)]^{\sf vir}\right)$ lie in the tautological subring of $H^\star(\Mbar_{g,n};\QQ)$ of the moduli space of curves.}

The speculation implies a remarkable amount of vanishing that is difficult to explain geometrically. For example, by picking $\alpha$ appropriately, we can arrange the GW cycle to be of odd cohomological degree. The conjecture implies that these cycles vanish. 

The evidence is fairly limited. Before this work, the speculation was verified for toric varieties and homogeneous spaces~\cite{GP99} and complete intersections in projective space~\cite{ABPZ,FP,Jan17}. The statement is also compatible with products and blowups~\cite{Beh97,Fan21,HLR08}.

The statement in Chow is motivated in part by the Bloch--Beilinson conjectures. It was first stated in 2010 by Pandharipande~\cite[Question 2]{PanOpen}, although the line of inquiry goes back further~\cite{FP}.

\noindent
{\bf Chow Speculation.} {\it If $X$ is defined over $\overline \QQ$, then the pushforward class $\pi_\star([\Mbar_{g,n}(X,\beta)]^{\sf vir})$ lies in the tautological subring of ${\sf CH}^\star(\Mbar_{g,n};\QQ)$.}

One can also consider the stronger statement, with insertions also defined over $\overline\QQ$. The Chow speculation has even less evidence. The statement with insertions holds for toric varieties and homogeneous spaces by~\cite{GP99} and partial results hold for curves via~\cite{FP}. The speculation is compatible with products and blowups. Without the assumption that $X$ is defined over $\overline\QQ$, the statement is false whenever $\mathsf{CH}_0(\Mbar_g)$ contains non-tautological class. The statement with insertions is false whenever $\mathsf{CH}_0(\Mbar_{g,n})$ has a non-tautological class\footnote{The group $\mathsf{CH}_0(\Mbar_{1,11})$ has non-tautollogical classes, while $\mathsf{CH}_0(\Mbar_{g})$ is expected to have non-tautological classes for large $g$. We comment further on this briefly.}. Both failures happen when $X$ is a curve, see~\cite[Section~0.4]{PZ19}. 

\subsection{Results: GW cycles and the tautological ring} We study these speculations via logarithmic degeneration techniques. We first state our results in Chow and then in cohomology. 

Given a subvariety $X\subset \mathbb P$, we refer to Chow cohomology classes that are restricted from $\mathbb P$ as {\it ambient classes}. 

\begin{customthm}{A}\label{thm: tautological-chow}
Let $X\subset \mathbb P$ be a smooth complete intersection in a product of projective spaces. The Gromov--Witten cycles of $X$ in Chow with ambient insertions lie in the tautological subring of $\mathsf{CH}^\star(\Mbar_{g,n})$. 

In fact, the result holds for $X$ a complete intersection in a toric variety, such that each hypersurface comprising the complete intersection has breakable Newton polytope. 
\end{customthm}

A Newton polytope is {\it breakable} if it admits a regular unimodular subdivision, see Figure~\ref{fig: unimod-sub}.\footnote{This class of polytopes is closed under dilations, products, joins, and certain fiber products. Each new example of a polytope that admits a unimodular triangulation therefore significantly enlarges the class of breakable polytopes. From the tautology that a unimodular simplex admits a unimodular triangulation, one obtains all complete intersections in products of projective space. Other examples include hypersimplices, generalized permutohedra, and order polytopes~\cite{HPPS21}.} Geometrically, it means that the ambient toric variety together with the toric variety, can be degenerated to a union of projective spaces, such that the hypersurface become linear.  

\begin{figure}[h!]\begin{tikzpicture}[tdplot_main_coords]

  \begin{scope}[scale=1.2]
    \definecolor{softlavender}{RGB}{223,210,255}

    \filldraw[fill=softlavender,opacity=0.2, draw=none] (0,0,0) -- (3,0,0) -- (0,3,0) -- (0,0,3) -- cycle;

    \newcommand{\drawtetra}[4]{
      \filldraw[fill=softlavender, draw=black, opacity=0.2, line width=0.3pt]
        (#1) -- (#2) -- (#3) -- cycle;
      \filldraw[fill=softlavender, draw=black, opacity=0.2, line width=0.3pt]
        (#1) -- (#2) -- (#4) -- cycle;
      \filldraw[fill=softlavender, draw=black, opacity=0.2, line width=0.3pt]
        (#1) -- (#3) -- (#4) -- cycle;
      \filldraw[fill=softlavender, draw=black, opacity=0.2, line width=0.3pt]
        (#2) -- (#3) -- (#4) -- cycle;
    }

    \coordinate (O) at (0,0,0);
    \coordinate (X) at (3,0,0);
    \coordinate (Y) at (0,3,0);
    \coordinate (Z) at (0,0,3);

    \foreach \a in {0,...,2} {
      \foreach \b in {0,...,2} {
        \foreach \c in {0,...,2} {
          \pgfmathtruncatemacro{\s}{\a+\b+\c}
          \ifnum\s<3
            \pgfmathsetmacro{\aB}{\a+1}
            \pgfmathsetmacro{\bC}{\b+1}
            \pgfmathsetmacro{\cD}{\c+1}
            \coordinate (A) at (\a,\b,\c);
            \coordinate (B) at (\aB,\b,\c);
            \coordinate (C) at (\a,\bC,\c);
            \coordinate (D) at (\a,\b,\cD);
            \drawtetra{A}{B}{C}{D}
          \fi
        }
      }
    }

    \foreach \i in {0,...,3} {
      \foreach \j in {0,...,3} {
        \foreach \k in {0,...,3} {
          \pgfmathtruncatemacro{\s}{\i+\j+\k}
          \ifnum\s<4
            \fill (\i,\j,\k) circle (0.3pt);
          \fi
        }
      }
    }

    \draw[thick, densely dotted, color=violet!80] (O) -- (X);
    \draw[thick, densely dotted, color=violet!80] (O) -- (Y);
    \draw[thick, color=violet!80] (Y) -- (X);
    \draw[thick, densely dotted, color=violet!80] (O) -- (Z);
    \draw[thick, color=violet!80] (X) -- (Z);
    \draw[thick, color=violet!80] (Y) -- (Z);
  \end{scope}

  \begin{scope}[xshift=6cm, yshift=0.5cm,scale=1.8]
    \definecolor{softlavender}{RGB}{223,210,255}
    \newcommand{\drawtetra}[4]{
      \filldraw[fill=softlavender, draw=black, opacity=0.2, line width=0.3pt]
        (#1) -- (#2) -- (#3) -- cycle;
      \filldraw[fill=softlavender, draw=black, opacity=0.2, line width=0.3pt]
        (#1) -- (#2) -- (#4) -- cycle;
      \filldraw[fill=softlavender, draw=black, opacity=0.2, line width=0.3pt]
        (#1) -- (#3) -- (#4) -- cycle;
      \filldraw[fill=softlavender, draw=black, opacity=0.2, line width=0.3pt]
        (#2) -- (#3) -- (#4) -- cycle;
    }

    \foreach \a in {0,...,2} {
      \foreach \b in {0,...,2} {
        \foreach \c in {0,...,2} {
          \pgfmathtruncatemacro{\sum}{\b+\c}
          \ifnum\sum>2
          \else
            \fill (\a,\b,\c) circle (0.3pt);
          \fi
        }
      }
    }
    \foreach \a in {0,1} {
      \foreach \b in {0,...,1} {
        \foreach \c in {0,...,1} {
          \pgfmathtruncatemacro{\sum}{\b+\c}
          \ifnum\sum>1
          \else
            \coordinate (A) at (\a,\b,\c);
            \coordinate (B) at (\a,\b+1,\c);
            \coordinate (C) at (\a,\b,\c+1);

            \coordinate (D) at ({\a+1},\b,\c);
            \coordinate (E) at ({\a+1},\b+1,\c);
            \coordinate (F) at ({\a+1},\b,\c+1);

            \drawtetra{A}{B}{C}{D};
            \drawtetra{B}{C}{E}{F};
            \drawtetra{B}{D}{E}{F};
            \drawtetra{C}{D}{E}{F};
            \drawtetra{B}{C}{D}{F};
            \drawtetra{C}{E}{D}{F};
          \fi
        }
      }
    }

    \coordinate (O) at (0,0,0);
    \coordinate (X) at (2,0,0);
    \coordinate (Y) at (0,2,0);
    \coordinate (Z) at (0,0,2);
    \coordinate (X1) at (2,0,0);
    \coordinate (Y1) at (2,2,0);
    \coordinate (Z1) at (2,0,2);

    \draw[thick, densely dotted, violet!80] (O) -- (Y);
    \draw[thick, densely dotted, violet!80] (O) -- (Z);
    \draw[thick, color=violet!80] (Y) -- (Z);

    \draw[thick, violet!80] (X) -- (Y1);
    \draw[thick, violet!80] (X) -- (Z1);
    \draw[thick, color=violet!80] (Y1) -- (Z1);
    \draw[thick, densely dotted, color=violet!80] (O) -- (X);
    \draw[thick, color=violet!80] (Y) -- (Y1);
    \draw[thick, color=violet!80] (Z) -- (Z1);
  \end{scope}

\end{tikzpicture}

\caption{Products of dilated simplices are examples of breakable Newton polytopes and correspond to degenerations of hypersurfaces in products of projective space. Unimodular triangulations of $3\cdot\Delta_3$ and $2\cdot\Delta_2\times2\cdot\Delta_1$ are shown.}\label{fig: unimod-sub}
\end{figure}
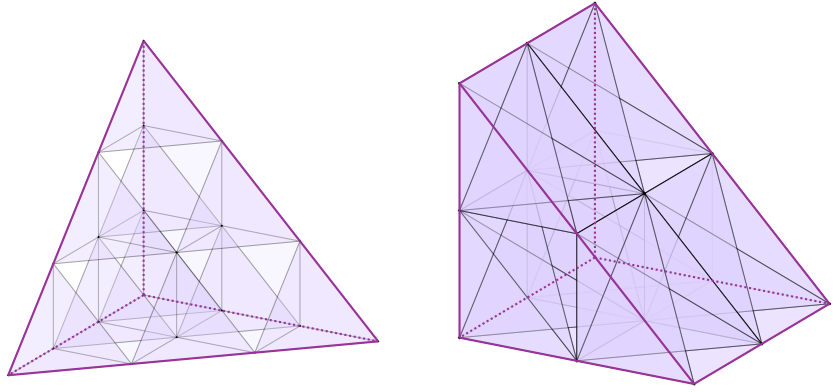

Logarithmic techniques seem essential in Chow -- traditional ``double point'' degeneration techniques~\cite{Li01,Li02,MP06} are often ineffective, since the diagonal in the double locus may not admit a K\"unneth decomposition. More extreme simple normal crossings degenerations often have strata that admit Chow--K\"unneth decompositions because they are {\it linear} in the sense of~\cite{Tot14}.

We note the following particular corollary:

\begin{customcor}{B}\label{cor: chowquintic}
Let $X_5$ be a quintic threefold in $\mathbb P^4$, then the pushforward of 
$[\Mbar_g(X_5,d)]^{\sf vir}$ in 
$\mathsf {CH}_{0}(\Mbar_{g})$
is proportional to a $0$-stratum. 
\end{customcor}

Since $\mathsf R_0(\Mbar_g)$ is $\mathbb Q$, the Chow class is known via the Gromov--Witten invariant. 

The Chow group of $0$-cycles ${\sf CH}_0(\Mbar_{g,n})$ is expected to be of infinite rank except for finitely many $(g,n)$, see~\cite[Speculation~1.1]{SP20}. 

It is known that the Chow groups of $0$-cycles of 
\[
\Mbar_{1,\geq 11}, \Mbar_{2,\geq 14}, \Mbar_{3,\geq 15} \ \textnormal{and } \Mbar_{4,\geq 16}
\]
have infinite rank. We produce geometrically interesting cycles here, for example, as follows. The quartic threefold $X_4$ is an irrational, unirational Fano. Curves of genus $g$ curves of degree $d$ move in virtual dimension $d$. By imposing point and line conditions we obtain $0$-cycles, in any of these infinite rank Chow groups, that are proportional to $0$-strata. There are many such examples. 

We move on to cohomology. 

\begin{customthm}{C}\label{thm: tautological-gw}
Let $\mathcal X$ be a smooth projective variety equipped with a flat and proper morphism
\[
\mathcal X\to B
\]
to a smooth proper curve. Assume the singular fibers are simple normal crossings divisors. Let $b\in B$ be a regular value. Then the cohomological Gromov--Witten cycles of $\cX_b$ lie in the $\mathsf {RH}^\star(\Mbar_{g,n})$-linear span of the Gromov--Witten cycles of the strata of any singular fiber $\cX_0$. 

In particular, if the GW cycles of the strata of the degeneration are tautological, then the cycles associated to the general fiber are also tautological. 
\end{customthm}

The statement of this theorem does not involve logarithmic structures at all; it is entirely about the absolute GW theory of the strata. The proof will, of course, use logarithmic GW theory. 

The theorem involves both constructive algorithms and purely theoretical arguments. We discuss the ingredients in Section~\ref{sec: intro-details} below.

We deduce the cohomology speculation for a large class of targets. 

\begin{customcor}{D}\label{cor: breakable}
Let $X$ be a smooth complete intersection of hypersurfaces in a toric variety with breakable Newton polytopes. All Gromov--Witten cycles of $X$ are tautological in cohomology. 
\end{customcor}

To the best of our knowledge, the results of the corollary are new for hypersurfaces in products of projective spaces. For hypersurfaces in projective spaces, see~\cite[Theorem~E]{ABPZ}. 

The breakability condition for toric hypersurfaces produces a broad class of examples that can be worked with concretely, but we expect it to be far from exhaustive. 

A polytope can be breakable, but not splittable into two, see Figure~\ref{fig: three-not-two} for an example. This fact produces examples of hypersurfaces without obvious double point degenerations but with good snc degenerations. 

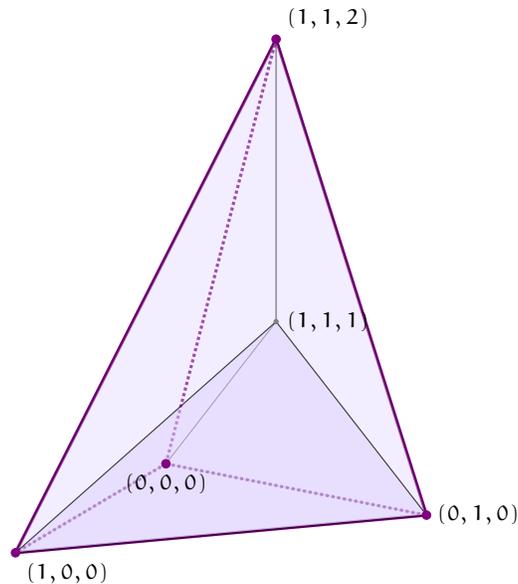
\begin{figure}[h!]
\tdplotsetmaincoords{70}{120}
\begin{tikzpicture}[tdplot_main_coords, scale=4]
\definecolor{softlavender}{RGB}{223,210,255}

\coordinate (A) at (0,0,0);
\coordinate (B) at (1,0,0);
\coordinate (C) at (0,1,0);
\coordinate (D) at (1,1,2);
\coordinate (P) at (1,1,1);

\draw[very thick, violet, densely dotted] (A) -- (B);
\draw[very thick, violet, densely dotted] (A) -- (C);
\draw[very thick, violet, densely dotted] (A) -- (D);

\draw[very thick, violet] (B) -- (D);
\draw[very thick, violet] (D) -- (C);
\draw[very thick, violet] (B) -- (C);

\filldraw[fill=softlavender, opacity=0.15] (A) -- (B) -- (C) -- cycle;
\filldraw[fill=softlavender, opacity=0.15] (A) -- (B) -- (P) -- cycle;
\filldraw[fill=softlavender, opacity=0.15] (A) -- (C) -- (P) -- cycle;
\filldraw[fill=softlavender, opacity=0.4] (B) -- (C) -- (P) -- cycle;

\filldraw[fill=softlavender, opacity=0.2] (B) -- (C) -- (D) -- cycle;
\filldraw[fill=softlavender, opacity=0.2] (B) -- (D) -- (P) -- cycle;
\filldraw[fill=softlavender, opacity=0.2] (C) -- (D) -- (P) -- cycle;
\filldraw[fill=softlavender, opacity=0.2] (B) -- (C) -- (P) -- cycle;

\draw[black, opacity=0.2] (A) -- (P);
\draw[ black, opacity=0.3] (B) -- (P);
\draw[ black, opacity=0.3] (C) -- (P);
\draw[ black, opacity=0.3] (D) -- (P);

\foreach \pt in {A, B, C, D}
  \filldraw[violet] (\pt) circle (0.4pt);

\filldraw[gray] (P) circle (0.2pt);

\node[anchor=north, font=\scriptsize] at (A) {$(0,0,0)$};
\node[anchor=north west, font=\scriptsize] at (B) {$(1,0,0)$};
\node[anchor=west, font=\scriptsize] at (C) {$(0,1,0)$};
\node[anchor=south west, font=\scriptsize] at (D) {$(1,1,2)$};
\node[anchor=west, font=\scriptsize] at (P) {$(1,1,1)$};

\end{tikzpicture}
\caption{An example of a polytope that is breakable but cannot be cut into just two pieces. This gives rise to an orbifold snc degeneration and can be used to produce four dimensional polytopes with analogous properties.}\label{fig: three-not-two}
\end{figure}

In contrast to Theorem~\ref{thm: tautological-chow}, there is no restriction on the insertions in these results. Theorem~\ref{thm: tautological-gw} is also significantly stronger than the corollary, which is perhaps the most elementary consequence of the result. The theorem itself applies much more broadly. For example, we expect that it implies the cohomology speculation for all Fano threefolds~\cite{CCGK16} and several hundred Fano fourfolds~\cite{CKP15}, but we discuss this in future work.

\subsection{Methods: from degeneration to reconstruction}\label{sec: intro-details} The principle that GW theory can be reconstructed from the strata of a degeneration has been a central tool in the subject for over two decades~\cite{FP,LR01,OP06}. One of the goals of this paper is to see how far these results can be pushed using advances in logarithmic GW theory. Our principal reconstruction result, stated in Theorem~\ref{thm: tautological-gw}, might be viewed as an end-state for this line of inquiry. 

We make use of much of the development of the subject over the last two decades. 
We summarize this before explaining what has to be done in order to prove the main results of this paper.

The first inputs are the foundational theory in the papers~\cite{AC11,Che10,GS13} which relates the GW theory of a smooth fiber $\cX_\eta$ to the {\it logarithmic} GW theory of a singular fiber $\cX_0$, and the decomposition formula that relates it to contributions from certain tropical curves~\cite{ACGS15}. These in turn build on J. Li's foundational work on the double point degeneration formula~\cite{Li01,Li02} and prior developments in symplectic geometry~\cite{IP01,IP04,LR01}. See also B. Parker's work on exploded manifolds~\cite{Par11}.

The next step is our previous work on the degeneration formula. \footnote{A parallel development of degeneration formulas using {\it punctured} Gromov--Witten theory has been undertaken by Abramovich--Chen--Gross--Siebert~\cite{ACGS17}. B. Parker has also proved gluing formulas in~\cite{Par17a}. For the reasons outlined in~\cite[Section~0.6]{MR23} we are not aware of how to trade our formalism for theirs in pursuing the results here.} The formula relates the tropical contributions to the logarithmic GW theory of curves in targets with prescribed tangency along a simple normal crossings divisor. The pairs that arise are compactified torus bundles over strata in $\cX_0$, see~\cite{MR23,R19} and Figure~\ref{fig: degeneration-picture} below. This is the starting point for the present paper.

\begin{figure}[h!]
\tikzset{every picture/.style={line width=0.75pt}} 

\begin{tikzpicture}[x=0.75pt,y=0.75pt,yscale=-1,xscale=1]

\draw  [fill={rgb, 255:red, 223; green, 210; blue, 255 }  ,fill opacity=0.4 ] (146,2291) .. controls (166,2281) and (270.5,2262) .. (250.5,2282) .. controls (244.77,2287.73) and (243.39,2297.79) .. (244.81,2309.96) .. controls (248.33,2340.27) and (269.23,2383.59) .. (283.5,2405) .. controls (303.5,2435) and (88.5,2482) .. (134.5,2417) .. controls (180.5,2352) and (126,2301) .. (146,2291) -- cycle ;
\draw    (118,2503) .. controls (158,2473) and (466.5,2533) .. (506.5,2503) ;
\draw [color={rgb, 255:red, 0; green, 0; blue, 0 }  ,draw opacity=1 ]   (421.5,2510) ;
\draw [shift={(421.5,2510)}, rotate = 0] [color={rgb, 255:red, 0; green, 0; blue, 0 }  ,draw opacity=1 ][fill={rgb, 255:red, 0; green, 0; blue, 0 }  ,fill opacity=1 ][line width=0.75]      (0, 0) circle [x radius= 3.35, y radius= 3.35]   ;
\draw  [color={rgb, 255:red, 208; green, 2; blue, 27 }  ,draw opacity=1 ][line width=0.75] [line join = round][line cap = round] (166.5,2314) .. controls (182.67,2314) and (194.16,2365.1) .. (170.5,2358) .. controls (165.83,2356.6) and (161.56,2345.98) .. (167.5,2344) .. controls (186.03,2337.82) and (208.56,2356.3) .. (212.5,2376) .. controls (215.91,2393.07) and (213.72,2406.11) .. (193.5,2403) .. controls (187.7,2402.11) and (190.59,2388.91) .. (193.5,2386) .. controls (194.63,2384.87) and (198.55,2383.14) .. (200.5,2383) .. controls (230.28,2380.87) and (226.26,2405.25) .. (246.5,2412) .. controls (258.28,2415.93) and (261.87,2410) .. (267.5,2410) ;
\draw  [fill={rgb, 255:red, 223; green, 210; blue, 255 }  ,fill opacity=0.4] (396,2284) .. controls (416,2274) and (478,2237) .. (501,2295) .. controls (524,2353) and (453,2494) .. (405,2441) .. controls (357,2388) and (368,2430) .. (346,2358) .. controls (324,2286) and (376,2294) .. (396,2284) -- cycle ;
\draw [color={rgb, 255:red, 74; green, 13; blue, 182 }  ,draw opacity=1 ]   (445.5,2266.04) .. controls (447.12,2267.75) and (447.08,2269.41) .. (445.37,2271.04) .. controls (443.66,2272.67) and (443.62,2274.33) .. (445.24,2276.04) .. controls (446.86,2277.75) and (446.82,2279.41) .. (445.11,2281.03) .. controls (443.4,2282.66) and (443.36,2284.32) .. (444.99,2286.03) .. controls (446.61,2287.74) and (446.57,2289.4) .. (444.86,2291.03) .. controls (443.15,2292.66) and (443.11,2294.32) .. (444.73,2296.03) .. controls (446.35,2297.74) and (446.31,2299.4) .. (444.6,2301.03) .. controls (442.89,2302.66) and (442.85,2304.32) .. (444.47,2306.03) .. controls (446.09,2307.74) and (446.05,2309.4) .. (444.34,2311.02) .. controls (442.63,2312.65) and (442.59,2314.31) .. (444.22,2316.02) .. controls (445.84,2317.73) and (445.8,2319.39) .. (444.09,2321.02) .. controls (442.38,2322.65) and (442.34,2324.31) .. (443.96,2326.02) .. controls (445.58,2327.73) and (445.54,2329.39) .. (443.83,2331.02) .. controls (442.12,2332.65) and (442.08,2334.31) .. (443.7,2336.02) -- (443.65,2337.96) -- (443.65,2337.96)(442.5,2265.96) .. controls (444.12,2267.67) and (444.08,2269.34) .. (442.37,2270.96) .. controls (440.66,2272.59) and (440.62,2274.25) .. (442.24,2275.96) .. controls (443.87,2277.67) and (443.83,2279.33) .. (442.12,2280.96) .. controls (440.41,2282.58) and (440.37,2284.24) .. (441.99,2285.95) .. controls (443.61,2287.66) and (443.57,2289.32) .. (441.86,2290.95) .. controls (440.15,2292.58) and (440.11,2294.24) .. (441.73,2295.95) .. controls (443.35,2297.66) and (443.31,2299.32) .. (441.6,2300.95) .. controls (439.89,2302.58) and (439.85,2304.24) .. (441.47,2305.95) .. controls (443.1,2307.66) and (443.06,2309.32) .. (441.35,2310.95) .. controls (439.64,2312.58) and (439.6,2314.24) .. (441.22,2315.95) .. controls (442.84,2317.66) and (442.8,2319.32) .. (441.09,2320.94) .. controls (439.38,2322.57) and (439.34,2324.23) .. (440.96,2325.94) .. controls (442.58,2327.65) and (442.54,2329.31) .. (440.83,2330.94) .. controls (439.12,2332.57) and (439.08,2334.23) .. (440.7,2335.94) -- (440.65,2337.89) -- (440.65,2337.89) ;
\draw [color={rgb, 255:red, 74; green, 13; blue, 182 }  ,draw opacity=1 ]   (443.21,2338.99) .. controls (443.2,2341.35) and (442.02,2342.52) .. (439.66,2342.51) .. controls (437.3,2342.5) and (436.12,2343.68) .. (436.11,2346.04) .. controls (436.1,2348.39) and (434.92,2349.57) .. (432.57,2349.56) .. controls (430.21,2349.55) and (429.03,2350.72) .. (429.02,2353.08) .. controls (429.01,2355.44) and (427.83,2356.62) .. (425.47,2356.61) .. controls (423.12,2356.6) and (421.94,2357.78) .. (421.93,2360.13) .. controls (421.92,2362.49) and (420.74,2363.67) .. (418.38,2363.66) -- (417.72,2364.31) -- (417.72,2364.31)(441.1,2336.86) .. controls (441.09,2339.22) and (439.91,2340.4) .. (437.55,2340.39) .. controls (435.19,2340.38) and (434.01,2341.55) .. (434,2343.91) .. controls (433.99,2346.27) and (432.81,2347.44) .. (430.45,2347.43) .. controls (428.09,2347.42) and (426.91,2348.6) .. (426.91,2350.96) .. controls (426.9,2353.32) and (425.72,2354.49) .. (423.36,2354.48) .. controls (421,2354.47) and (419.82,2355.64) .. (419.81,2358) .. controls (419.8,2360.36) and (418.62,2361.54) .. (416.26,2361.53) -- (415.6,2362.18) -- (415.6,2362.18) ;
\draw [color={rgb, 255:red, 74; green, 13; blue, 182 }  ,draw opacity=1 ]   (416.72,2364.75) .. controls (415.13,2366.48) and (413.46,2366.55) .. (411.73,2364.96) .. controls (410,2363.37) and (408.33,2363.44) .. (406.73,2365.17) .. controls (405.14,2366.9) and (403.47,2366.97) .. (401.74,2365.38) .. controls (400,2363.79) and (398.33,2363.86) .. (396.74,2365.6) .. controls (395.15,2367.33) and (393.48,2367.4) .. (391.75,2365.81) .. controls (390.02,2364.22) and (388.35,2364.29) .. (386.75,2366.02) .. controls (385.15,2367.75) and (383.48,2367.82) .. (381.75,2366.23) .. controls (380.01,2364.64) and (378.35,2364.71) .. (376.76,2366.45) .. controls (375.16,2368.18) and (373.49,2368.25) .. (371.76,2366.66) .. controls (370.03,2365.07) and (368.36,2365.14) .. (366.77,2366.87) .. controls (365.18,2368.61) and (363.51,2368.68) .. (361.77,2367.09) .. controls (360.04,2365.5) and (358.37,2365.57) .. (356.78,2367.3) -- (352.06,2367.5) -- (352.06,2367.5)(416.6,2361.75) .. controls (415,2363.48) and (413.33,2363.55) .. (411.6,2361.96) .. controls (409.87,2360.37) and (408.2,2360.44) .. (406.6,2362.17) .. controls (405.01,2363.91) and (403.35,2363.98) .. (401.61,2362.39) .. controls (399.88,2360.8) and (398.21,2360.87) .. (396.61,2362.6) .. controls (395.02,2364.33) and (393.35,2364.4) .. (391.62,2362.81) .. controls (389.88,2361.22) and (388.21,2361.29) .. (386.62,2363.03) .. controls (385.03,2364.76) and (383.36,2364.83) .. (381.63,2363.24) .. controls (379.9,2361.65) and (378.23,2361.72) .. (376.63,2363.45) .. controls (375.04,2365.18) and (373.37,2365.25) .. (371.64,2363.66) .. controls (369.9,2362.07) and (368.23,2362.14) .. (366.64,2363.88) .. controls (365.05,2365.61) and (363.38,2365.68) .. (361.65,2364.09) .. controls (359.92,2362.5) and (358.25,2362.57) .. (356.65,2364.3) -- (351.94,2364.5) -- (351.94,2364.5) ;
\draw [color={rgb, 255:red, 74; green, 13; blue, 182 }  ,draw opacity=1 ]   (418.16,2363.26) .. controls (419.81,2364.94) and (419.8,2366.61) .. (418.12,2368.26) .. controls (416.44,2369.91) and (416.43,2371.58) .. (418.08,2373.26) .. controls (419.73,2374.94) and (419.72,2376.61) .. (418.04,2378.26) .. controls (416.36,2379.91) and (416.35,2381.58) .. (418,2383.26) .. controls (419.65,2384.94) and (419.64,2386.61) .. (417.96,2388.26) .. controls (416.28,2389.91) and (416.27,2391.58) .. (417.92,2393.26) .. controls (419.57,2394.94) and (419.56,2396.61) .. (417.88,2398.26) .. controls (416.2,2399.91) and (416.19,2401.58) .. (417.84,2403.26) .. controls (419.49,2404.93) and (419.48,2406.6) .. (417.81,2408.26) .. controls (416.13,2409.91) and (416.12,2411.58) .. (417.77,2413.26) .. controls (419.42,2414.94) and (419.41,2416.61) .. (417.73,2418.26) .. controls (416.05,2419.91) and (416.04,2421.58) .. (417.69,2423.26) .. controls (419.34,2424.94) and (419.33,2426.61) .. (417.65,2428.26) .. controls (415.97,2429.91) and (415.96,2431.58) .. (417.61,2433.26) .. controls (419.26,2434.94) and (419.25,2436.61) .. (417.57,2438.26) .. controls (415.89,2439.91) and (415.88,2441.58) .. (417.53,2443.26) -- (417.5,2447.01) -- (417.5,2447.01)(415.16,2363.24) .. controls (416.81,2364.92) and (416.8,2366.59) .. (415.12,2368.24) .. controls (413.44,2369.89) and (413.43,2371.56) .. (415.08,2373.24) .. controls (416.73,2374.92) and (416.72,2376.59) .. (415.04,2378.24) .. controls (413.36,2379.89) and (413.35,2381.56) .. (415,2383.24) .. controls (416.65,2384.92) and (416.64,2386.59) .. (414.96,2388.24) .. controls (413.28,2389.89) and (413.27,2391.56) .. (414.92,2393.24) .. controls (416.57,2394.92) and (416.56,2396.59) .. (414.88,2398.24) .. controls (413.2,2399.89) and (413.19,2401.55) .. (414.84,2403.23) .. controls (416.49,2404.9) and (416.48,2406.57) .. (414.81,2408.23) .. controls (413.13,2409.88) and (413.12,2411.55) .. (414.77,2413.23) .. controls (416.42,2414.91) and (416.41,2416.58) .. (414.73,2418.23) .. controls (413.05,2419.88) and (413.04,2421.55) .. (414.69,2423.23) .. controls (416.34,2424.91) and (416.33,2426.58) .. (414.65,2428.23) .. controls (412.97,2429.88) and (412.96,2431.55) .. (414.61,2433.23) .. controls (416.26,2434.91) and (416.25,2436.58) .. (414.57,2438.23) .. controls (412.89,2439.88) and (412.88,2441.55) .. (414.53,2443.23) -- (414.5,2446.99) -- (414.5,2446.99) ;
\draw [color={rgb, 255:red, 74; green, 13; blue, 182 }  ,draw opacity=1 ]   (442.19,2336.43) .. controls (443.9,2334.8) and (445.57,2334.85) .. (447.19,2336.56) .. controls (448.82,2338.27) and (450.48,2338.31) .. (452.19,2336.68) .. controls (453.9,2335.06) and (455.56,2335.1) .. (457.19,2336.81) .. controls (458.81,2338.52) and (460.47,2338.56) .. (462.18,2336.94) .. controls (463.89,2335.32) and (465.55,2335.36) .. (467.18,2337.07) .. controls (468.81,2338.78) and (470.47,2338.82) .. (472.18,2337.2) .. controls (473.89,2335.58) and (475.55,2335.62) .. (477.18,2337.33) .. controls (478.81,2339.04) and (480.47,2339.08) .. (482.18,2337.46) .. controls (483.89,2335.84) and (485.55,2335.88) .. (487.18,2337.59) .. controls (488.8,2339.3) and (490.46,2339.34) .. (492.17,2337.72) .. controls (493.88,2336.09) and (495.54,2336.13) .. (497.17,2337.84) .. controls (498.8,2339.55) and (500.46,2339.59) .. (502.17,2337.97) -- (505.54,2338.06) -- (505.54,2338.06)(442.11,2339.43) .. controls (443.82,2337.8) and (445.48,2337.84) .. (447.11,2339.55) .. controls (448.74,2341.26) and (450.4,2341.3) .. (452.11,2339.68) .. controls (453.82,2338.06) and (455.48,2338.1) .. (457.11,2339.81) .. controls (458.74,2341.52) and (460.4,2341.56) .. (462.11,2339.94) .. controls (463.82,2338.32) and (465.48,2338.36) .. (467.11,2340.07) .. controls (468.73,2341.78) and (470.39,2341.82) .. (472.1,2340.2) .. controls (473.81,2338.58) and (475.47,2338.62) .. (477.1,2340.33) .. controls (478.73,2342.04) and (480.39,2342.08) .. (482.1,2340.46) .. controls (483.81,2338.84) and (485.47,2338.88) .. (487.1,2340.59) .. controls (488.73,2342.3) and (490.39,2342.34) .. (492.1,2340.71) .. controls (493.81,2339.09) and (495.47,2339.13) .. (497.1,2340.84) .. controls (498.72,2342.55) and (500.38,2342.59) .. (502.09,2340.97) -- (505.46,2341.06) -- (505.46,2341.06) ;
\draw  [color={rgb, 255:red, 208; green, 2; blue, 27 }  ,draw opacity=1 ][line width=0.75] [line join = round][line cap = round] (398.5,2284) .. controls (403.38,2288.27) and (411.01,2296.38) .. (417.5,2298) .. controls (423.71,2299.55) and (438.59,2302.62) .. (439.5,2309) .. controls (441.64,2323.96) and (413.27,2329.45) .. (408.5,2339) .. controls (404.98,2346.04) and (401.07,2357.97) .. (393.5,2361) .. controls (389.47,2362.61) and (386.33,2357.83) .. (384.5,2356) .. controls (383.57,2355.07) and (367.88,2343.13) .. (367.5,2343) .. controls (362.77,2341.42) and (354.63,2340.36) .. (349.5,2341) .. controls (346.97,2341.32) and (346.77,2345) .. (343.5,2345) ;
\draw  [color={rgb, 255:red, 208; green, 2; blue, 27 }  ,draw opacity=1 ][line width=0.75] [line join = round][line cap = round] (464.5,2272) .. controls (464.5,2275.06) and (467.17,2277.04) .. (467.5,2280) .. controls (468.23,2286.56) and (465.36,2295.14) .. (460.5,2300) .. controls (454.9,2305.6) and (439.96,2301.91) .. (445.5,2313) .. controls (448.54,2319.08) and (459.95,2314.23) .. (465.5,2317) .. controls (471.73,2320.11) and (467.72,2332.8) .. (478.5,2331) .. controls (480.36,2330.69) and (485.46,2318.04) .. (488.5,2315) .. controls (489.73,2313.77) and (495.4,2313) .. (494.5,2313) ;
\draw  [color={rgb, 255:red, 208; green, 2; blue, 27 }  ,draw opacity=1 ][line width=0.75] [line join = round][line cap = round] (453.5,2421.5) .. controls (451.17,2421.5) and (441.19,2408.19) .. (438.5,2405.5) .. controls (436.97,2403.97) and (428.45,2398.98) .. (427.5,2398.5) .. controls (426.77,2398.14) and (419.75,2396.51) .. (419.5,2395.5) .. controls (418.82,2392.78) and (420.92,2388.08) .. (422.5,2386.5) .. controls (431.01,2377.99) and (466.57,2372.72) .. (470.5,2384.5) .. controls (472.71,2391.14) and (467.62,2400.32) .. (458.5,2398.5) .. controls (451.96,2397.19) and (453.07,2382.37) .. (455.5,2377.5) .. controls (459.18,2370.14) and (462.47,2358.53) .. (468.5,2352.5) .. controls (471.66,2349.34) and (471.22,2342.93) .. (475.5,2341.5) .. controls (477.9,2340.7) and (481.32,2343.57) .. (481.5,2346.5) .. controls (482.1,2356.43) and (475.32,2382.5) .. (490.5,2382.5) ;
\draw  [color={rgb, 255:red, 208; green, 2; blue, 27 }  ,draw opacity=1 ][line width=0.75] [line join = round][line cap = round] (360.5,2388) .. controls (371.33,2388) and (377.05,2376.3) .. (383.5,2372) .. controls (385.7,2370.53) and (389.84,2364.23) .. (392.5,2366) .. controls (393.3,2366.53) and (401.46,2384.92) .. (402.5,2387) .. controls (403.94,2389.89) and (416.84,2394.49) .. (414.5,2398) .. controls (411.96,2401.81) and (402.17,2405.33) .. (398.5,2409) .. controls (392.8,2414.7) and (395.5,2424.29) .. (395.5,2431) ;

\draw (74.18,2308) node [anchor=north west][inner sep=0.75pt]    {$\mathcal{X}$};
\draw (76.43,2489) node [anchor=north west][inner sep=0.75pt]    {$B$};

\end{tikzpicture}
\caption{A depiction of a degeneration with simple normal crossings special fiber, and the snc pair arising in the special fiber.}\label{fig: degeneration-picture}
\end{figure}
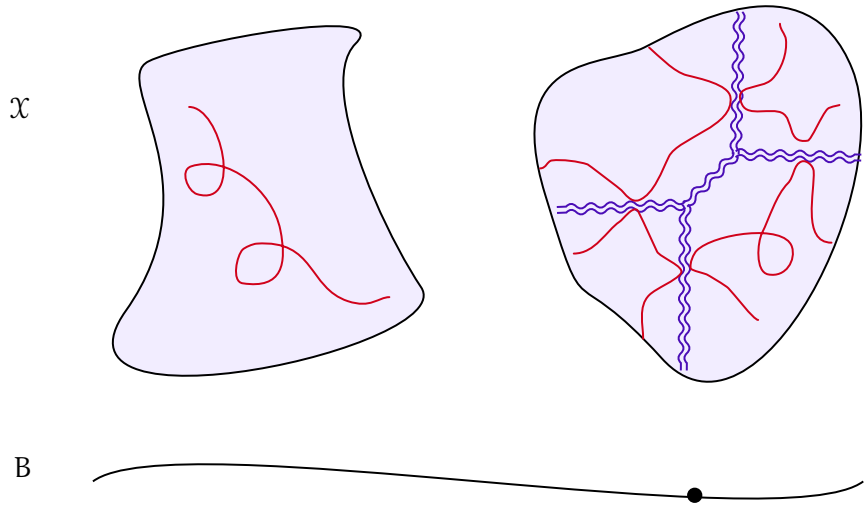

However, the degeneration formula by itself is quite far from a reconstruction result that would compute the Gromov--Witten theory of a variety from that of the strata of a degeneration, for example as in Theorem~\ref{thm: tautological-gw}.   Several problems need to be overcome in order to produce a useful algorithm.

The main issues are the following:
\begin{enumerate}[(i)]
\item {\bf Logarithmic from absolute.} The cycles in the degeneration formula involve the {\it logarithmic} Gromov-Witten theory of projective bundles over the strata of the central fiber, and these need to be related to {\it absolute} GW cycles.

\item {\bf Exotic insertions.} The degeneration formula of \cite{MR23} requires exotic insertions along the boundary -- these are non-local incidence conditions, coupled across multiple boundary strata. They arise from the logarithmic cohomology of products of strata in $\cX_0$, and logarithmic cohomology almost never admits a K\"unneth decomposition. The exotic insertions therefore need to be converted to more geometric, non-exotic insertions.

\item {\bf Vanishing cohomology.} The degeneration formula works with cohomology insertions that live on the total space of the degeneration; a formalism is needed to treat classes that don't extend to the singular fiber, e.g., primitive cohomology classes.

\item {\bf Projective bundles.} We require several results about projectivizations of split vector bundles, and more generally, ``broken'' toric bundles. These arise in two ways -- in deformation to the normal cone analysis for the logarithmic/absolute conversion, and in the final step converting logarithmic GW classes of expansions into classes of the strata of $\cX_0$.
\end{enumerate}

The bulk of this paper is devoted to handling each of these issues.  In total, they convert the logarithmic degeneration formula into an effective tool for proving theorems in Gromov--Witten theory inductively, via snc degenerations. 

A similar package, excluding (iii), was established for double point degenerations in~\cite{MP06}. There are several obstacles that prevent these arguments from being generalized to our setting. The key techniques there are torus localization and the double point degeneration formula~\cite{GV05,Li02}.  In the full logarithmic context, there is no appropriate replacement for localization and the combinatorial calculus governing the degeneration formula is much more complicated. The need for exotic insertions is also invisible in the double point setting. The proofs of most of our results are therefore of a different nature, and include a workaround for localization and significant input from tropical geometry and logarithmic intersection theory~\cite{HMPPS,MR21,PRSS}. 

For the remainder of this introduction, we sketch how we handle each of the items outlined above.


\subsection{Reconstruction I: logarithmic/absolute}

Let $(Y|\partial Y)$ be a simple normal crossings pair with divisor components labelled $E_i$. The space $\Mbar_\Lambda(Y|\partial Y)$ parameterizes logarithmic stable maps from $n$ pointed, genus $g$:
\[
C\to Y,
\]
with fixed curve class and prescribed tangency orders $c_{ij}$ at the marked point $p_j$ along the divisor $E_i$. The space is proper and carries a virtual class. Its evaluation space $\mathsf{Ev}$ is a product of strata of $(Y|\partial Y)$ -- the stratum associated with a marked point is the smallest stratum that it maps to, given its contact order. 

Logarithmic GW classes are obtained by the pull/push operation above, with insertions allowed to come from the {\it logarithmic} cohomology of the evaluation space. We say more shortly. 

\begin{customthm}{E}[Section~\ref{sec: log-absolute}]\label{thm: log-absolute}
Let $(Y|\partial Y)$ be a simple normal crossings pair. The logarithmic Gromov--Witten cycles of $(Y|\partial Y)$ can be uniquely and effectively reconstructed from the absolute GW cycles of the strata of $(Y|\partial Y)$. 
\end{customthm}

The geometric content is that logarithmic GW theory does not produce new invariants -- tangency conditions in logarithmic geometry are well-approximated by absolute descendant invariants. 

The key to the proof is understanding how to remove a divisor from the boundary $\partial Y$. When the degeneration formula is applied to the degeneration to the normal cone of a divisor, it expresses the invariants of a pair via a system of equations involving the special fiber of a degeneration, one piece of which {\it adds} $D$ to the boundary~\cite{MR23}. The theorem is a consequence of an invertibility statement for this system of equations, obtained by an analysis of tropical curves. 

The theorem itself has been a folklore conjecture in logarithmic Gromov--Witten theory for many years; it is alluded to in~\cite{AW,R19}. It simultaneously generalizes parallel statements in the smooth pair case~\cite[Theorem~2]{MP06} and the genus $0$ case~\cite[Corollary~Z]{BNR22}.

The theorem also confirms a conjecture of Urundolil Kumaran and the second author, that the logarithmic GW cycles of smooth toric varieties with any toric logarithmic structure lie in the tautological ring~\cite[Conjecture~C]{RUK22}. It also follows that logarithmic GW cycles of products of projective spaces, with {\it any} logarithmic structure, lie in the tautological ring, confirming a number of cases of~\cite[Conjecture~D]{RUK22}.

\subsection{Reconstruction II: exotic/non-exotic}

In earlier work, we introduced ``exotic'' evaluation conditions in logarithmic cohomology~\cite{MR23}. They are needed to apply logarithmic degeneration formulas, but are also used here to control primitive cohomology, see Section~\ref{sec: primitive} for further details. 

The usual evaluation map is given by
\[
\Mbar_\Lambda(Y|\partial Y)\to \mathsf{Ev}.
\]
The space $\mathsf{Ev}$ is recalled in the next section, but for now we note that it is a smooth space with a simple normal crossings boundary. The space comes with an expression as a product of varieties -- one for each marked point. In addition to standard cohomology classes, in the formalism of~\cite{MR23}, we can pull back any cohomology class from a birational model obtained by blowups along strata. An {\it exotic insertion} is a cohomology class on a blowup that is not pulled back from $\mathsf{Ev}$.  Heuristically, these encode tangency conditions that are coupled between different strata of the boundary divisor.

While they are more complicated than standard insertions, these exotic insertions are crucial to formulate the degeneration formula -- they are needed to express a logarithmic GW invariant of a reducible normal crossings variety in terms of those of its strata.  In order to handle this new complexity, we prove the following reconstruction result.

\begin{customthm}{F}[Sections~\ref{sec: exotic-evaluations} {\it \&}~\ref{sec: rigidification}]\label{thm: exotic-non-exotic}
Let $(Y|\partial Y)$ be a simple normal crossings pair. Its Gromov--Witten cycles with exotic insertions can be uniquely and effectively reconstructed from the Gromov--Witten cycles with non-exotic insertions.  
\end{customthm}

An important piece of the proof is {\it rigidification}. The degeneration formula does not, a priori, split strata in the logarithmic mapping space into pieces, due to the presence of non-rigid targets. This comes from the isotropy groups on the stack of logarithmic structures. This is more of an obstacle in the logarithmic context compared to smooth pairs, as the isotropy groups act non-locally on the expansion~\cite{CN21}. 

We introduce a procedure that takes a GW cycle on a stratum in $\Mbar_\Lambda(Y|\partial Y)$ and expresses it as a GW class on a component of the space of maps to a degeneration. This procedure should be of use in other contexts. 

We then use this rigidification and splitting for the strata of $\Mbar_\Lambda(Y|\partial Y)$ to prescribe an invertible transformation rule that trades exotic GW classes for non-exotic ones. 

\subsection{Reconstruction III: vanishing cohomology}

As it happens, exotic insertions also provide the solution to the issue of vanishing cohomology insertions.  Given a degeneration
$$\mathcal{X} \rightarrow B,$$
a cohomology class on a general fiber $\cX_b$ may not be the restriction of a class on 
$\mathcal{X}$.  The obstruction arises from the monodromy action on the cohomology of $\cX_b$.  However, we show the following:


\begin{customthm}{G}[Section~\ref{sec: primitive}]\label{thm: primitive-move}
The GW cycles on $\cX_b$ with arbitrary insertions lie in the span of the
\textit{exotic} GW classes of a special fiber $\cX_0$
\end{customthm}

The result is related to the work~\cite{ABPZ}, which produces an algorithmic reconstruction result for complete intersections in projective space. Our result is a ``pure thought'' argument and does not use an explicit algorithm; this is both a feature and bug -- it applies in complete generality, but further analysis is necessary to extract formulas.

We highlight a key consequence. In combination with Theorem~\ref{thm: exotic-non-exotic}, this allows the logarithmic degeneration formula to be {\it iterated} in cohomology -- once a variety is degenerated, its pieces can be degenerated further and then analyzed independently.

\subsection{Reconstruction IV: bundles}

The final key step in our arguments is the treatment of projective bundles. Given a pair $(Y,\partial Y)$, as above, and a line bundle $L$ on $Y$, we can form the associated projective completion $p: \mathbb P_Y\to Y$. The space $\mathbb P_Y$ can be equipped with a logarithmic structure by using four different divisors:
\[
p^{-1}(\partial Y), \ \ \ p^{-1}(\partial Y)\cup Y_0, \ \ \ p^{-1}(\partial Y)\cup Y_\infty, \ \ \ p^{-1}(\partial Y)\cup Y_0\cup Y_\infty.
\]
This construction gives rise to four distinct GW theories associated with the total space $\mathbb{P}_Y$. By further blowups along certain strata, we obtain four corresponding variants of ``broken'' $\mathbb{P}^1$-bundle theories over $Y$.

\begin{customthm}{H}[Sections~\ref{sec: uni-P1-bundle} {\it \&}~\ref{sec: uni-log-bundle}]\label{thm: bundle-reconstruction}
The Gromov--Witten classes of each of the four theories can be uniquely and effectively reconstructed from the Gromov--Witten classes of $(Y|\partial Y)$. 
\end{customthm}

Broken projective and toric bundles are ubiquitous in logarithmic enumerative geometry. In the expanded formalism~\cite{R19} they arise as components in expansions, while in the punctured theory~\cite{ACGS17} they arise from maps to targets with generically nontrivial logarithmic structure. The theorem allows us to analyze the deformation to the normal cone of a divisor, which is key to Theorem~\ref{thm: log-absolute}.

When $\partial Y$ is empty, we recover~\cite[Theorem~1]{MP06}, but a different proof is required in the general case. The key ingredient there is the localization formula, which, as we mentioned, is not available in the logarithmic setting. Our approach is to prove a version of the result over the universal Picard stack instead of $\Mbar_\Lambda(Y|\partial Y)$ and then promote it to logarithmic cohomology to deduce the statement. In particular, no new logarithmic localization formula is required. This is a point of contact for our work with recent work on higher double ramification cycles~\cite{BHPSS,HMPPS,HS21,MR21,RUK22}.

\subsection{Further questions} Let us return to the cohomology speculation in light of the results. By Theorem~\ref{thm: tautological-gw}, if a smooth variety $X$ can be degenerated to a union of varieties whose strata have tautological GW theory, then $X$ itself has tautological GW theory. Since the degeneration formula can be iterated, we can degenerate, take apart the special fiber, degenerate again, and so on. Even beyond the ones we state here, the strategy should provide a wealth of examples. A fundamental open question is how close this is to all varieties. This question seems related to algebraic cobordism, and motivated the conjecture of~\cite{LP09}. 

The cohomology speculation also has a ``birational'' nature. Let $X$ be a smooth projective variety and $Z$ a smooth subvariety. By the degeneration formula and the results of~\cite{Fan21}, the GW cycles of $X$ can be reconstructed from the GW theory of $\mathsf{Bl}_Z X$ and those of $Z$. It is also known that the GW cycles of $\mathsf{Bl}_Z X$ can be reconstructed from those of $X$ and of $Z$ by~\cite{HLR08,Fan21}.

\noindent
{\bf Observation.} {\it If $X$ and $X'$ are birational surfaces or threefolds, then $X$ has tautological GW theory if and only if $X'$ does. }

Independent of our results, this means that every rational threefold has tautological GW theory. The case of rational fourfolds seems completely open. It leads to the following:

\noindent
{\bf Question.} {\it Is every surface (resp. threefold) birational to one that admits an snc degeneration to a union of rational surfaces (resp. to threefolds with rational double locus)?}

In a different direction, Theorem~\ref{thm: tautological-gw} allows us to confirm the cohomology speculation for complete intersections with breakable Newton polytope, but more generally one can ask:

\noindent
{\bf Question.} {\it Let $X$ be a smooth complete intersection in a toric variety. Does $X$ have tautological GW theory?}

The case of a hypersurface in a weighted projective space -- viewed as a smooth orbifold -- is open. A {\it fake weighted projective space} is a toric variety defined by a lattice simplex with no lattice points except the vertices. Every hypersurface in a toric variety admits a degeneration to a union of hypersurfaces with such Newton polygons.\footnote{Hypersurfaces of this form are not always rational. The 1983 Warwick PhD thesis of Fine constructs a surface of this form with non-negative Kodaira dimension, see~\cite[Appendix to \S 4]{Reid87}.} If their GW cycles are tautological, then all hypersurfaces in toric varieties have tautological GW theory.

For each new class of varieties that can be degenerated effectively, the package of Theorem~\ref{thm: tautological-gw} establishes new verifications of the speculation. Toric complete intersections offer first examples, but that is unlikely to be the limit of the techniques.

\noindent
{\bf Question.} {\it Do subvarieties of Grassmannians (or other homogeneous spaces) given by sections of tautological bundles admit adequate snc degenerations?}

Gromov--Witten theory of $K3$ and abelian surfaces (and other holomorphic symplectic manifolds) offers a different direction. The ordinary Gromov--Witten theory of a $K3$ surface vanishes, but the reduced virtual class provides nontrivial classes in the cohomology of $\Mbar_{g,n}$.

\noindent
{\bf Question.} {\it Do the reduced GW classes of K3 and abelian surfaces lie in the tautological ring?}

These surfaces admit good degenerations to unions of rational surfaces, but the resulting degeneration of the reduced virtual class is not fully understood. For primitive classes on K3 surfaces, one can use~\cite[Theorem~20]{MPT10}. In the abelian setting, the question should be closely related the recent theory of correlated GW classes due to Blomme--Carocci~\cite{BC25b,BC25}. For related results in Chow, see~\cite[Theorem~1.5]{SP20}. 

The reconstruction procedures above can be carried out algorithmically, with one exception -- the analysis of vanishing cohomology. A constructive procedure requires an analysis of the monodromy action in particular types of degenerations. 

\noindent
{\bf Problem.} {\it Produce explicit formulas for the GW invariants of families with good snc degenerations in toric varieties or flag varieties.}

A different direction of inquiry is related to {\it punctured maps}~\cite{ACGS17}. Heuristically, punctured invariants are combinations of standard logarithmic maps and higher double ramification cycles. The latter are converted to invariants of the base by Theorem~\ref{thm: bundle-reconstruction}. We make the following:

\noindent
{\bf Conjecture.} {\it Punctured GW cycles of $(Y|\partial Y)$ are determined by absolute GW cycles of the strata.}

In fact, there are multiple versions of this conjecture, depending on which class one chooses on the puncturing stack. A canonical choice that behaves well under the splitting axiom is the refined punctured class of~\cite{BNR24}.\footnote{Johnston has informed us that his forthcoming preprint establishes results along the lines of the above conjecture for the punctured theory of split toric bundles. }

\subsection{Roadmap of paper} In Section~\ref{sec: loggy-background} we recall the basics of logarithmic GW theory and the degeneration formula, including the exotic evaluation formalism of~\cite{MR23}. We analyze vanishing cohomology in Section~\ref{sec: primitive}, in terms of exotic insertions. In Section~\ref{sec: stars-partial-ordering}, we introduce an ordering on the discrete data for logarithmic moduli problems involving the geometry of tropical curves. In Section~\ref{sec: exotic-evaluations} we introduce {\it strata GW classes}, which mediate between exotic and non-exotic insertions. In Section~\ref{sec: rigidification} we explain how to convert strata GW classes into non-exotic classes, completing the exotic/non-exotic trade. The next two sections deal with projective bundles. In Section~\ref{sec: uni-P1-bundle} we construct GW classes over the Picard stack using the universal $\mathbb P^1$-bundle. We include a discussion of how such a theory can be used to circumvent logarithmic localization. In Section~\ref{sec: uni-log-bundle} we complete this picture and deduce that $\mathbb P^1$-bundles, and broken toric bundles, can be constructed from GW cycles of the base. The results are recorded in Section~\ref{sec: broken-bundles}. In Section~\ref{sec: log-absolute}, we use this, with an invertibility argument, to prove Theorem~\ref{thm: log-absolute}. In Section~\ref{sec: final-proofs} we construct the necessary degenerations to establish Theorems~\ref{thm: tautological-chow} and~\ref{thm: tautological-gw}, as well as Corollary~\ref{cor: breakable}. 

\subsection{Notation and conventions} We work over the complex numbers. Stacks are assumed to be locally of finite type. For logarithmic stacks, we work in the fine and saturated logarithmic category, and logarithmic structures on our spaces will often be specified by distinguished maps to certain Artin fans. Our running notation is $\cX\to B$ for a degeneration and $(Y|\partial Y)$ for a simple normal crossings (snc) pair. Chow and cohomology will be taken with rational coefficients unless explicitly stated otherwise. 

\subsection*{Acknowledgments} We have benefited from many conversations with friends and colleagues over the years about tautological rings, degeneration formulas, and related topics. We especially thank L. Battistella, F. Carocci, S. Canning, R. Cavalieri, S. Johnston, P. Kennedy-Hunt, S. Koyama, M. Gross, S. Molcho, N. Nabijou, I. Smith, T. Song, Q. Shafi, A. Urundolil Kumaran, and J. Wise. We are grateful to R. Pandharipande for encouraging us to pursue the questions in Chow and for helpful comments on an earlier draft. The work was first presented at the conference ``Curves, abelian varieties, and their moduli'' at Humboldt-Universit\"t zu Berlin in May 2024, organized by G. Farkas and R. Pandharipande. We thank them for this opportunity.

\subsection*{Funding} D.M. was supported by a Simons Investigator grant. D.R. was supported by EPSRC Horizon Europe Guarantee EP/Y037162/1 in lieu of an ERC starting grant and by EPSRC New Investigator Award EP/V051830/1. 

\section{A review of the logarithmic degeneration package}\label{sec: loggy-background}

The degeneration package in logarithmic geometry has undergone several years of development. The foundations were laid in~\cite{AC11,Che10,GS13}, the approach via expanded degenerations is established in~\cite{MR20,R19}, and the degeneration formula in the form we will need is proved in~\cite{MR23}. 

\subsection{Stable maps to pairs}\label{sec: stable-maps-to-pairs} Let $(Y|\partial Y)$ be a simple normal crossings pair and let $\Sigma_{Y|\partial Y}$ be its cone complex. We fix the following discrete data for a logarithmic stable maps problem:
\begin{enumerate}[(i)]
    \item a curve class $\beta\in H_2(Y;\ZZ)$,
    \item a positive integer $g$ for the genus,
    \item a number $n$ of marked points, labelled by the set $[n]$, and
    \item for a component $E_j\subset\partial Y$ and point $p_i$ a non-negative integer {\it contact order}, recorded in a matrix $[c_{ij}]$.
\end{enumerate}

\noindent
The discrete data will be collectively denoted by $\Lambda$. 

\begin{definition}[Disjoint contact orders]\label{def: disjoint-contact}
A contact order matrix $[c_{ij}]$ is {\it disjoint} if for each fixed $i$, there is at most one $j$ such that $c_{ij}$ is nonzero. That is, each row has at most one nonzero entry. 
\end{definition}

Given a logarithmic blowup $Y'\to Y$, the discrete data for $Y$ has a canonical lift to discrete data $\Lambda'$ for maps to $Y'$. Indeed, each vector $(c_{ij})_j$ for $i$ fixed determines a vector in $\Sigma_{Y|\partial Y}$ -- add all the ray generators weighted by the contact order they have with $p_i$. The map 
\[
\Sigma_{Y'|\partial Y'}\to \Sigma_{Y|\partial Y}
\]
is a bijection on lattice points so this lifts canonically. Reversing the process gives the new contact order matrix. 

The contact order matrix determines intersection numbers for a curve on $Y'$ with each component of $\partial Y'$. By Poincar\'e duality, there is a unique curve class on $Y'$ that pushes forward to $\beta$ and has these fixed intersection numbers. 

We refer to this as the {\it canonical lift} of the discrete data under a logarithmic blowup. 

\begin{lemma}
Let $\Lambda$ be a discrete data set for $Y$. There exists a logarithmic blowup $Y'\to Y$ such that the canonical lift of $\Lambda$ has disjoint contact orders. If the contact order matrix of $\Lambda$ is disjoint, then that of its canonical lift to any logarithmic blowup is disjoint. 
\end{lemma}

\begin{proof}
Disjointness is exactly the condition that the lattice point associated to each marked point lies on a ray. This makes both claims obvious.
\end{proof}

We will mostly work in the disjoint case, but there are some places where the more general setup is convenient. 

Fix a map of pairs
\[
(Y|\partial Y)\to (X|\partial X)
\]
and discrete data $\Lambda$ on $(Y|\partial Y)$. There are ``pushforward'' discrete data on $X$ -- push forward $\beta$ and keep $g$ and $n$ the same; for the contact order, express the tangency data in terms of lattice points on the cone complex $\Sigma_{Y|\partial Y}$ and use the induced map of cone complexes.

We adopt a notational convention -- if we have fixed discrete data $\Lambda$ and a morphism $Y\to X$, we will not introduce new symbols for the pushforward data, and use $\Lambda$ for the induced data as well. This should not cause confusion, where it is used. 

Fix discrete data $\Lambda$. The foundational papers in logarithmic Gromov--Witten theory deal with a moduli problem on logarithmic schemes, parameterizing maps
\[
\begin{tikzcd}
    C\arrow{r}{f}\arrow{d} & (Y|\partial Y)\\
    S,
\end{tikzcd}
\]
where $C/S$ is an $n$-marked logarithmically smooth curve of genus $g$, and $f$ is a logarithmic morphism. The tangency orders $c_{ij}$ are determined by the induced maps on characteristic monoids between the $E_j$ and $p_i$. A logarithmic map is {\it stable} if the underlying map on schemes $C\to Y$ is stable. The papers of Abramovich--Chen and Gross--Siebert establish the following~\cite{AC11,Che10,GS13}.

\begin{theorem}
The moduli problem on the category of logarithmic schemes of logarithmic stable maps with data $\Lambda$ is representable by a Deligne--Mumford stack with logarithmic structure, denoted $\Mbar_\Lambda(Y|\partial Y)$. The forgetful morphism
\[
\Mbar_\Lambda(Y|\partial Y)\to \Mbar_{g,n}(Y,\beta)
\]
is quasi-finite and proper. The space $\Mbar_\Lambda(Y|\partial Y)$ is equipped with a relatively perfect obstruction theory over the stack of logarithmic structures, and carries a virtual fundamental class $[\Mbar_\Lambda(Y|\partial Y)]^{\sf vir}$ defined in Chow homology, of homological degree
\[
{\sf vdim} = (\dim X-3)(1-g)+c_1(T^{\sf log}_{Y|\partial Y})\cdot \beta+n.
\]
\end{theorem}

Given a point $p_i$, let $W_i$ be the intersection of all divisor components $E_j$ in $\partial Y$ such that the contact order $c_{ij}$ is positive. The stratum $W_i$ carries a natural logarithmic structure, given by the intersection of $W_i$ with those $E_j$ that {\it do not} contain it. 

There is a logarithmic evaluation morphism\footnote{Note that some authors prefer to use the logarithmic structure on $W_i$ that is pulled back from $Y$, e.g.~\cite{ACGS17} This forces the logarithmic structure on the moduli space to be enhanced. We prefer not to do this. See~\cite[Section~2.1]{RUK22} for a discussion.}
\[
{\sf ev}_i\colon \Mbar_\Lambda(Y|\partial Y)\to W_i.
\]

A crucial result for us is Abramovich and Wise's logarithmic birational invariance~\cite[Theorem~1.1.1]{AW}.

\begin{theorem}
Let $Y'\to Y$ be a logarithmic blowup. Given discrete data $\Lambda$ for $Y$ and canonical lift $\Lambda'$ for $Y'$, the morphism
\[
\Mbar_{\Lambda'}(Y'|\partial Y')\to \Mbar_\Lambda(Y|\partial Y)
\]
induced by composition and stabilization identifies virtual classes under proper pushforward.  
\end{theorem}

In light of this, we will be cavalier about choice of logarithmic birational models. By birational invariance and the projection formula, we can replace $(Y|\partial Y)$ with a logarithmic blowup without {\it losing} any GW classes. Therefore, there is generally no harm in assuming that the contact order matrix $[c_{ij}]$ has the disjointness property of Definition~\ref{def: disjoint-contact}, although, as we mentioned, there will be moments when we consider the general case.  We flag these instances in the appropriate context.

There is a proper forget-and-stabilize morphism
\[
\Mbar_\Lambda(Y|\partial Y)\to \Mbar_{g,n}
\]
which we will use to construct logarithmic GW cycles.

\subsection{Stable maps to degenerations} Let $\cX$ non-singular quasi-projective variety and $B$ a non-singular connected curve. Consider a {\it simple normal crossings degeneration} -- a flat and proper morphism
\[
\cX\to B
\]
such that the fiber over every critical value is a simple normal crossings divisor in $\cX$. If we equip $\cX$ with the divisorial logarithmic structure associated to these singular fibers, and equip $B$ with the logarithmic structure coming from the critical values, the map $\cX\to B$ is logarithmically smooth. It will be useful to allow the additional generality that $\cX$ is equipped with divisors that dominate $B$ and, together with all singular fibers, have simple normal crossings. In other words, we allow the generic fiber to have nontrivial logarithmic structure. 

Fix a curve class $\beta$, supported on fibers\footnote{When monodromy acts nontrivially on the cohomology of a general fiber, we can view $\beta$ as an element of monodromy coinvariants of $H_2(X_\eta;\mathbb Z)$. Since the Hodge class $\beta$ has finite monodromy orbit, we can alternatively pass to a finite cover of $B$. In practice, we fix a finite set of divisors on $X_\eta$, and record the intersection numbers with these divisors. As long as the set of divisors includes an ample, we will obtain sensible formulas.}. We can consider the family version of the space of logarithmic stable maps, over the base $B$. Its objects are commutative diagrams of logarithmic schemes
\[
\begin{tikzcd}
    C\arrow{r}\arrow{d} & \mathcal X\arrow{d}\\
    S\arrow{r} & B,
\end{tikzcd}
\]
where $C/S$ is a family of logarithmically smooth curves as before, and the underlying map is a family of stable maps in the usual sense. 

The foundations, in this situation relative to $B$, are also established in~\cite{AC11,Che10,GS13}. The moduli problem described above is representable by a Deligne--Mumford stack, is proper over $B$, and is equipped with a logarithmic structure. The morphism 
\[
\Mbar_{g,n}(\mathcal X/B,\beta)\to B
\]
carries a perfect obstruction theory relative to the stack of logarithmic structures over $B$. This leads to a virtual class
\[
\left[\Mbar_{g,n}(\mathcal X/B,\beta)\right]^{\sf vir}\in H_\star(\Mbar_{g,n}(\mathcal X/B,\beta);\QQ).
\]

Let $\cD_1,\ldots,\cD_r$ be the divisors in $\cX$ that are horizontal with respect to $\cX\to B$ and for each $I\subset [r]$ we let $\cD_I$ denote the stratum where $\cD_i$ intersect for $i\in I$. For each marked point $p_i$, the logarithmic tangency order $c_{ij}$ of a stable map
\[
C\to\cX
\]
with $\cD_j$ at the point $p_i$ is well-defined and locally constant in flat families. If we let $\Lambda$ denote the discrete data $(g,n,\beta,[c_{ij}])$ then the moduli spaces further decompose as
\[
\Mbar_\Lambda(\cX/B)\to B.
\]
Each of these components is open in the larger moduli, and so also has a virtual class. 

The fibers of $\cX\to B$ are logarithmically smooth, and by the foundational theory~\cite{GS13}, the logarithmic mapping spaces to the fibers are also equipped with virtual classes. The following result captures the compatibility between fibers and the family:

\begin{theorem}
If $\iota_b:\{b\}\hookrightarrow B$ then:
\[
[\Mbar_\Lambda(\mathcal X_b)]^{\sf vir} = \iota_b^! [\Mbar_\Lambda(\mathcal X/B)]^{\sf vir} \ \ \textnormal{in } \mathsf{H}_\star(\Mbar_\Lambda(\cX_b);\mathbb Q).
\]
\end{theorem}

If the curve $B$ is a rational curve, then the statement above holds in the Chow group of $\Mbar_\Lambda(\cX_b)$ with rational coefficients. 

The target of the natural evaluation map from this space is the fibered power:
\[
\Mbar_{g,n}(\mathcal X/B,\beta)\to \mathcal X^n_B.
\]
This can be refined slightly in the presence of horizontal divisors. If $p_i$ is a marked point, the indices $j$ for which $c_{ij}$ is positive distinguish a subset of the divisors, and so by intersection, a stratum of the form $\cD_I$, which we will denote $\mathcal W_i$. If $c_{ij}$ is $0$ for all $j$ and a fixed $i$, we define $\mathcal W_i = \mathcal X$. Consolidating, we have an evaluation morphism
\[
\Mbar_{g,n}(\mathcal X/B,\beta)\to \mathcal{E}v_\Lambda :=\mathcal W_1\times_B\cdots\times_B \mathcal W_n.
\]
Note that this {\it evaluation space} $\mathcal E v_\Lambda$ is usually a singular variety. 

The deformation invariance of the virtual class formally leads to the following.

\begin{corollary}
Fix a cohomology class $\gamma\in H^\star(\mathcal E v_\Lambda)$ and denote the restriction to the fibers by $\gamma_b$. The Gromov--Witten class in $H^\star(\Mbar_{g,n})$ associated to $\gamma_b$ is independent of $b$. 
\end{corollary}


If $\cX_b$ is a general fiber of $\cX\to B$, with induced divisors $D_b$, then the corollary moves the calculation of a logarithmic Gromov--Witten class associated to $(\cX_b,D_b)$ onto the singular fibers, with the caveat that, in the setup above, the cohomology classes must extend from $X$ to the total space of this singular degeneration. 

\begin{remark}
If we fix some $b\in B$, the restriction map on cohomology of this evaluation space, for example, the map
\[
H^\star(\mathcal X_B^n)\to H^\star(\cX^n_b)
\]
is typically not surjective. Any cohomology class that extends to the total space is necessarily invariant under monodromy, but since the total space is singular, even algebraic classes on a general fiber may not extend to cohomology classes on the total space. The corollary above does not address this issue, but we will deal with it in the next section by classical arguments --resolution and monodromy invariance. 
\end{remark}

\subsection{Stable maps to the singular fiber}\label{sec: decomposition} Fix an snc degeneration $\cX\to B$, possibly with horizontal divisors. Both $\cX$ and $B$ have cone complexes. For $B$ it is a union of rays, attached to each other at their vertices. The number of rays is equal to the number of singular fibers of $\cX\to B$. Since $\cX\to B$ is a map of pairs, there is an induced map of cone complex $\Sigma_{\cX}\to\Sigma_B$. The associated map on topological spaces is a proper map if and only if there are no horizontal divisors. 

Let $b$ be a point of $B$. It has a dual polyhedral complex $\Delta(X_b)$ -- if $b$ is a regular value, it is the fiber of $\Sigma_{\cX}\to\Sigma_B$ over the vertex, if $b$ is a critical value, it is the fiber over $1$ in the cone corresponding to $b$. The vertices of $\Delta(X_b)$ are indexed by a set, say $[N]$, of irreducible components over $b$. We compress the notation to $\Delta$ when it is clear from context. 

The compact simplices of $\Delta$ are indexed by subsets $I\subset [N]$. Each subset corresponds to a stratum of $X_0$ whose codimension is the dimension of the corresponding cell of $\Delta$. Denote the strata by $Y_I$. Each cell spans an integral affine subspace of the space $\mathbb R^{\Delta[1]}$ where $\Delta[1]$ is the set of vertices. An unbounded cell in $\Delta$ has some unbounded rays attached to a bounded simplex -- a cell of this type is dual to the intersection of the stratum corresponding to the simplex with the horizontal stratum corresponding to the unbounded rays. 

Each stratum $Y_I$ has split normal bundle in $\cX$, since $Y_I$ is the complete intersection of $Y_i$ over $i\in I$. This will be relevant when we analyze components in fibers of logarithmic blowups and ramified base changes of $\cX\to B$.

Let $\cX_b$ be a fiber of $\cX\to B$. We will mostly be interested in the singular fibers. The decomposition theorem of~\cite{ACGS15} asserts that the space $\Mbar_{g,n}(\cX_b,\beta)$ has a decomposition into ``virtual'' irreducible components:
\[
\Mbar_{g,n}(\cX_b,\beta) = \bigcup_\gamma \Mbar_{\gamma}(\cX_b,\beta).
\]
The indexing set for this decomposition is the set of {\it rigid tropical curves} in $\Delta(\cX_b)$. We refer to~\cite{ACGS15} for details, but for our purposes, it will be sufficient to note that a rigid tropical curve $\gamma$ determines the following data:
\begin{enumerate}[(i)]
\item a $(g,n)$ prestable dual graph ${{\sf G}(\gamma)}$, including a genus decoration $g_V$ and a subset $[n]_V$ of legs based at $V$, 
\item a curve class decoration $\beta_V$ on the associated stratum for each vertex $V$ of ${{\sf G}(\gamma)}$ adding up, after push forward, to the total class $\beta$,
\item a cell attached to every vertex/edge of ${{\sf G}(\gamma)}$,
\item for each oriented edge of ${{\sf G}(\gamma)}$, an edge direction in the associated cell.
\end{enumerate}

{\it Rigidity} is the property that these data cannot be deformed without changing the combinatorial type, see Figure~\ref{fig: rigid-tropical-curve}.

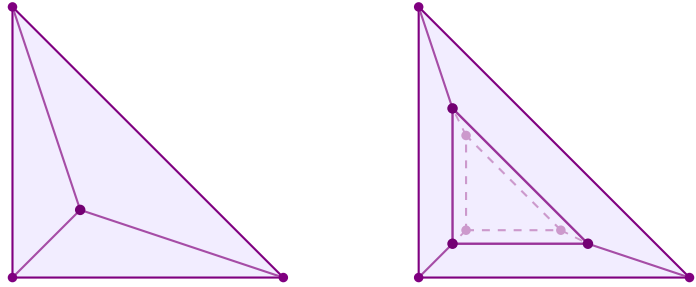
\begin{figure}[h!]
\begin{tikzpicture}[scale=1.8]

  \definecolor{softlavender}{RGB}{223,210,255}

  \begin{scope}[shift={(-3,0)}]
    \coordinate (A) at (0,0);
    \coordinate (B) at (2,0);
    \coordinate (C) at (0,2);
    \coordinate (P) at (0.5,0.5); 

    \filldraw[fill=softlavender, draw=none,opacity=0.4] (A) -- (B) -- (C) -- cycle;

    \draw[thick, violet] (A) -- (B) -- (C) -- cycle;

    \draw[violet!70, line width=0.8pt] (P) -- (A);
    \draw[violet!70, line width=0.8pt] (P) -- (B);
    \draw[violet!70, line width=0.8pt] (P) -- (C);

    \foreach \pt in {A,B,C} {
      \fill[color=violet] (\pt) circle (1pt);
    }
    \fill[violet!90!black] (P) circle (1.1pt);
  \end{scope}

  \coordinate (A) at (0,0);
  \coordinate (B) at (2,0);
  \coordinate (C) at (0,2);
  
  \coordinate (PA) at (0.25,0.25); 
  \coordinate (PB) at (1.25,0.25); 
  \coordinate (PC) at (0.25,1.25); 

  \path let 
    \p1 = (PA), \p2 = (PB), \p3 = (PC)
    in 
    coordinate (Centroid) at 
      ($ (0.333*\x1 + 0.333*\x2 + 0.333*\x3, 0.333*\y1 + 0.333*\y2 + 0.333*\y3) $);

  \def\shrink{0.7}

  \coordinate (PAp) at ($ (Centroid) ! \shrink ! (PA) $);
  \coordinate (PBp) at ($ (Centroid) ! \shrink ! (PB) $);
  \coordinate (PCp) at ($ (Centroid) ! \shrink ! (PC) $);

  \filldraw[fill=softlavender, draw=none,opacity=0.4] (A) -- (B) -- (C) -- cycle;

  \draw[thick, violet] (A) -- (B) -- (C) -- cycle;

  \draw[violet!70, line width=0.8pt] (PA) -- (A);
  \draw[violet!70, line width=0.8pt] (PB) -- (B);
  \draw[violet!70, line width=0.8pt] (PC) -- (C);
  \draw[violet!80, line width=0.9pt] (PA) -- (PB) -- (PC) -- cycle;

  \draw[violet!40, dashed, line width=0.8pt] (PAp) -- (PBp) -- (PCp) -- cycle;
  \draw[violet!40, dashed, line width=0.7pt] (PAp) -- (PA);
  \draw[violet!40, dashed, line width=0.7pt] (PBp) -- (PB);
  \draw[violet!40, dashed, line width=0.7pt] (PCp) -- (PC);

  \foreach \pt in {PA,PB,PC} {
    \fill[violet!90!black] (\pt) circle (1.1pt);
  }

  \foreach \pt in {PAp,PBp,PCp} {
    \fill[violet!40] (\pt) circle (1pt);
  }

  \foreach \pt in {A,B,C} {
    \fill[color=violet] (\pt) circle (1pt);
  }

\end{tikzpicture}
\caption{A rigid tropical curve in a dual complex $\Delta$ on the left, given by the three interior edges. A non-rigid curve on the right that arises as a deformation. The dual complex is the triangle and the tropical curve is the union of the edges in the interior.}\label{fig: rigid-tropical-curve}
\end{figure}

These data determine a moduli problem of stable maps to snc pairs associated with every vertex $V$ in $\Gamma$. We need the following notion, which will come up repeatedly. 

\begin{definition}[Broken toric bundle]
Let $(B|\partial B)$ be an snc pair and let $Y_\sim \to B$ be the projectivization of a direct sum of line bundles and take $\partial Y_\sim$ to be the union of the pullback of $\partial B$ with the coordinate subspaces of $Y_\sim$. Let $Y\to Y_\sim$ be any logarithmic blowup of it, with the induced logarithmic structure given by $\partial Y$.

A variety of the form $(Y|\partial Y)\to (B|\partial B)$ described above will be called a {\it broken toric bundle}. 
\end{definition}

Each vertex determines a target $Y_{\gamma(V)}$, which is a broken toric bundle over a stratum of $\cX_0$. The vertex $V$ has an attached cell in $\Delta(\cX_0)$, and therefore a stratum of $\cX_0$. For our purposes, it is sufficient to note that $Y_{\gamma(V)}$ is a broken toric bundle coming from the normal bundle of the stratum. The logarithmic structure on $Y_{\gamma(V)}$ will be given by pull back the residual divisors obtained by intersection of the stratum with the components of $X_0$ {\it and} the invariant divisors for the normal torus torsor.\footnote{ The precise choice is not important for us, provided the contact order for the problem is taken to be disjoint.}

We state the decomposition theorem~\cite[Theorem~1.2]{ACGS15}
\begin{theorem}
There is an equality of virtual classes:
\[
\iota_0^![\Mbar_{g,n}(\mathcal X,\beta)]^{\sf vir} = [\Mbar_{g,n}(X_0,\beta)]^{\sf vir} = \sum_\gamma m_\gamma\cdot [\Mbar_{\gamma}(X_0,\beta)]^{\sf vir}.
\]
The constant $m_\gamma$ is a positive rational number. 
\end{theorem}

The constant $m_\gamma$ is combinatorial but not relevant here.

\subsection{Cutting maps to a reducible target} From this point on, until we say otherwise, we assume {\it disjointness} of the contact order matrix, as in Definition~\ref{def: disjoint-contact}. 

If we fix a rigid tropical curve $\gamma$, the associated moduli space admits a cutting map:
\[
\kappa: \Mbar_{\gamma}(X_0,\beta)\to \prod_{V\in V({{\sf G}(\gamma)})} \Mbar_{\gamma,V}(Y_V|\partial Y_V).
\]
The rigid tropical curve $\gamma$ induces genus, marking, curve class, and tangency data for each $Y_V$, and these determine the spaces of maps in the factors on the right hand side of the arrow above (as a space of logarithmic stable maps to pairs). 

The degeneration formula is proved in our previous work~\cite{MR23,R19} and expresses the pushforward
\[
\kappa_\star[\Mbar_{\gamma}(X_0,\beta)]^{\sf vir} \in H_\star\left(\prod_{V\in V({{\sf G}(\gamma)})} \Mbar_{\gamma,V}(Y_V|\partial Y_V)\right)
\]
in terms of the factor-wise virtual classes, evaluation classes, and logarithmic blowup operations.

In order to do this, two key notions are needed -- virtual birational models of the mapping spaces and exotic evaluation classes.

Fix a vertex $V$ of ${{\sf G}(\gamma)}$ and the associated target $Y_V$ described above. The logarithmic structure on $Y_V$ determines a cone complex $\Sigma_V$. For each marked point $p_i$ with $i$ in $[n]_V$, there is an associated edge/leg. The edge direction decoration determines tangency orders for $p_i$ with the components of the boundary divisor $\partial Y_V$. 

We have an evaluation map
\[
\Mbar_{\gamma,V}(Y_V|\partial Y_V)\to \mathsf{Ev}_{\gamma,V}(Y|\partial Y) = \prod_i W_i
\]
where the final product is taken over the half-edges incident to $V$. 

The moduli space $\Mbar_{\gamma,V}(Y_V|\partial Y_V)$ carries a virtual class, but in fact, this virtual class lifts to logarithmic Chow homology~\cite[Section~3.6]{R19}. We explain this without using the formalism of logarithmic Chow. 

Let $\Mbar_\Lambda(Y|\partial Y)$ be a logarithmic stable mapping space and consider the evaluation space
\[
{\sf ev}\colon \Mbar_\Lambda(Y|\partial Y)\to \mathsf{Ev}_\Lambda(Y|\partial Y).
\]
The space $\mathsf{Ev}_\Lambda(Y|\partial Y)$ is itself a simple normal crossings pair. Let
\[
\mathsf{Ev}_\Lambda(Y|\partial Y)^\diamond\to \mathsf{Ev}_\Lambda(Y|\partial Y)
\]
be a strata blowup -- a sequence of blowups along smooth centers. We can form the pullback in the category of fine and saturated logarithmic algebraic stacks:
\[
\begin{tikzcd}
    \Mbar_\Lambda(Y|\partial Y)^\diamond\arrow{r}\arrow{d} & \Mbar_\Lambda(Y|\partial Y)\arrow{d}\\
    \mathsf{Ev}_\Lambda(Y|\partial Y)^\diamond\arrow{r}& \mathsf{Ev}_\Lambda(Y|\partial Y).
\end{tikzcd}
\]
By the results of~\cite[Section~3]{R19}, the stack $\Mbar_\Lambda(Y|\partial Y)^\diamond$ carries a virtual fundamental class, and the morphism
\[
\Mbar_\Lambda(Y|\partial Y)^\diamond\to \Mbar_\Lambda(Y|\partial Y)
\]
is proper, and identifies virtual classes under pushforward. We therefore refer to these spaces as {\it virtual birational models}. 

We will be interested in classes obtained by pulling back cohomology classes along
\[
\Mbar_\Lambda(Y|\partial Y)^\diamond\to \mathsf{Ev}_\Lambda(Y|\partial Y)^\diamond
\]
and capping with the virtual class. There will be several variants of this -- in singular homology, Chow homology, logarithmic singular, and  logarithmic Chow -- but the basic picture is the same for all of them. 

\begin{definition}\label{def: exotic-insertion}
    An {\it exotic insertion} is an element in the (singular or Chow) cohomology of some blowup $\mathsf{Ev}_\Lambda(Y|\partial Y)^\diamond$ of $\mathsf{Ev}_\Lambda(Y|\partial Y)$ that is not pulled back from $\mathsf{Ev}_\Lambda(Y|\partial Y)$.
\end{definition}

Two important things should be kept in mind when working with exotic invariants and virtual birational models. 

\begin{remark}
The square above is typically {\it not} Cartesian in the category of algebraic stacks, so one cannot use pull/push compatibility to recover classes of these form, even after pushforward to $\mathsf{Ev}_\Lambda(Y|\partial Y)$, just by using the space $\Mbar_\Lambda(Y|\partial Y)$ and its evaluation space. See~\cite{MR21,PR24} for further discussion.    
\end{remark}

\begin{remark}
    The space $\mathsf{Ev}_\Lambda(Y|\partial Y)$ is itself a product of various strata of $(Y|\partial Y)$. It is crucial for our purposes that the strata blowups of $\mathsf{Ev}_\Lambda(Y|\partial Y)$ need not preserve the product decomposition -- they are ``criss-cross''. 
\end{remark}

\subsection{The degeneration formula}\label{sec: deg-formula} We state the degeneration formula in the form needed for this paper. See~\cite[Section~8]{MR23} for a slightly more explicit discussion. 

Fix a rigid tropical curve $\gamma$ as above. The evaluation space at each vertex a preferred decomposition into a product, one for each leg incident to a vertex. Given $\gamma$, each edge $E$ in ${{\sf G}(\gamma)}$ corresponds to a leg in two mapping spaces -- with targets $Y_V$ and $Y_{V'}$, if $V$ and $V'$ are the vertices of $E$. Therefore, if we consider the product
\[
\prod_{V\in V({{\sf G}(\gamma)})} \mathsf{Ev}_{\gamma,V}(Y_V|\partial Y_V)
\]
then for each edge $E$, there is a repeated factor in the product. This defines a natural ``diagonal'' locus
\[
\Delta_\gamma\subset \prod_{V\in V({{\sf G}(\gamma)})} \mathsf{Ev}_{\gamma,V}(Y_V|\partial Y_V).
\]
We will now blowup the ambient space in which this diagonal lives. The strict transform of of this diagonal will always be denoted $\Delta^\diamond_\gamma$. We have an evaluation morphism
\[
\mathsf{ev}: \prod_{V} \Mbar_{\gamma,V}(Y_V|\partial Y_V)^\diamond \to \prod_{V} \mathsf{Ev}_{\gamma,V}(Y_V|\partial Y_V)^\diamond
\]
by the discussion in the previous section.

We are now able to state the degeneration formula~\cite[Section~8]{MR23}. 

\begin{theorem}\label{thm: degeneration-formula}
For any sufficiently refined logarithmic modification
\[
\prod_{V\in V({{\sf G}(\gamma)})} \mathsf{Ev}_\gamma(\partial Y_V)^\diamond\to \prod_{V\in V({{\sf G}(\gamma)})} \mathsf{Ev}_\gamma(\partial Y_V)
\]
there is an equality of virtual classes:
\[
\kappa_\star[\Mbar_{\gamma}(X_0,\beta)]^{\sf vir,\diamond} = k_\gamma\cdot \left[\prod_{V\in V({{\sf G}(\gamma)})} \Mbar_{\gamma,V}(Y_V|\partial Y_V)^\diamond\right]^{\sf vir} \cap \mathsf{ev}^\star(\Delta_\gamma^\diamond)
\]
in Chow or Borel--Moore homology of $\prod_{V\in V(\Gamma)} \Mbar_{\gamma,V}(Y_V|\partial Y_V)^\diamond$ with rational coefficients, where $k_\gamma$ is a positive rational constant. 
\end{theorem}
Again, the constant $k_\gamma$ is explicit, and comes from gluing multiplicities and automorphisms, but it is irrelevant here.

Note that the class $\Delta_\gamma^\diamond$ lives on a product of spaces indexed by vertices of the graph ${{\sf G}(\gamma)}$ and, in singular cohomology or in Chow when there is a Chow--K\"unneth decomposition, admits a K\"unneth splitting across the factors. These components on the factors are often {\it exotic} insertions. The evaluation spaces can also be written as a product over {\it half-edges} at each vertex, but this product structure cannot usually be preserved while achieving sufficient refinedness. See~\cite[Section~9]{MR23}.

There is an explicit combinatorial criterion for when a modification $\prod_{V} \mathsf{Ev}_\gamma(\partial Y_V)^\diamond$ is {\it sufficiently refined}, for the theorem above to hold. For each vertex, we can take any subdivision such that the induced map
\[
 \Mbar_{\gamma,V}(Y_V|\partial Y_V)^\diamond \to \mathsf{Ev}_{\gamma,V}(Y_V|\partial Y_V)^\diamond
\]
is {\it combinatorially flat}, i.e. on the induced map on cone complexes, every cone surjects onto a cone. This can always be achieved, for instance after sufficiently many barycentric subdivisions, see for instance~\cite[Section~4]{AK00}.


\section{Vanishing cohomology in the degeneration}\label{sec: primitive}

In the degeneration formula of Theorem~\ref{thm: degeneration-formula}, it was important to work with insertions coming from the blowup of the evaluation space -- the {\it exotic} insertions of Definition~\ref{def: exotic-insertion} and~\cite{MR23}, and in particular, to work with blowups of products that are not themselves products. 

We now consider a seemingly unrelated issue -- the degeneration formula in the presence of insertions in the general fiber of a degeneration that may not extend to the total space. Perhaps surprisingly, we find that the same geometry -- exotic insertions -- provide a solution to this problem as well, but now applied to the ``internal'', i.e. zero contact order, markings. 

Unlike much of this paper, the discussion in this section is new already in the smooth pair setting~\cite{ABPZ,Li02,MP06}\footnote{The traditional degeneration formula, e.g. in~\cite{Li02} in the smooth pair geometry, is a formula for GW cycles of the general fiber with insertions {\it from the total space}. }. It is also specific to singular homology, rather than Chow.

\subsection{Setup and the goal} We have a smooth proper variety $\mathcal X$ equipped with a flat and proper map $\mathcal X\to B$ to a smooth proper curve, whose fibers are reduced, simple normal crossings divisors in the total space. 

There are two things that we will analyze. To explain the first, fix a point $b$ in $B$, the restriction map
\[
H^\star(\mathcal X)\to H^\star(\mathcal X_b)
\]
may not be surjective, so there may be more Gromov--Witten cycles associated to $\mathcal X_b$ than those that arise from insertions in the cohomology of $\mathcal X$ via restriction. A class on a general fiber extend to the total space if and only if it is monodromy invariant. 

The second is that the natural space of insertions for GW cycles built from $\mathcal X_b$ is the cohomology of $H^\star(\mathcal X_b)^{\otimes n}$, which is a fiber of the fibered power $\mathcal X^n_B$ of $\mathcal X/B$. However, even if a class in $H^\star(\mathcal X_b)^{\otimes n}$ is monodromy invariant, it may not extend, as the fibered power $\mathcal X^n_B$ has {\it singular} total space, e.g. if $\mathcal X/B$ has a double point fiber, the fibered square has quadric cone singularity. 

The first issue will be controlled by the monodromy invariance property of the virtual class. The second issue -- singularity of the evaluation family -- will be handled by exotic insertions.

\noindent
{\bf Goal.} {\it Describe a subspace of $H^\star(\Mbar_{g,n};\QQ)$ containing all GW cycles of a general fiber $\cX_\eta$ in terms of appropriate ``exotic'' logarithmic GW cycles of the strata of a singular fiber $\cX_0$.}

\begin{remark}
    In special geometries, one can make higher resolution statements by understanding the monodromy action in more detail and obtain explicit formulae, see~\cite{ABPZ} for the case when $\mathcal X_b$ is a complete intersection in $\mathbb P^m$. However, to our knowledge, even the case of hypersurfaces in smooth toric varieties has not been understood in this level of detail. 
\end{remark}


\subsection{Some basic Hodge theory}\label{sec: hodge-basics} The result of the section is based on Deligne's study of monodromy actions applied to $\cX_B^n\to B$. 

We also know, again from work of Deligne~\cite{HodgeI}, the action of $\pi_1$ of curve $B$ minus the critical points on the cohomology of a smooth fiber is semisimple. We therefore have a projection operator
\[
H^\star(\cX_\eta)\to H^\star(\cX_\eta)^{\pi_1}. 
\]
So every cohomology class on $\cX_\eta^n$ has a monodromy invariant projection.

\begin{remark}\label{sec: poincare-compatibility}
The monodromy group preserves the Poincar\'e pairing on $H^\star(\mathcal X_b)$, and the invariant subspace is orthogonal (with respect to the Poincar\'e pairing) to the complementary summand.
\end{remark}

We record the following:

\begin{proposition}
    Let $\cX_\eta$ be a general fiber of $\cX\to B$. Let $\alpha\in H^\star(\cX_\eta^n)$ be cohomology class and let $\overline\alpha$ be its monodromy invariant projection. Consider the evaluation map
    \[
    {\sf ev}\colon \Mbar_{g,n}(\cX_\eta,\beta)\to \cX_\eta^n
    \]
    and let $\pi\colon \Mbar_{g,n}(\cX_\eta,\beta)\to \Mbar_{g,n}$ be the forgetful map. There is an equality of GW cycles
    \[
    \pi_\star\left([\Mbar_{g,n}(\cX_\eta,\beta)]^{\sf vir}\cap {\sf ev}^\star(\alpha)\right) = \pi_\star\left([\Mbar_{g,n}(\cX_\eta,\beta)]^{\sf vir}\cap {\sf ev}^\star(\overline{\alpha})\right) \  \textnormal{in } H^\star(\Mbar_{g,n};\QQ).
    \]
\end{proposition}

\begin{proof}
    Since the virtual class is monodromy invariant, the pushforward of the virtual class to the cohomology of $\Mbar_{g,n}\times\cX_\eta^n$ is orthogonal to $\alpha-\overline{\alpha}$. The result follows from Remark~\ref{sec: poincare-compatibility}.
\end{proof}

Next, we recall the statement of Deligne's global invariant cycle theorem~\cite{HodgeII}. 

\begin{theorem}
Let $Y^\circ\to B^\circ$ be a smooth projective map of algebraic manifolds. Let $Y^\circ \hookrightarrow Y$ be an open immersion into a projective manifold with a map to a proper base $B$, and fix $b\in B$. Then the image of $H^\star(Y)$ in $H^\star(Y_b)$ under restriction coincides with the subspace of monodromy invariants in $H^\star(Y_b)$. 
\end{theorem}

In particular, every monodromy invariant cycle in $H^\star(\cX_\eta^n)$ extends over any semistable modification of $\cX_B^n\to B$. The next step will be to construct such a semistable degeneration. 

The smoothness of $Y$ is crucial in the theorem -- toroidal families do not suffice. Indeed, even the Poincar\'e dual of a hypersurface in the general fiber may not extend as a cohomology class in the presence of toric singularities.

\subsection{A semistable degeneration of the evaluation space}\label{sec: ss-eval} We assume for simplicity in this section that the evaluation space ${\mathcal E}v_\Lambda\to B$ is equal to the fibered power of $\cX\to B$. It is a notational exercise to extend this to other evaluation spaces.

In order to apply Deligne's theorem, the total space of the degeneration must be smooth, and as we have noted the evaluation space is not because a fibered power of an snc family is not. We can perform semistable reduction to the evaluation space as follows. The family
\[
\mathcal X^n_B\to B
\]
is toroidal. Its cone complex structure can be described as follows. Recall we described the map of cone complexes,
\[
\Sigma_{\cX}\to \Sigma_B,
\]
induced by $\cX\to B$. The target has a single vertex and one ray for each singular fiber in the family. Each cone of $\Sigma_{\cX}$ maps to $\Sigma_B$ as a cone over a unimodular simplex, of dimension equal to the fiber dimension of $\cX/B$. The cone complex of the fibered power is the fibered power of the cone complex, so the cones of 
\[
\Sigma_{\cX^n_B}\to \Sigma_B
\]
are cones over products of unimodular simplices. We could now perform semistable reduction following~\cite{KKMSD}, but the toroidal singularities here are of a very particular form one can do this very efficiently:

\begin{lemma}
    There exists a subdivision $\Sigma^\diamond_{\cX^n_B}\to \Sigma_{\cX^n_B}$ that (i) is smooth, (ii) adds no new rays. In particular, the primitive generators for the rays $\Sigma^\diamond_{\cX^n_B}$ generate the image lattice in $\Sigma_B$. The associated birational modification $\cX_B^{n,\diamond}\to B$ is semistable.
\end{lemma}

\begin{proof}
    Products of unimodular simplices can be triangulated into unimodular simplices without adding new vertices, for instance, using a staircase triangulation, see~\cite[Section~2.3]{HPPS21}. The lemma is a consequence.
\end{proof}

\begin{remark}
    For a more detailed discussion of semistabilizations of the evaluation space, including moduli interpretations, see~\cite{SCM24}. For discussion in the case of double point degenerations, see~\cite{PP12}.
\end{remark}

\subsection{Decomposition across tropical curves}\label{sec: decomposition-prim} The properties in the preceding lemma are not crucial, but the construction is pleasant and avoids the ramified base change step that is needed in the general case, so we have used it. The rest of the discussion can be carried out for an arbitrary semistabilization. 

We have a semistabilization
\[
\cX_B^{n,\diamond}\to \cX_B^n\to B
\]
of the family of evaluation spaces. The mapping space $\Mbar_{g,n}(\cX,\beta)$ will typically not have a map to this modification,  for the same reasons as we saw in Section~\ref{sec: deg-formula}. Perform fine and saturated logarithmic base change:
\[
\begin{tikzcd}
    \Mbar_{g,n}(\cX,\beta)^\diamond\arrow{d}\arrow{r} & \Mbar_{g,n}(\cX,\beta)\arrow{d}\\
    \cX_B^{n,\diamond}\arrow{r} & \cX_B^n.
\end{tikzcd}
\]
Again, the square is not typically cartesian. There is a compatibility of virtual classes by the same arguments as before, so we can conclude that pushforward along the top arrow identifies virtual structures on each fiber over $B$, so for any fiber $b$, we have a proper map
\[
\Mbar_{g,n}(\cX_b,\beta)^\diamond\to \Mbar_{g,n}(\cX_b,\beta),
\]
and this identifies virtual classes. In particular, this is true for any singular fiber. 

Let $\cX_b$ be a singular fiber of $\cX\to B$ and consider an insertion $\overline\alpha$ in $H^\star(\cX_\eta^n)$ that is invariant under the action of the monodromy group. Let $\alpha_b$ be the specialization to $\cX_b^{n,\diamond}$. We now decompose the GW cycle associated to $\overline\alpha_b$. 

The decomposition theorem stated in Section~\ref{sec: decomposition} applies equally well to $\Mbar_{g,n}(\cX_b,\beta)^\diamond$, as explained in~\cite[Section~7]{MR23}. There is a decomposition
\[
[\Mbar_{g,n}(\cX_b,\beta)^\diamond]^{\sf vir} = \sum_\gamma [\Mbar_{\gamma}(\cX_b,\beta)^\diamond]^{\sf vir}
\]
where $\gamma$ is a tropical map. These tropical maps are not ``rigid'' in the usual sense, but there are a finite set of distinguished ones, as explained below.

\begin{remark}
    The tropical curves $\gamma$ index the rays of the cone complex of $\Mbar_{g,n}(\cX_b,\beta)^\diamond$ that dominate the cone in $\Sigma_B$ corresponding to the point $b$. In the usual situation~\cite{ACGS15}, each $\gamma$ is a {\it rigid tropical map} -- a map that cannot be deformed within its combinatorial type. In the model $\Mbar_{g,n}(\cX_b,\beta)^\diamond$, the rays of the corresponding cone complex need not be indexed by rigid tropical maps -- they can correspond to an arbitrary tropical map. The tropical maps $\gamma$ will simply correspond to the generators of the rays in the subdivision that defines $\Mbar_{g,n}(\cX_b,\beta)^\diamond$.
\end{remark}

Fix an index $i\in[n]$. The marked point $p_i$ is attached to a vertex $V$ of $\gamma$. The tropical map $\gamma$ maps the vertex $V$ to a point in the cone complex $\Sigma_{\cX^n_B}$, and by subdividing to include this point, we obtain a component $Y_V$. This component is an equivariant compactification of the normal torus bundle of the stratum corresponding to the cone containing $V$. As we have noted before, the precise compactification of this torsor is not important. 

We have an evaluation map
\[
{\sf ev}_i\colon \Mbar_\gamma(\cX_b,\beta)\to Y_V. 
\]
We can arrange for there to be a factorization:
\[
\Mbar_\gamma(\cX_b,\beta)^\diamond\to \left(\prod_V Y_V \right)^\diamond\to \cX_b^{n,\diamond}
\]
where $\left(\prod_V Y_V \right)^\diamond$ is a blowup of the product. Indeed, we can first produce a map
\[
\prod_V Y_V\to \cX_b^n
\]
by identifying each $Y_V$ with a torus torsor over a stratum of $\cX_0$. The evaluation space from $\Mbar_\gamma(\cX_b,\beta)$ itself, before blowup, factors through $\prod_V Y_V$ by construction. We then pull back the logarithmic modification $\cX_b^{n,\diamond}\to\cX_b^{n}$, and obtain the factorization above. 

We now have the ingredients to establish the main goal of this section. The statement of the result requires the notion of an exotic insertion from Definition~\ref{def: exotic-insertion}. For convenience we introduce one more piece of terminology -- given $\cX\to B$ and a stratum of a special fiber $\cX_b$, we call its the exceptional fiber of the blowup of the stratum the {\it inflation} of the stratum. 

\begin{theorem}\label{thm: primitive-in-exotic}
    Let $\alpha$ be a insertion in $H^\star(\cX_\eta^n)$. The GW cycle $\pi_\star\left([\Mbar_{g,n}(\cX_\eta,\beta)]^{\sf vir}\cap {\sf ev}^\star(\alpha)\right)$ lies in the span of the exotic logarithmic GW for all pairs $(Y|\partial Y)$ arising as inflations of strata of $\cX_b$. 
\end{theorem}

\begin{proof}
    We bring together the elements of the discussion of the section. First, given $\alpha$, we can use the discussion of Section~\ref{sec: hodge-basics} we can replace it with its projection $\overline \alpha$ onto the monodromy invariants, and this does not change the GW class. Next, we can pass to a semistable model $\cX_B^{n,\diamond}$ of the evaluation space and specialize $\overline\alpha$ to an insertion $\alpha_b$ on the special fiber. We can apply the discussion of Section~\ref{sec: decomposition-prim} to write this new GW cycle as a sum over tropical curves, and for each tropical curve, write its contribution as a pullback of a cohomology class along the enhanced evaluation:
    \[
    \Mbar_\gamma(\cX_b,\beta)^\diamond\to \left(\prod_V Y_V \right)^\diamond.
    \]

    This reduces the problem to the special fiber. To conclude the theorem, we will apply the degeneration formula of Theorem~\ref{thm: degeneration-formula} and a push/pull argument. 
    
    Fix a tropical map $\gamma$ contributing to the calculation above. By the degeneration formula, the GW cycle associated to $\gamma$ can be obtained by pulling back a cohomology class along 
    \[
    \left( \prod_V \Mbar_{\gamma,V}(Y_V) \right)^\diamond\to \left(\prod_{i=1}^n Y_{V(i)}\right)^\diamond\times \left(\prod_V \mathsf{Ev}^{\sf pos}_{\gamma,V}(Y_V|\partial Y_V)\right)^\diamond
    \]
    and pushing forward to $\Mbar_{g,n}$; the superscript ${\sf pos}$ indicates that we have separated out the factors corresponding to contracted legs of $\gamma$, i.e. contact order $0$ markings, and the rest; and also $V(i)$ is the vertex to which the marked point $i$ belongs.

    We can group $\prod_{i=1}^n Y_{V(i)}$ to write it as a product over vertices of $\gamma$:
    \[
    \prod_{i=1}^n Y_{V(i)} = \prod_V {\sf F}_V,
    \]
    where ${\sf F}_V$ is the product of $Y_i$'s with $V(i) = V$. Observe that we can pass to subdivisions such that 
    \[
    \Mbar_{\gamma,V}(Y_V)^\dagger\to {\sf F}_V^\dagger
    \]
    is combinatorially flat for each $V$. We can similarly do this for the evaluation spaces. This leads to a commutative diagram:
    \[
    \begin{tikzcd}
        \left( \prod_V \Mbar_{\gamma,V}(Y_V) \right)^\diamond\arrow{r}\arrow{d} & \prod_V \Mbar_{\gamma,V}(Y_V)^\dagger\arrow{d}\\
        \left(\prod_{i=1}^n Y_{V(i)}\right)^\diamond \times \left(\prod_V \mathsf{Ev}^{\sf pos}_{\gamma,V}(Y_V|\partial Y_V)\right)^\diamond \arrow{r} & \prod_V {\sf F}_V^\dagger\times \prod_V \mathsf{Ev}^{\sf pos,\dagger}_\gamma(Y_V|\partial Y_V).
    \end{tikzcd}
    \]
    The GW cycle we are interested in is the pushforward to $\Mbar_{g,n}$ of a pullback of a cohomology class along the left vertical, capped with the virtual class. The cohomology class is itself the pullback of $\overline\alpha_b$ on the first factor, and the strict transform of the diagonal on the second factor.

    We can also assume both spaces on the bottom row are smooth, and so the arrow has an lci pullback. 
    
    Now by compatibility of virtual classes under subdivisions, the GW class of interest is unchanged if we replace the top left of the diagram by its fine and saturated base change, so we assume it is logarithmically cartesian. But the right vertical is combinatorially flat by construction, so it is also stack theoretically cartesian. By the same discussion used in~\cite[Section~12]{MR23}, the vertical arrows have compatible obstruction theories giving their usual virtual classes, and so a pull/push argument shows that the GW class of interest is equal to a class on the space $\prod_V \Mbar_{\gamma,V}(Y_V)^\dagger$. Since these are, by definition, exotic cycles, we conclude the result. 
\end{proof}

\section{An ordering on tropical curves}\label{sec: stars-partial-ordering}

We have seen that {\it exotic} cycles are basic features in our analysis of the behaviour of GW cycles under degenerations, both in Theorem~\ref{thm: degeneration-formula} and Theorem~\ref{thm: primitive-in-exotic}. Our next conceptual task is to reconstruct the exotic cycles in terms of standard logarithmic GW cycles. The reconstruction is inductive in nature -- we start with a non-exotic approximation of an exotic cycle, and argue that the rest of the terms are of ``smaller complexity'', and can be split off into smaller problems of the same kind. 

In order to make the notion of ``smaller complexity'' precise, in this section, we introduce and analyze an ordering on the discrete data governing the different moduli problems. 

\begin{remark}[Tropical data structures]
    The standard data type that indexes strata in the space of logarithmic stable maps are combinatorial types of {\it tropical stable maps} -- metric dual graphs with genus, degree, and marking decorations, with a piecewise linear map to $\Sigma$, see~\cite{ACGS15,GS13}. In logarithmic DT theory, the natural data type is a {\it Hilbert $1$-complex}, see~\cite[Section~2]{MR23} -- a $1$-dimensional polyhedral subcomplex of $|\Sigma|$ with Euler characteristic, degree, and weights on the edges. Both give rise to a common type of object -- a {\it Chow $1$-complex}, which we define precisely in a moment. Although DT theory does not appear here, the ordering, as well as the exotic/non-exotic manoeuvre, will be useful in later work. We therefore set the theory up in as uniform a way as possible.
\end{remark}

\subsection{Preliminaries} Fix a target $Y$ with simple normal crossings boundary divisor $\partial Y$. Let $\Sigma$ denote the cone complex. It will be convenient to view
\[
|\Sigma|\subset\mathbb R^r
\]
where $r$ is the number of rays: each ray of $\Sigma$ is embedded as the positive direction in the corresponding axis, and the rest is determined by linearity\footnote{Strictly speaking if this is to be an embedding we require the strata of $(Y|\partial Y)$ to be connected. If this is not the case, we can blowup to make it true, or adapt the arguments here with a little more care. We leave the details to the reader.}.

Let $p$ be a point of $|\Sigma|$. The {\it space tangent directions at $p$} are the set of integral tangent vectors at $p$ in the ambient $\mathbb R^r$ that point into one of the cones of $\Sigma$ that contain $p$. We denote it by $T_p\Sigma$ and refer to its elements simply as tangent vectors. It is useful to view a tangent vector as a primitive vector times a positive integer {\it weight}, which we will refer to as the weight of the tangent direction.

\subsubsection{The logarithmic balancing condition}  A tangent direction encodes a possible contact order at a point of a logarithmic map to $Y$, or to a component in an degeneration thereof. The point $p$, and specifically the cone of $\Sigma$ in which it lies, determines the component of such an expanded degeneration. 

The tangent directions that could arise are constrained by the balancing condition~\cite[Section~1]{GS13}. We recall this. Let $\ell_1,\ldots,\ell_r$ be the linear coordinate functions on $\mathbb R^r$. Each corresponds to a boundary divisor in $Y$, say $E_1,\ldots, E_r$. 

\begin{definition}
Let $p\in\sigma^\circ$ be a point of $\Sigma$ and let $\beta\in H_2(W_\sigma;\ZZ)$ be an effective curve class on the associated stratum. A tuple of tangent vectors $( v_1,\ldots,v_m)$ with $v_i\in T_p\Sigma$ is called {\it balanced with respect to $\beta$} if for every index $1\leq i\leq r$, the following equality holds:
\[
\sum_{j=1}^m \ell_i(v_j) = \beta\cdot [E_i]\in \mathbb Z^r.
\]
The {\it balancing direction} associated to $\beta$ is the vector
\[
v_\beta = -(\beta\cdot [E_1],\ldots,\beta\cdot [E_m]).
\]
\end{definition}

In other words, to check if $(v_1,\ldots,v_m,\beta)$ is balanced, append $v_\beta$ to the list and check balancing in the traditional sense in tropical geometry -- the sum of the outgoing vectors at $p$ is equal to $0$. We say that a collection of vectors is {\it traditionally balanced} if this latter condition holds without adding any appendages. 

\begin{definition}[Stars and equivalence]\label{def: star}
A {\it star} in $\Sigma$ is a tuple $(p,v_1,\ldots,v_m,\beta,K)$, denoted $\mathsf v$, where $p$ is a point in a cone $\sigma^\circ$, the $v_i$ are tangent vectors based at $p$, the element $\beta$ is an effective curve class on the stratum $W_\sigma$ dual to $\sigma$, and $K$ is a finite set, subject to the condition is $(v_1,\ldots,v_m)$ are balanced with respect to $\beta$. The point $p$ is the {\it base} of the star, the element $\beta$ is its {\it curve class}, and $K$ is its {\it set of internal markings}. 

Two stars $\mathsf v$ and $\mathsf v'$ are {\it equivalent} if their base points both lie in $\sigma^\circ$, translation identifies their set of tangent directions, the curve classes are the same, and if their sets of internal markings have the same cardinality. 
\end{definition}

\begin{remark}[Stars and discrete data]\label{rem: stars-discrete-data}
A star determines partial discrete data for a logarithmic moduli problem for maps. It is partial for two reasons -- because the genus is not part of the data and, although the total tangency with any given divisor is determined, the precise partition is not specified. To record the rest of the discrete data from a star, we first need a target. The closed stratum $W_p$ dual to the cone containing $p$ has a split normal bundle in $Y$. Pass to the associated torus bundle of this normal bundle, obtained by deleting $0$-sections. Let $\Sigma(p)$ be the star of this cone\footnote{Recall from toric geometry that the star is the projection of the union cones in $\Sigma$ containing the cone $\sigma$ attached to $p$ as a face, into the quotient by the linear space spanned by $\sigma$.}. A broken toric bundle can be built by choosing a fan structure on $\Sigma(p)\times N_{\RR}(p)$, where $N_{\mathbb R}(p)$ is the cocharacter lattice of the fiber torus. The $v_i$ are naturally vectors in this space. Choose a fan structure such that the $v_i$ lie on rays; this corresponds to the disjointness condition. The curve class is fixed; its intersection with the boundary divisors of the bundle is determined by these $v_i$, and the curve class in the base direction, namely $W_p$, is fixed by $\beta$. 
\end{remark}

\begin{definition}\label{def: Chow-1-complex}
A {\it weighted $1$-complex in $\Sigma$} is a $1$-dimensional rational polyhedral complex $\Gamma$ embedded in $|\Sigma|$, together with a positive integer weight on each of its edges given by $w_\Gamma\colon E(\Gamma)\to \NN$. Each flag $(V,E)$ of a vertex determines a tangent vector, given by the primitive direction in the outgoing direction of $E$ scaled by the edge weight. 

A {\it Chow $1$-complex} is a tuple 
$$
\left(\Gamma, w_\Gamma, (\beta_V)_{V\in V(\Gamma)}, (K_V)_{V\in V(\Gamma)}\right).
$$
where $(\Gamma,w_\Gamma)$ is a weighted $1$-complex, the elements $\beta_V$ are effective curve class decorations on the strata determined by $V$, and the $K_V$ are finite sets of internal markings, such that for every $V$, the tangent directions at $V$ are balanced with respect to $\beta_V$. 
\end{definition}

We will slightly abuse notation and write $\Gamma$ for the full data of the Chow $1$-complex. 

One of our main tools is the inductive structure of the moduli space of logarithmic stable maps. In each genus, the natural data structure indexing the different moduli spaces are exactly stars, so in order to control the induction, we need a partial ordering on the set of stars. 

\subsection{Ordering stars -- the traditional case} We first treat the case of traditionally balanced stars and traditionally balanced Chow $1$-complexes. This essentially corresponds to the toric case. We then deduce the general case from this case.

A {\it traditionally balanced star} in $\mathbb R^r$ is a finite set of integral vectors $v_1,\ldots,v_m$ whose sum is zero. A {\it traditional tropical curve} is a weighted $1$-dimensional polyhedral complex in $\mathbb R^r$ whose star at each vertex is traditionally balanced. 

The curve class and internal marking data do not appear for this next part of the discussion.

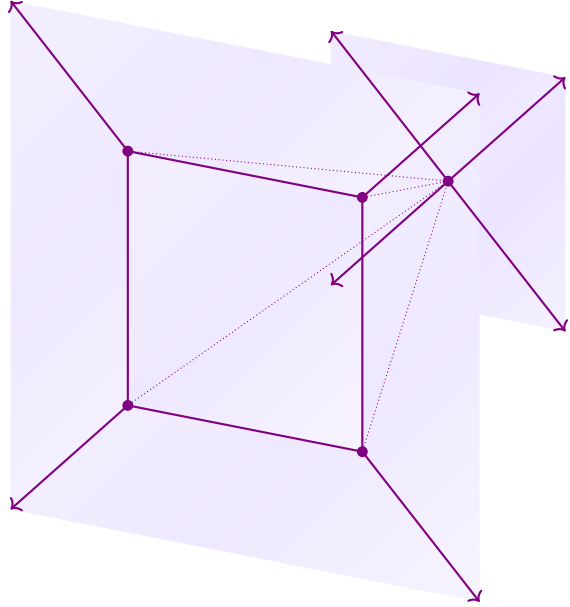
\begin{figure}
\begin{tikzpicture}[tdplot_main_coords, scale=1.8]

\definecolor{customblue}{RGB}{223,210,255}

\pgfdeclareverticalshading{bluecone1}{100bp}{
  color(0bp)=(customblue!20!white);
  color(50bp)=(customblue!60!white);
  color(100bp)=(customblue!35!white)
}
\pgfdeclareverticalshading{bluecone2}{100bp}{
  color(0bp)=(customblue!15!white);
  color(50bp)=(customblue!50!white);
  color(100bp)=(customblue!30!white)
}

\shade[shading=bluecone1, shading angle=45]
  (-2,1,1) -- (-2,0,0) -- (-2,1,-1) -- cycle;
\shade[shading=bluecone1, shading angle=45]
  (-2,1,-1) -- (-2,0,0) -- (-2,-1,-1) -- cycle;
\shade[shading=bluecone1, shading angle=45]
  (-2,-1,-1) -- (-2,0,0) -- (-2,-1,1) -- cycle;
\shade[shading=bluecone1, shading angle=45]
  (-2,-1,1) -- (-2,0,0) -- (-2,1,1) -- cycle;

\shade[shading=bluecone2, shading angle=45]
  (1,1,1) -- (1,2,2) -- (1,2,-2) -- (1,1,-1) -- cycle;
\shade[shading=bluecone2, shading angle=45]
  (1,2,2) -- (1,1,1) -- (1,-1,1) -- (1,-2,2) -- cycle;
\shade[shading=bluecone2, shading angle=45]
  (1,-2,2) -- (1,-1,1) -- (1,-1,-1) -- (1,-2,-2) -- cycle;
\shade[shading=bluecone2, shading angle=45]
  (1,-2,-2) -- (1,-1,-1) -- (1,1,-1) -- (1,2,-2) -- cycle;
\shade[shading=bluecone2, shading angle=45]
  (1,1,1) -- (1,-1,1) -- (1,-1,-1) -- (1,1,-1) -- cycle;

\draw[violet, line width=0.8pt]
  (1,1,1) -- (1,1,-1) -- (1,-1,-1) -- (1,-1,1) -- cycle;

\draw[->,violet, line width=0.8pt] (1,1,1) -- (1,2,2);
\draw[->,violet, line width=0.8pt] (1,1,-1) -- (1,2,-2);
\draw[->,violet, line width=0.8pt] (1,-1,1) -- (1,-2,2);
\draw[->,violet, line width=0.8pt] (1,-1,-1) -- (1,-2,-2);

\draw[->,violet, line width=0.8pt] (-2,0,0) -- (-2,1,1);
\draw[->,violet, line width=0.8pt] (-2,0,0) -- (-2,-1,1);
\draw[->,violet, line width=0.8pt] (-2,0,0) -- (-2,1,-1);
\draw[->,violet, line width=0.8pt] (-2,0,0) -- (-2,-1,-1);

\draw[violet, densely dotted]
  (-2,0,0) -- (1,1,1);
\draw[violet, densely dotted]
  (-2,0,0) -- (1,1,-1);
\draw[violet, densely dotted]
  (-2,0,0) -- (1,-1,1);
\draw[violet, densely dotted]
  (-2,0,0) -- (1,-1,-1);

\fill[violet] (-2,0,0) circle (0.4mm);
\foreach \x/\y/\z in {1/1/1,1/-1/1,1/1/-1,1/-1/-1}
  \fill[violet] (\x,\y,\z) circle (0.4mm);

\end{tikzpicture}

\caption{The asymptotic star of a $1$-complex.}\label{fig: asymptotic}
\end{figure}

Recall that a traditional tropical curve in $\mathbb R^r$ has an {\it asymptotic fan}, obtained constructing the cone
\[
\mathsf{Cone}(\Gamma)\subset \mathbb R^r\times\mathbb R_{\geq 0}\to \mathbb R_{\geq 0}
\]
and taking the fiber over $0$. The fiber is naturally weighted -- there is retraction of graphs from the underlying graph of the fiber over $1$ to the fiber over $0$, and the weight is given by adding the weights of the edges according to this retraction map. The star of the asymptotic fan is called the {\it asymptotic star}, see Figure~\ref{fig: asymptotic}.

There are two ways to obtain (traditionally balanced) stars from traditionally balanced tropical curves -- choose a vertex and use its tangent directions or take the asymptotic fan. These correspond, respectively, to the component of a reducible map and the general fiber of a smoothing. The interaction between these gives rise to the ordering. 

Recall that two stars are equivalent if they differ by translation. 

\begin{definition}
Let $\mathsf v$ and $\mathsf w$ be traditionally balanced stars in $\mathbb R^r$. We say that $\mathsf w$ is {\it smaller} than $\mathsf v$, written as $\mathsf w\preceq \mathsf v$ if there exists a traditionally balanced tropical curve $\Gamma$ in $\mathbb R^r$ with asymptotic star equivalent to $\mathsf v$ and a vertex $V$ of $\Gamma$, such that the star at $V$ is equivalent to $\mathsf w$. 
\end{definition}

We now prove that this is a useful ordering for inductive purposes.

\begin{proposition}\label{prop: ordering}
The following properties hold for the relation $\preceq$ among stars in $\mathbb R^r$ up to equivalence:
\begin{enumerate}[(i)]
\item If $\mathsf w$ is smaller than $\mathsf v$, and $\mathsf v$ is smaller than $\mathsf w$, then they must be equivalent. 
\item The relation is ``finite below'', i.e. for any star $\mathsf v$, the set of $\mathsf w$ such that $\mathsf w\preceq \mathsf v$ is finite. 
\end{enumerate}
\end{proposition}


We prove this proposition in two steps. First for $r\leq 2$, and then in general. 

\begin{lemma}\label{lem: two-dimensional-ordering}
The statement of Proposition~\ref{prop: ordering} holds when $r\leq 2$. 
\end{lemma}

\begin{proof}
If $r = 1$ the statement is essentially empty. If a star is bivalent in any dimension, then every tropical curve with those asymptotic directions has only bivalent vertices by the balancing condition. The result in that case again follows from elementary arguments. 

Now suppose $r = 2$ and neither star is bivalent. We appeal to Newton polygon techniques\footnote{We refer the reader to~\cite[Chapter~6,7]{GKZ}.}. By standard toric geometry, associated to any star $\mathsf v$ is a Newton polygon $\mathsf N(\mathsf v)$.\footnote{Geometrically, we take the toric surface determined by the primitive tangent directions of $\mathsf v$ and consider the line bundle whose intersection multiplicities with the divisors are given by the weights. This is ample as the weights are positive, so determines a polarized toric surface.} There is a natural bijection:
\[
\{\textnormal{Tropical curves with asymptotic star $\mathsf v$}\}\leftrightarrow \{\textnormal{Coherent polytopal lattice subdivisions of $\mathsf N({\sf v})$}\}.
\]
Under this bijection, the vertices of the tropical curve correspond to the top dimensional cells of the polytopal complex, and the polytope determined by the star of a vertex is exactly the corresponding polytopal cell. Now, if $\mathsf w\preceq \mathsf v$, then the volume of $\mathsf N(\mathsf w)$ is at most that of $\mathsf N(\mathsf w)$ with equality only when $\mathsf w = \mathsf v$. The first statement of Proposition~\ref{prop: ordering} is a consequence. Since the volume of a lattice simplex is bounded below, the finiteness from below also follows. 
\end{proof}

\noindent
{\it Proof of Proposition~\ref{prop: ordering}}. Let $\pi\colon \mathbb R^r\to\mathbb R^2$ be an integral linear projection. The image of a star resp. a tropical curve in $\mathbb R^r$ under a projection of this form is again a star. resp. a tropical curve. Furthermore, the relation is compatible with specialization of tropical curves, in the sense that if $\mathsf w\preceq \mathsf v$ then $\pi(\mathsf w)\preceq \pi(\mathsf v)$. By Lemma~\ref{lem: two-dimensional-ordering}, if $\mathsf w$ is smaller than $\mathsf v$, and $\mathsf v$ is smaller than $\mathsf w$, then they must be equal after {\it any projection} from $\mathbb R^r\to\mathbb R^2$. But it is straightforward to see that a tropical curve is completely determined by all its projections to $\mathbb R^2$. 

For the finiteness from below, we again appeal to the result of Nishinou--Siebert~\cite[Section~2]{NS06}. Specifically,  it is known that every tropical curve in $\mathbb R^r$ with asymptotics $\mathsf v$ is a subcurve in a complete intersection of pullbacks of of $(r-1)$ tropical curves under projections $\mathbb R^r\to \mathbb R^2$, in such a way that the required projections depend only on $\mathsf v$, and the asymptotic directions of these tropical curves are determined by the images of $\mathsf v$. The finiteness from below is a consequence. \qed

\subsection{Ordering stars -- the general case}  Equipped with Proposition~\ref{prop: ordering} we place an ordering on stars for any target. Fix $Y$ with cone complex $\Sigma$; the definitions below need the information of $Y$ even to make sense. 

The notions of the star at a vertex and the asymptotic star make perfect sense for any Chow $1$-complex in $\Sigma$. For the asymptotic version, the asymptotic fan is constructed exactly as in the traditional version, with the additional fact that the curve classes are pushed forward along $W_\sigma\hookrightarrow X$ and then added. Similarly, we take the disjoint union of the marking sets. 

\begin{definition}\label{def: general-ordering}
Let $\mathsf v = (p,v_1,\ldots,v_m,\beta_{\mathsf v},K_{\mathsf v})$ and $\mathsf w = (q,w_1,\ldots,w_s,\beta_{\mathsf w},K_{\mathsf w})$ be stars in $\Sigma$. We say that $\mathsf w$ is {\it smaller} than $\mathsf v$, written as $\mathsf w\preceq \mathsf v$ if:
\begin{enumerate}[(i)]
\item The cone whose relative interior contains $q$ contains the cone whose relative interior contains $p$ as a face, and in the case of equality, 
\item the curve class $\beta_{\mathsf w}$ is smaller than $\beta_{\mathsf v}$ in the usual partial order on effective classes, and in the case of equality
\item after appending the balancing ray $v_\beta$, and looking at the associated traditionally balanced stars in $\sigma^{\sf gp}$, denoted $\widetilde w$ and $\widetilde v$, the star $\widetilde w$ is smaller than $\widetilde v$.
\end{enumerate}
\end{definition}

The ordering is geometrically motivated -- as a logarithmic stable map degenerates, its tropicalization changes, and we are paying attention to $\mathsf v$ which is the imagined stable map we start with, and $\mathsf w$ which is a component in the imagined degeneration. Geometrically, (i) corresponds to components moving into deeper strata, (ii) corresponds to a curve breaking and the curve class distributing, and (iii) is the specialization of the tangency data. The final one is subtle, because there are in principle infinitely many potential divisors, corresponding to divisors in blowups and expansions. But this is exactly the case handled in the previous section using tropical techniques.

\begin{proposition}
Let $\mathsf w$ and $\mathsf v$ be stars in $\Sigma$. If according to the ordering of Definition~\ref{def: general-ordering}, we have both that $\mathsf w$ is smaller than $\mathsf v$ and $\mathsf v$ is smaller than $\mathsf w$, then $\mathsf v$ is equal to $\mathsf w$. Furthermore, this ordering is finite below in the set of all such stars. 
\end{proposition}

\begin{proof}
Since the faces of $\Sigma$ form a partially ordered set already and the monoid of effective curve classes also forms a partially ordered set, the only nontrivial check occurs when there is equality in (i) and (ii). In this case, the statement reduces to Proposition~\ref{prop: ordering}. 
\end{proof}

\subsection{A remark about the fully toric case} The following discussion is not logically necessary, but some readers may find it helpful. Suppose $Y$ is a compact toric variety and $\partial Y$ is its {\it full} toric boundary divisor. The embedding procedure we have used places the abstract cone complex $\Sigma$ into a vector space $\mathbb R^r$. The vector space is of much higher dimension than the rank of the torus. We study polyhedral complexes with curve class decorations that are balanced, according to the balancing procedure described above. 

There is nothing wrong with this formalism, but for this fully toric situation, some pleasant features of the geometry are obscured. In essence, instead of Chow $1$-complexes, in this case, we should work with traditionally balanced tropical curves in $N_{\RR}$ with no curve class decoration. These are equivalent data. 

We have an embedding $|\Sigma|\hookrightarrow \mathbb R^r$, which is linear on each cone of $\Sigma$, but only piecewise linear overall. There is a natural map
\[
\mathbb R^r\to N_{\RR}.
\]
To describe it, observe that any linear function $\mathbb R^r$ gives a piecewise linear function on $|\Sigma|$. This determines a line bundle by toric geometry. The line bundles that are trivial determine a subspace of the dual that is identified with the character lattice of the torus. The map above is the dual surjection to the subspace inclusion. The quotient sends the cone complex $\Sigma$ isomorphically onto the toric fan of $Y$. 

Suppose $\Gamma$ is a Chow $1$-complex in $\Sigma$, with respect to $Y$, viewed only as a logarithmic scheme rather than a toric variety. By toric intersection theory, for any balanced tropical curve, all the balancing rays lie in the quotient of this map, and the image in $N_{\RR}$ is balanced in the traditional sense. The class $\beta_{\mathsf v}$ on a star can be recovered from the directions and multiplicities of the rays -- this is true because the toric boundary generates the Picard group.

The upshot -- in the fully toric situation, we can dispense with the Chow $1$-complex and work with traditionally balanced tropical curves. 

\section{Exotic evaluations}\label{sec: exotic-evaluations}

In the overall blueprint for controlling GW cycles by degeneration, by Theorem~\ref{thm: degeneration-formula} and Theorem~\ref{thm: primitive-in-exotic}, any GW cycle of the general fiber $\cX_\eta$ of a degeneration can be expressed in terms of logarithmic GW cycles of bundles over strata in a special fiber $\cX_0$. However, these logarithmic GW cycles will involve exotic insertions. 

In this section and the next one we analyze exotic evaluations and turn them back into standard insertions. We do this in a two-step procedure. In the first step, we turn exotic GW classes into certain {\it decorated strata GW classes}. These are virtual strata of the moduli space of logarithmic maps, with additional tautological classes from the Artin fan of the moduli space, but without any exotic insertions. In the second step, we take these more complicated classes and develop a rigidification calculus to convert them into standard cycles. 

The first step is the focus of the present section, with the main result appearing as Theorem~\ref{thm: exotic-to-strata} below; the second step is carried out in Section~\ref{sec: rigidification}.

\begin{remark}[Genus independence of the algorithm]\label{rem: genus-independence}
    We prove something stronger in the analysis, by producing an algorithm to convert exotic insertions into strata classes in a way that is independent of the genus parameter. The additional strength is not necessary here, but is useful in GW/DT/PT comparison statements, and will appear in future work.
\end{remark}

\subsection{Numerical setup} Fix a pair $(Y|\partial Y)$ as above with cone complex $\Sigma$, and a star in the sense of Definition~\ref{def: star}, based at the origin in $\Sigma$. It is specified by a collection
\[
(\mathsf v,\beta,[k])
\]
of a set of tangent vectors, balanced with respect to a curve class $\beta$ on $Y$, and a set of internal markings indexed by $[k]$. We also assume that the contact order is disjoint, i.e. the $v_i$ are parallel to the rays of $\Sigma$. 

For each tangent vector $v_i$ in $\mathsf v$, we can view it as
\[
v_i = w_i\cdot \overline v_i, \ \ w_i\in\mathbb N.
\]
for $w_i$ a {\it weight} and $\overline v_i$ a primitive integral vector. For each $i$ we fix an {\it ordered} partition
\[
\mu_i\vdash w_i
\]
of the weight. Altogether we fix a vector of partitions, one for each $v_i$ denoted
\[
\bm \mu = (\mu_1,\ldots,\mu_m).
\]
The discrete data $(\mathsf v,\beta,[k],{\bm\mu})$ can almost be converted into the discrete data of a logarithmic stable maps problem. We also need to specify the genus. Once we fix a genus $g$, then we obtain discrete data $\Lambda$: the curve class $\beta$ is fixed, the sum of the lengths of the partitions in ${\bm \mu}$ plus $k$ is equal to the number of marked points. The partitions $\mu_i$ also determine the tangency order along the markings; the $k$ markings have $0$ tangency.

\subsection{Spaces of Chow $1$-complexes} We work towards building the appropriate moduli spaces of Chow $1$-complexes. We do this in steps; in the first step, we ignore the internal markings and tangency conditions.

We introduced the notion of a Chow $1$-complex in Definition~\ref{def: Chow-1-complex}. Given a Chow $1$-complex $\Gamma$, there is a star obtained by the asymptotic fan construction.

The collection of all Chow $1$-complexes with asymptotic star equal to $\mathsf v$ is parameterized by a space $\mathcal T_\mathsf v(Y|\partial Y)$. As with many of the constructions in~\cite{MR23}, the cone structure on this space is not canonical, but there exists an inverse system parameterizing cone structures on it.\footnote{The space of Chow $1$-complexes with curve class decoration is constructed in~\cite{MR23} without the imposed balancing condition. Only cosmetic changes are needed in the construction to include the balancing condition.} 

\begin{definition}
Let $\Gamma$ be a Chow $1$-complex in $\Sigma$. A vertex of valency $0$ is called a {\it free vertex}. A bivalent vertex $V$ of $\Gamma$ is called {\it linear} if some open neighborhood of $V$ is contained in the relative interior of a cone of $\Sigma$, and moreover is contained in a straight line in that cone. 

A Chow $1$-complex $\Gamma$ is {\it stable} if it has neither linear bivalent vertices with curve class $0$ nor free vertices with curve class $0$. 
\end{definition}

To orient the reader, note that unstable Chow $1$-complexes can arise from geometry -- the presence of a logarithmic stable map can stabilize it. The definition captures those vertices that look unstable at the level of the underlying cycle. Free vertices are relatively rare in connected GW theory, but are ever-present in DT, PT, and disconnected GW theory.

Let $T_{\mathsf v}(Y|\partial Y)\subset \mathcal T_{\mathsf v}(Y|\partial Y)$ be the subspace of stable Chow $1$-complexes. 

\begin{proposition}
    The space of stable Chow $1$-complexes $T_{\mathsf v}(Y|\partial Y)$ is a finite dimensional cone complex. 
\end{proposition}

\begin{proof}
    Given a Chow $1$-complex, we obtain a traditionally balanced tropical curve by appending the balancing rays, as in Section~\ref{sec: stars-partial-ordering}.  The result now follows from the boundedness of~\cite[Section~2]{NS06}.
\end{proof}

Observe that we obtain boundedness without fixing the genus. This requires both stability and balancing; it is parallel to boundedness of the Chow variety of cycles of fixed degree. 

There is a {\it stabilization} map
\[
\mathcal T_{\mathsf v}(Y|\partial Y)\to  T_{\mathsf v}(Y|\partial Y)
\]
obtained by ``erasing'' the linear bivalent vertex -- a neighborhood of a linear vertex looks like a subdivided open interval in a vertex space, and we simply undo the subdivision. At the level of topological spaces, the composition
\[
|T_{\mathsf v}(Y|\partial Y)|\subset |\mathcal T_{\mathsf v}(Y|\partial Y)|\to  |T_{\mathsf v}(Y|\partial Y)|
\]
is the identity, so this is a section. We promote this to be compatible with cone structures. 

\begin{proposition}
For appropriate choices of cone structures, we have a diagram of cone complexes,
\[
T_{\mathsf v}(Y|\partial Y) \subset \mathcal T_{\mathsf v}(Y|\partial Y) \to  T_{\mathsf v}(Y|\partial Y)
\]
where the first map is an inclusion of a union of faces and the second map is combinatorially flat. 
\end{proposition}

\begin{proof}
The point is to first construct the space $T_{\mathsf v}(Y|\partial Y)$ of stable $1$-complexes, and then build $\mathcal T_{\mathsf v}(Y|\partial Y)$ over it in a way that makes it obviously flat. By the combinatorial results in~\cite{MR20,MR23}, we can choose some arbitrary cone structure on $|\mathcal T_{\mathsf v}(Y|\partial Y)|$, call it $\mathcal T^{\sf aux}_{\mathsf v}(Y|\partial Y)$ and take the induced cone structure on the stable locus to be $\mathcal T_{\mathsf v}(Y|\partial Y)$. Now we forget about $\mathcal T^{\sf aux}_{\mathsf v}(Y|\partial Y)$. A point of $|\mathcal T_{\mathsf v}(Y|\partial Y)|$ differs from its stabilization only by finitely many $2$-valent vertices on the interiors of its edges. So for each cone $\sigma$ in the space $\mathcal T_{\mathsf v}(Y|\partial Y)$ and a choice of edge in the underlying graph of the associated $1$-complex, we can build a cone complex $\widetilde \sigma$, relatively over $\sigma$, parameterizing subdivisions of this edge. The map $\widetilde\sigma\to\sigma$ is combinatorially flat. Ranging over all cones of $\mathcal T_{\mathsf v}(Y|\partial Y)$ and all edges associated to these cones, the resulting cone complexes glue together to a cone structure on $\mathcal T_{\mathsf v}(Y|\partial Y)$. 

\end{proof}

Note that by semistable reduction, if we change the cone structure on $T_{\mathsf v}(Y|\partial Y)$ by a subdivision , we can always make further refinements $\mathcal T_{\mathsf v}(Y|\partial Y)$ to ensure the flatness. 

\subsection{Tracking the tangency data} We now include the tangency data determined by ${\bm\mu}$. Let $\overline{\bm\mu}$ be the associated vector of {\it unordered partitions} $\overline \mu_i$. We additionally fix finite set $K$ of ``internal labels''. 

\begin{definition}
Denote the subcomplex 
\[
    T_{\mathsf v}^{\overline{\bm\mu}}(Y|\partial Y)\subset T_{\mathsf v}(Y|\partial Y)
\]
defined by taking the closure of the locus of $[\Gamma]$ that satisfy the following condition: for each $i$, the unbounded rays that are parallel to $v_i$ each have multiplicity $\overline \mu_i$. Let
\[
T_{\mathsf v}^{{\bm\mu}}(Y|\partial Y)\to T_{\mathsf v}^{\overline{\bm\mu}}(Y|\partial Y)
\]
be the finite cover obtained by marking the unbounded rays so they are labelled by $\bm\mu$. We refer to this as the {\it space of Chow $1$-complexes with fixed tangency}. Finally, we can also allow for internal markings, forming a space $T_{\mathsf v,K}^{{\bm\mu}}(Y|\partial Y)$ where we allow Chow $1$-complexes $\Gamma$ with up to $|K|$ linear $2$-valent or free vertices with $0$ curve class, record a map
\[
K\to V(\Gamma)
\]
and impose that every such unstable vertex receives at least one element of $K$. If $K = [k]$, we follow the standard conventions and write $T^{\bm \mu}_{\mathsf v,k}(Y|\partial Y)$. We refer to it as the {\it space of Chow $1$-complex with fixed tangency and internal markings labelled by $K$}. The set of labels of the ends, together with $K$, will be referred to as {\it the markings} of the $1$-complex. 

There are analogous cone complexes $\mathcal T_{\mathsf v,k}^{\bm \mu}(Y|\partial Y)$ obtained without the stability condition; these are obtained by pulling back the map above along the stabilization map. 
\end{definition}

\subsection{Maps, evaluations, and Artin fans} Continue to fix discrete data $\mathsf v$, as well as a genus $g$. We can view
\[
\Lambda = (\beta,g,n,[c_{ij}]) ``=" (g,\mathsf v,\bm \mu)
\]
since $\mathsf v$ includes the non-tangency markings and the vector $\bm \mu$ captures the tangency markings. 

Consider the standard moduli space $\Mbar_\Lambda(Y|\partial Y)$ of logarithmic stable maps. The cone space $\mathcal T_{\mathsf v,k}^{\bm \mu}(Y|\partial Y)$ of possibly unstable Chow $1$-complexes has an associated Artin fan, which we denote by $\mathsf{a} \mathcal T_{\mathsf v,k}^{\bm \mu}(Y|\partial Y)$. See~\cite{CCUW} for the basic theory. As explained in~\cite[Section~2.5]{MR23}, there is typically no map
\[
\Mbar_\Lambda(Y|\partial Y)\nrightarrow \mathsf a\mathcal T_{\mathsf v,k}^{\bm \mu}(Y|\partial Y),
\]
however, there is a morphism\footnote{This is a logarithmic morphism, with the obvious logarithmic structures on the two sides, but it is not strict. The logarithmic structure here does not detect nodes of the domain curve that map to the smooth locus of the target.} from a subdivision
\[
\Mbar_\Lambda(Y|\partial Y)^\diamond\to \mathsf a\mathcal T_{\mathsf v,k}^{\bm \mu}(Y|\partial Y).
\]
The space $\Mbar_\Lambda(Y|\partial Y)^\diamond$ can be interpreted as a moduli space of expanded stable maps~\cite{MR23,R19}, though only the existence of the arrow is needed. The map is equipped with a relative perfect obstruction theory, and gives rise to the standard virtual class.

There is a similar picture for the evaluation space $\mathsf{Ev}_\Lambda(Y|\partial Y)$. The evaluation map is, by definition, a product
\[
\mathsf{Ev}(Y|\partial Y) = \prod_i W_i.
\]
The cone complex of this moduli space is $P_{\mathsf v,k}^{\bm \mu}(Y|\partial Y)$. It is just the product of fans of these strata $W_i$. There is a tropical evaluation morphism
\[
\mathcal T_{\mathsf v,k}^{\bm \mu}(Y|\partial Y)\to P_{\mathsf v,k}^{\bm \mu}(Y|\partial Y)
\]
of cone complexes. We can describe it separately for each marked point. For a marked point $i$ that has tangency $0$, it is just the evaluation that tracks the location of the marking in $\Sigma$. For a marked point with positive tangency, in the disjoint case, this corresponds to a ray parallel to one of the rays of $\Sigma$. The position of this marked point in the quotient by the span of the ray, gives the evaluation. In the case where tangencies are not disjoint, it similarly evaluates to the quotient of the span of the cone attached to the marking; we will not use this directly though.

Since the evaluation $\mathcal T_{\mathsf v,k}^{\bm \mu}(Y|\partial Y)\to P_{\mathsf v,k}^{\bm \mu}(Y|\partial Y)$ is insensitive to free or linear bivalent vertices, there is a factorization through the retraction:
\[
\mathcal T_{\mathsf v,k}^{\bm \mu}(Y|\partial Y)\to T_{\mathsf v,k}^{\bm \mu}(Y|\partial Y)\to P_{\mathsf v,k}^{\bm \mu}(Y|\partial Y).
\]
Passing to Artin fans, we have the following commutative diagram:
\[
\begin{tikzcd}
    \Mbar_\Lambda(Y|\partial Y)^\diamond\arrow{d} \arrow{rr}& &\mathsf{Ev}_\Lambda(Y|\partial Y)\arrow{d}\\
    \mathsf a\mathcal T_{\mathsf v,k}^{\bm \mu}(Y|\partial Y)\arrow{r} & \mathsf a T_{\mathsf v,k}^{\bm \mu}(Y|\partial Y)\arrow{r} & P_{\mathsf v,k}^{\bm \mu}(Y|\partial Y).
\end{tikzcd}
\]

We will later be interested in passing to subdivisions such that 
\[
\mathsf{a}\mathcal T_{\mathsf v,k}^{\bm \mu}(Y|\partial Y)\to \mathsf{a}P_{\mathsf v,k}^{\bm \mu}(Y|\partial Y)
\]
is flat. We can always choose 
\[
\mathsf a\mathcal T_{\mathsf v,k}^{\bm \mu}(Y|\partial Y)\to \mathsf a T_{\mathsf v,k}^{\bm \mu}(Y|\partial Y)
\]
so it is already flat, so we can focus on the latter map
\[
\mathsf a T_{\mathsf v,k}^{\bm \mu}(Y|\partial Y)\to \mathsf{a}P_{\mathsf v,k}^{\bm \mu}(Y|\partial Y).
\]
Flattening this by subdivisions and pulling back along stabilization, we obtain a uniform construction of flattening, as needed for the degeneration formula of Theorem~\ref{thm: degeneration-formula}, without genus restrictions. 

\subsection{The exotic/non-exotic manoeuvre}\label{sec: exotic-nonexotic-move} We introduce the required class of GW strata classes for our inductive analysis.  We emphasize that everything in this section assumes disjointness of the contact order matrix. 

We do a few things to reduce the burden of the notation. First, we replace the standard moduli space $\Mbar_\Lambda(Y|\partial Y)$ with a modification as above, so there is a morphism
\[
\Mbar_\Lambda(Y|\partial Y)\to \mathsf{a} T^{\bm \mu}_{\sf v,k}(Y|\partial Y).
\]
Next, we drop discrete data from the notation and use the symbols $\mathsf M$, $T$, $\mathsf{Ev}$ and $\mathsf P$, in place of $\Mbar_\Lambda(Y|\partial Y)$, $T_{\mathsf v,k}^{\bm\mu}(Y|\partial Y)$, $\mathsf{Ev}_\Lambda(Y|\partial Y)$, and $P_{\mathsf v,k}^{\bm\mu}(Y|\partial Y)$, respectively. The discrete data will not change in this section, so there should be no room for confusion. 

In the following, a {\it stratum} of $\mathsf M$ is the preimage of the closure of a point in $\mathfrak{a}T$ under the map
\[
\mathsf{t}\colon \mathsf M \to \mathfrak{a}T.
\]
Note that this is the map to the space of {\it stable} Chow $1$-complexes, so this gives a coarser notion of stratum than standard one. This is in line with Remark~\ref{rem: genus-independence}.

The strata are, of course, in bijection with cones in $T$. If we have a cone $\sigma$, it maps to a cone $\tau$ in $\mathsf P$, we obtain a corresponding stratum in $\mathsf{Ev}$. 

\begin{definition}\label{def: stratum-GW-class}
A {\it stratum GW class} is any class obtained as follows. Choose any subdivisions
\[
\mathsf M^\diamond\to \mathsf{Ev}^\diamond
\]
and let $T^\diamond$ be the corresponding subdivision of $T$. Consider a stratum $\mathsf S\subset\mathsf M^\diamond$ and the corresponding evaluation stratum by $\mathsf{Ev}^\diamond_{\mathsf S}$ leading to
\[
\mathsf{ev}_{\mathsf S}\colon\mathsf S\to\mathsf{Ev}^\diamond_\mathsf{S}. 
\]
Define the GW cycle as the pushforward to $\Mbar_{g,n}$ of
\[
\left(\mathsf{ev}^\star_{\mathsf S}(\upalpha)\cup \mathsf t^\star(\vartheta)\right)\cap [\mathsf S]^{\sf vir},
\]
where $\vartheta\in H^\star(\mathfrak a T^\diamond;\QQ)$ and $\upalpha\in H^\star(\mathsf{Ev}^\diamond_\mathsf{S};\QQ)$. 

A stratum GW cycle is called {\it non-exotic} if $\upalpha$ is pulled back from a product of blowups of the stratum in $\mathsf{Ev}$ to which it maps, {\it strongly non-exotic} if it can be obtained from this stratum without blowups, and {\it exotic} if it is not {\it non-exotic}. 

We make analogous definitions in Chow instead of (co)homology
\end{definition}

We come to the main result of this section. 

\begin{theorem}\label{thm: exotic-to-strata}
The exotic stratum GW cycles associated to the moduli space $\Mbar_\Lambda(Y|\partial Y)$ can be effectively reconstructed from the strongly non-exotic stratum GW cycles. 
\end{theorem}

\begin{proof}
The proof proceeds in several steps. 

\noindent
{\sc Step I. Removing the virtual structure.} We introduce an intermediate space that translates the problem into one about traditional intersection theory, without the presence of the virtual class. Define a space $\mathscr M$, replacing $\mathsf M$, by the following fiber product:
\[
\begin{tikzcd}
\mathscr M\arrow{d}\arrow{r} & \mathsf{Ev} \arrow{d}\\
\mathsf a T\arrow{r} &\mathsf{a} \mathsf P.
\end{tikzcd}
\]
The fiber product is a stack of Artin type, since  the map of cone complexes $T\to \mathsf P$ has positive dimensional fibers. Geometrically, this corresponds to an expansion of $Y$ along $\partial Y$ together with a labelled collection of points on the various expansions of $\mathsf{Ev}_i$'s, compatible with the discrete data of $\mathsf P$. There is a morphism
\[
\mathsf M\to \mathscr M
\]
that has two key properties: (i) it is {\it combinatorially flat} since $\mathsf M\to \mathsf a T$ is combinatorially flat, so compatible with subdivision, and (ii) since the morphism stabilizing $1$-complexes from $\mathcal T\to\mathsf T$ can be chosen to be combinatorially flat, the map above is equipped with a virtual pullback, and this gives rise to the standard virtual class on $\mathsf M$.

Due to the these compatibilities, it suffices to analyze intersections of strata in blowups $\mathscr M^\diamond$, arising from subdivisions of $T$, with classes from $T^\diamond$ and $\mathsf{Ev}^\diamond$. Our goal is to reduce these intersections to statements about strata on the original $\mathscr M$. We can then just pull back to the geometric moduli spaces $\sf M$ themselves.

We prove the theorem by establishing the following, more precise, proposition. Fix a cone $\sigma$ of $T^\diamond$, a cohomology class $\vartheta$ on the Artin fan $\mathfrak{a}T$, and a evaluation class $\alpha$ on the corresponding evaluation space ${\sf Ev}^\diamond_\sigma$. 

Denote by 
\[
[\![ \sigma,\vartheta,\upalpha ]\!]
\]
the pushforward $\mathscr M$ of the class obtained by pulling back $\upalpha$ and $\vartheta$ to $\mathscr M^\diamond$, and capping with the stratum corresponding to $\sigma$.

The class $\upalpha$ if it is pulled back from a class on the product space ${\sf Ev}_S$. 

\begin{proposition}\label{prop: precise-version}
Any, possibly exotic, class of the form $[\![ \sigma,\vartheta,\upalpha ]\!]$ can be expressed in the form
\[
[\![ \sigma,\vartheta,\upalpha ]\!] = \sum_j [\![ \sigma_j,\vartheta_j,\upalpha_j ]\!]
\]
where $\sigma_j$ is a cone of $T$ and $\upalpha_j$ is non-exotic. 
\end{proposition}

\noindent
{\sc Step II. Geometric setup and the plan.} We describe the intersection problems that must be analyzed. By an application of the projection formula, we limit ourselves to considering cases when $\mathsf{Ev}^\diamond\to \mathsf{Ev}$ is a sequence of blowups along smooth centers. Take $\mathscr M^\diamond$ to be {any} logarithmic birational model compatible with this blowup of the evaluation space, i.e. so we have lifting
\[
\mathscr M^\diamond\to\mathsf{Ev}^\diamond
\]
of the evaluation. We take a closed stratum $\mathscr S^\diamond$ in $\mathscr M^\diamond$ whose map $\mathsf{Ev}^\diamond$ factors as $\mathscr{S}^\diamond\to\mathsf{Z}^\diamond$. We then consider two cohomology classes on $\mathscr{S}^\diamond$, namely, a pull back of a class from $\mathsf Z^\diamond$ and a class from the space $\mathfrak aT^\diamond$. We refer to these in short hand as the {\it $T$-class} and the {\it evaluation class}. We take the product of these classes on $\mathscr S^\diamond$, pair with the fundamental class, and push down to $\mathscr M$. Note that $\mathscr M^\diamond$ is proper and birational over $\mathscr M$, and the stratum $\mathscr S^\diamond$ is closed, so the pushforward is well-defined.

\noindent
{\sc Step III. Tuning blowups to evaluation strata.} The goal is to understand intersection products of the following form: select a stratum $\mathsf Z^\diamond$ of $\mathsf{Ev}^\diamond$, then select a stratum $\mathscr S^\diamond$ in its preimage $\mathscr M^\diamond$ that maps to it, pull back an evaluation and a $\vartheta$-class, pair and push forward. 

We first exert some control over the shape of the blowup. This is for notation compactness, and helps avoid having to keep track of various different dimensions. The claim is that it suffices to treat the case when $\mathsf Z^\diamond$ is a divisor or $\mathsf{Ev}^\diamond$ itself. If $\mathsf Z^\diamond$ is not a divisor, then we can achieve this by blowing up $\mathsf Z^\diamond$ inside the ambient space $\mathsf{Ev}^\diamond$. Take the strict transform of $\mathscr M^\diamond$ under the map to $\mathsf{Ev}^\diamond$. The result is a diagram
\[
\begin{tikzcd}
\widetilde{\mathscr M^\diamond}\arrow{r}\arrow{d} & \widetilde {\mathsf{Ev}^\diamond}\arrow{d}\\
\mathscr M^\diamond\arrow{r}& \mathsf{Ev}^\diamond
\end{tikzcd}
\]
We can find a stratum of $\widetilde{\mathscr M^\diamond}$ that dominates $\mathscr S^\diamond$. Since this stratum will map as a broken toric bundle, by passing to a stratum, we can find a stratum that maps birationally onto $\mathscr S^\diamond$, identifying fundamental classes. Now pull back the cohomology insertion from $\mathsf Z^\diamond$ to the exceptional over $\mathsf Z^\diamond$. We can also pull back the $T$-class, so this reduces us to the case where $\mathsf Z^\diamond$ is a divisor. 


The next goal is to move this down to an integral on a stratum of $\mathscr M$ itself, and this forms the bulk of the analysis. 

\noindent
{\sc Step IV. The blowup and bundle analysis.} The case when the stratum that we start with is $\mathsf{Ev}$ itself is essentially the same but easier as when it is a proper substratum, so we only treat the latter case. The system of blowups of a simple normal crossings pair obtained by sequences of blowups along smooth centers is cofinal in the system of all blowups. Therefore we can assume that $\mathsf Z^\diamond$ is obtained as follows: 
\[
Z^\diamond = \mathsf Z(m) \rightarrow\cdots\rightarrow \mathsf Z(2)\rightarrow \mathsf Z(1)\rightarrow\mathsf Z(0),
\]
where $Z(0)$ is a stratum of $\mathsf{Ev}$ and $\mathsf Z(1)$ is a weighted projective space bundle\footnote{It is not really necessary to use weighted blowups here, but it allows us to just do one blowup to create a birational copy of $\mathsf Z^\diamond$ and then the rest of the blowups are to ensure that it dominates $\mathsf Z^\diamond$. It can be done with standard blowups, at the cost of throwing in some initial blowups of strata that meet $Z$ transversely.} over $\mathsf Z(0)$. After this, each step is a blowup, induced by a blowup of the ambient $\mathsf{Ev}(i)$ containing it, at a center that is smooth and transverse to $\mathsf Z(1)$.

Let us first get the intersection problem from $\mathsf Z(m)$ down to $\mathsf Z(1)$, and deal with the bundle formula at the end. We start with a stratum $\mathscr S(m)$, obtained as a stratum in the strict transform of $\mathscr M$ under the ambient blowup of $\mathsf{Ev}(m)\to \mathsf{Ev}$. Let $\mathscr S(m-1)$ be the smallest closed stratum in $\mathscr M(m-1)$ that this $\mathscr S(m)$ maps to. We now have the following diagram:
\[
\begin{tikzcd}
\mathscr S(m)\arrow{r}&{\mathscr F(m)}\arrow{d}\arrow{r} & \mathsf Z(m)\arrow{d} \\
 &\mathscr S(m-1)\arrow{r} & \mathsf Z(m-1)
\end{tikzcd}
\]
where the square is Cartesian and the map from $\mathscr S(m)$ is proper birational onto the strict transform. There are two possibilities for the map $\mathscr S(m-1)\to \mathsf Z(m-1)$: it either has image that is not contained in the center or it is contained in the center. In the first case, the space $\mathscr F(m)$ is the total transform under the blowup, and $\mathscr S(m)$ is the strict transform. In the second case, the vertical map presents $\mathscr F(m)$ as a projective bundle. We take the two possibilities in turn. 

In the first case, we use the analysis outlined in~\cite[Section~3]{MR20} -- apply Fulton's blowup formula and Aluffi's formula for the Segre classes of monomial subschemes. Suppose that $V\subset \mathsf Z(m-1)$ is the center of the blowup and let $\mathbb P_V$ denote the exceptional projective bundle in $\mathsf Z(m)$. Our goal is to move the integral to $\mathscr F(m)$, and use pull/push compatibility to reduce to an integral on $\mathscr S(m-1)$. 

Since we are pushing forward to $\mathscr S(m-1)$, by an application of the projection formula, we can assume that the evaluation class on $\mathsf Z(m)$ is pushed forward from the exceptional divisor $\mathbb P_V$. Fulton's blowup formula expresses the class of the intersection product, which we denote $[\mathscr F(m)]$, as a sum:
\[
[\mathscr F(m)] = [\mathscr S(m)]+[\mathsf{excess}].
\]
The class $[\mathsf{excess}]$ is pushed forward from the pullback of $\mathbb P_V$ to $\mathscr S(m-1)$. Each is a sum of terms involving the following ingredients; the precise formula does not matter to us:
\begin{enumerate}[(i)]
\item The pullback of the projective bundle to a stratum of $\mathscr S(m-1)$ that maps into the center of $\mathsf Z(m-1)$. 
\item Chern classes of the line bundles associated to the boundary divisors of $\mathscr S(m-1)$.
\item The Chern class of the hyperplane line bundle of $\mathbb P_V\to V$. 
\end{enumerate}
Momentarily ignoring the $T$-class, we see that since any class on $\mathbb P_V$ is the pullback of a class on $V$, capped with some power of the relative hyperplane bundle, some bookwork allows us to conclude that the integral can be recovered from a strata invariant associated to $\mathscr S(m-1)$, up to correction terms that are strata GW cycles on deeper strata, all of which still live on $\mathscr M(m-1)$. 

To include the $T$-class, observe that $\mathscr S(m)\to\mathscr S(m-1)$ is an lci proper birational morphism, and push/pull shows that the $T$-class is pulled back, up to terms supported over the exceptional locus of this map. The correction terms relating the original $T$-class with the push/pull is therefore supported on deeper strata. Inducting on the depth of the strata (for each fixed number $m$ of modifications to the evaluation space). We conclude that when $\mathscr S(m-1)$ does not map into the center of the blowup. 

The second case occurs when $\mathscr F(m)\to\mathscr S(m-1)$ is a projective bundle, namely the pullback of $\mathbb P_V\to V$. Now, the space $\mathscr F(m)$ will be smooth and smooth over $\mathscr S(m-1)$, but possibly of the wrong dimension. In this case the map
\[
\mathscr S(m)\to\mathscr S(m-1)
\]
is a broken toric bundle. Again, let us momentarily ignore the $T$-class. Then by the projective bundle formula, the class of $\mathscr S(m)$ pushed forward to $\mathscr F(m)$ is the pullback of $\mathscr S(m-1)$ times some power of the Chern class of the relative hyperplane bundle. Again, the only relevant insertions are pullbacks from $V$ to $\mathbb P_V$, times some power of this same class. A diagram chase express this as a strata integral on $\mathscr S(m-1)$. 

To include the $T$-class, we proceed as above. The $T$-class is pulled back from $\mathscr S(m-1)$ up to deeper strata on $\mathscr S(m)$, and similarly, a second induction on the strata depth moves the integral to $\mathscr S(m-1)$. 

The final step is to deal with the projective bundle $\mathsf Z(1)\to\mathsf Z(0)$. The analysis is actually identical to the final step above. We have $\mathscr S(1)$ which maps onto some $\mathscr S(0)$. The pullback of $\mathscr S(0)$ under $\mathsf Z(1)\to\mathsf Z(0)$ has two possibilities, depending on whether $\mathscr S(0)$ maps into the center or not. In the former case, we analyze it by the weighted projective bundle formula, and in the latter case, the weighted blowup formula~\cite{AOA}.\footnote{As we work rationally, after appropriate root stacks, this is just the ordinary blowup formula~\cite[Section~3.6]{NR19}.}

The outcome is an inductive procedure that express a strata intersection on $\mathscr M^\diamond$ with a $T^\diamond$-class and an evaluation class with a series of strata classes on $\mathscr M$, with evaluations from $\mathsf{Ev}$ and $T$-class. Now using the fact that $\mathsf M\to\mathscr M$ and $\mathsf M^\diamond\to\mathscr M^\diamond$ has a perfect relative obstruction theory. Apply virtual pull back to the formula to conclude the theorem statement.

The proofs of the proposition and the theorem are complete.
\end{proof}

\section{Rigidification and strata geometry}\label{sec: rigidification}

\subsection{Overview and the ideas}  Let's recap where we are in the big picture. We have converted GW cycles on the general fiber into exotic invariants on bundles over strata of the special fiber, and then into non-exotic strata GW cycles on these moduli spaces. 

We would like to now say that the strata GW cycles can be decomposed into contributions associated with the vertices of that complex. A basic technical issue arises -- the degeneration formula describes the virtual irreducible components of degenerate moduli spaces, rather than virtual strata of higher codimension. The results of~\cite{MR23} do not directly apply\footnote{The formalism of punctured maps includes a strata statement~\cite{ACGS17}, but does not provide a splitting result sufficient for our purposes.}. The geometric issue is that descriptions of strata typically involve maps to {\it non-rigid} targets, sometimes called {\it rubber} geometries. 

In the smooth pair case, a rigidification/rubber calculus was developed in~\cite{MP06}. In the general simple normal crossings setup, the rubber theory is significantly more complicated; as explained by Carocci--Nabijou, the relevant rubber automorphism groups act non-locally~\cite{CN21}. A description of strata analogous to the degeneration formula of~\cite{MR23} has not yet been achieved. 

We offer a different way out, via a simple geometric idea. Given a stratum of $\Mbar_\Lambda(Y|\partial Y)$, we degenerate the entire moduli space to the normal cone of that stratum. The key point is that this component is virtually birational to a component in the degeneration problem; this can be done constructively and explicitly. We then provide a calculus that converts a GW strata class into an integral on this compactified normal cone. This allows us to invoke the results of~\cite{MR23}. Together with an analysis of the cohomology class coming from the stack $\mathsf a T$, this implies the reconstruction\footnote{For readers familiar with~\cite{MP06}, an important idea is to reverse the logical order of operations -- there, a stratum is split first and then rigidified. Here, we rigidify first, and then split; this allows us to skirt around the non-locality of action of the torus. }. 

\subsection{Discrete data and genus decorations} Fix a moduli space $\Mbar_\Lambda(Y|\partial Y)$ of logarithmic stable maps. As explained in the previous section, we make a subdivision so there is a morphism
\[
\Mbar_\Lambda(Y|\partial Y)\to \mathsf{a} T^{\bm \mu}_{\sf v,k}(Y|\partial Y).
\]
We consider non-exotic stratum GW invariants, coming from the previous section. These are specified by a cone $\Theta$ of the moduli space $T^{\bm \mu}_{\sf v,k}(Y|\partial Y)$, and determine a stratum
\[
S_\Theta\subset \Mbar_\Lambda(Y|\partial Y).
\]
We are interested in classes formed by capping $[S_\Theta]^{\sf vir}$ with non-exotic evaluation classes and $T$-classes. 

The cone $\Theta$ gives rise to a combinatorial type of Chow $1$-complex $\Gamma$. However, if we specify the genus at each vertex of $\Gamma$, we obtain a refinement of the stratum:
\[
S_\Theta = \bigcup_{\hat \Gamma} S_{\hat\Gamma}, \ \ \ [S_\Theta]^{\sf vir} = \sum_{\hat \Gamma} [S_{\hat\Gamma}]^{\sf vir}.
\]
Our goal is to control the classes derived from $S_\Theta$, but during the induction, it will be useful to prove a more refined statement about these $S_{\hat \Gamma}$ spaces. We will explain, as we go, where the genus aspect of the induction is useful. 

\subsection{Strata reconstruction theorem} To recap the notation, we have a cone $\Theta$ in the $T$-space, a Chow $1$-complex type $\Gamma$, and when we enhance this with genus distributions, we denote this $\hat \Gamma$. 

The stars of vertices of $\Gamma$, together with the genus distribution on $\hat \Gamma$, determine partial data for a logarithmic stable maps problem. As recorded in Remark~\ref{rem: stars-discrete-data}, if the vertex of a star is in a cone $\sigma$ in $\Sigma$, then the associated problem is one on broken toric bundle of rank $\dim \sigma$, over the closed stratum $W_\sigma$ in $X$. We note that the tangency at these stars is not fully specified, but there are only finitely many spaces, so we can just take the disjoint of the spaces taken over all tangencies. 

Just as the data $\Gamma$ determines a stratum of $\Mbar_\Lambda(Y|\partial Y)$ the additional decoration $\hat \Gamma$ also describes a stratum of this space. The notion of an exotic stratum GW cycle is defined in the same sense as before -- the $T$-class, which was denoted $\vartheta$, is still pulled back from $\mathfrak aT$ just as in the previous section. 

We now state the main theorem. 

\begin{theorem}\label{thm: rigid-split}
Let $\mathsf S_\Theta$ be a stratum of $\Mbar_\Lambda(Y|\partial Y)$ corresponding to a combinatorial type of Chow $1$-complex, with genus distribution, together denoted $\hat \Gamma$. Any strongly non-exotic stratum GW cycle associated to $S_\Theta$ can be effectively reconstructed from non-exotic logarithmic GW cycles associated with (i) stars of vertices of $\Gamma$, (ii) stars of strictly lower order, and (iii) stars with strictly smaller genus than $\Lambda$. 
\end{theorem}

The plan for the proof is as follows:
\begin{enumerate}[(i)]
\item We construct a target degeneration such that the associated degeneration of the moduli space has a component that maps to $\mathsf S_\Theta$ as a ``virtual fibration'';
\item we then construct an insertion on this component that allows us to transfer the integral on the stratum $\mathsf S_\Theta$ to this rigidification;
\item we use intersection theory on the Artin fan, generalizing the ``rubber calculus'' argument of~\cite{MP06}, to remove the $T$-class above. This puts us exactly in the situation of~\cite{MR23};
\item finally, we carry out an inductive proof of the theorem by splitting the strata $S_{\hat \Gamma}$, re-running the exotic-to-stratum conversion, splitting again, and so on.
\end{enumerate}

\begin{remark}[Dependence of the steps on the genus parameter]
Step (i) will be done uniformly in the genus parameter, in other words, we will construct a rigidified space $S_\Theta^{\sf rig}$ that virtually fibers over $S_\Theta$, and analogous rigidifications of the genus-decorated substrata are obtained by restricting to virtual irreducible components of $S_\Theta^{\sf rig}$. In Step (ii) we provide two constructions, one of which is genus-dependent, in the sense that it is a different construction for each distribution, and the other is genus independent. Step (iii) is again genus independent, and step (iv) is genus dependent.     
\end{remark}

\subsection{A degeneration of the moduli space}\label{sec: degeneration-of-strata} We construct the rigidifying spaces by a well-chosen degeneration to the normal cone of the moduli space. This is best explained using subdivisions.

\subsubsection{A reminder on subdivisions} Two types of subdivisions, also known as logarithmic modifications, naturally occur in logarithmic Gromov--Witten theory: subdivisions of the target, and subdivisions of the moduli space. A subdivision of the target naturally forces a subdivision of the moduli space. Let us explain how.

Consider a space of stable $1$-complexes in $\Sigma$ with a chosen cone structure.\footnote{The numerical data does not need to be fixed for this discussion}. We have a moduli diagram of cone spaces:
\[
\begin{tikzcd}
\mathbb G \arrow{dr} \arrow{rr}& & \Upsilon\arrow{dl}\arrow{r} & \Sigma\\
&\mathsf T_{\mathsf v,k}^{\bm\mu}(Y|\partial Y).& & 
\end{tikzcd}
\]
The family $\mathbb G$ is the family of Chow $1$-complexes, while the family $\Upsilon$ is the natural family of complete subdivisions of $\Sigma$ by polyhedral complexes; the first arrow in the top arrow is, fiberwise, an inclusion of a $1$-dimensional polyhedral subcomplex.

Now consider a subdivision $\Sigma'\to \Sigma$. We can pull back this subdivision to $\Upsilon$, and to $\mathbb G$, which is just a subcomplex of $\Upsilon$. We can resolve singularities in $\Upsilon$ as well\footnote{The precise requirement is only that $\mathbb G$ is smooth as a cone complex.}. The resulting map
\[
\Upsilon'\to \mathsf T_{\mathsf v,k}^{\bm\mu}(Y|\partial Y)
\]
may no longer be combinatorially flat. We can apply semistable reduction to obtain a new diagram
\[
\begin{tikzcd}
\mathbb G' \arrow{dr} \arrow{rr}& & \Upsilon''\arrow{dl}\arrow{r} & \Sigma'\\
&\mathsf T_{\mathsf v,k}^{\bm\mu}(Y|\partial Y).& & 
\end{tikzcd}
\]
In this way, a subdivision of the target induces a subdivision of the moduli space of Chow $1$-complexes. By pulling this back along the stabilization map
\[
\mathcal T_{\mathsf v,k}^{\bm\mu}(Y|\partial Y)\to \mathsf T_{\mathsf v,k}^{\bm\mu}(Y|\partial Y)
\]
we get a subdivision of $\Mbar_\Lambda(Y|\partial Y)$. 

The key trick is to use this idea above, but to apply it to $Y\times\mathbb A^1$, with logarithmic structure given by $\partial Y\times\A^1\cup Y\times 0$.

\subsubsection{The rigidified moduli space} The construction will be performed with respect to a choice of $r$ in the interior of the cone $\Theta$. It corresponds to a Chow $1$-complex
\[
\Gamma\to \Sigma. 
\]
We therefore get an associated simple normal crossings degeneration
\[
\mathcal Y_r\to \mathbb A^1,
\]
by the correspondence between polyhedral subdivisions and degenerations, see~\cite[Section~1,2]{MR23}. There is an associated moduli space of logarithmic curves
\[
\mathsf M(\mathcal Y_r/\mathbb A^1)\to \mathbb A^1.
\]
Taking the natural lift of the numerical data, we identify that the Chow $1$-complex $\Gamma$ with a rigid Chow $1$-complex for the degeneration problem above, and we isolate the component
\[
S_\Theta^{\sf rig}\hookrightarrow \mathsf M(\mathcal Y_r/\mathbb A^1)_0.
\]
This component carries a virtual fundamental class. 

The space $S_\Theta^{\sf rig}$ can be understood as follows. We have earlier used notation akin to $T(Y|\partial Y)$, but this will quickly become burdensome so we drop $\partial Y$ and the like from the notation. Take the space $T(Y\times\mathbb A^1)$ and perform the blowup
\[
\mathcal Y_r\to Y\times\mathbb A^1. 
\]
This gives rise to an induced subdivision $T(\mathcal Y_r)\to T(Y\times\mathbb A^1)$ and of the moduli space. The space $S_\Theta^{\sf rig}$ is the preimage in $\mathsf M(\mathcal Y_r/\mathbb A^1)$ of the exceptional ray in  $T(\mathcal Y_r)\to T(Y\times\mathbb A^1)$. By construction this ray maps to $\sigma$, sending the generator of the ray to $p$. 

In particular, we get a morphism
\[
S_\Theta^{\sf rig}\to S_\Theta.
\]
This morphism is the pullback of an exceptional divisor in a blowup of Artin fans, so automatically is an equivariant compactification of a torus bundle. 

The following result is not strictly speaking necessary, but gives a sense of what is going on. Given a point $p\in\sigma$, we can perform that ``$r$-weighted'' deformation to the normal cone. Namely, take the cone space
\[
T(Y\times\mathbb A^1) = T(Y)\times\mathbb R_{\geq 0}. 
\]
and consider the point $(r,1)\in\sigma\times\mathbb R_{\geq 0}$. Perform stellar subdivision with respect to this point, adding the ray through this point. See Figure~\ref{fig: p-weighted-deformation} for an example when $\sigma$ is a point. 

\begin{figure}

\tikzset{every picture/.style={line width=0.75pt}} 

\begin{tikzpicture}[x=0.75pt,y=0.75pt,yscale=-1,xscale=1]

\draw  [draw opacity=0][fill={rgb, 255:red, 223; green, 210; blue, 255 }  ,fill opacity=0.4 ][dash pattern={on 0.84pt off 2.51pt}] (416.5,1011) -- (599.5,1011) -- (599.5,1124.5) -- (416.5,1124.5) -- cycle ; \draw  [dash pattern={on 0.84pt off 2.51pt}] (416.5,1011) -- (416.5,1124.5)(436.5,1011) -- (436.5,1124.5)(456.5,1011) -- (456.5,1124.5)(476.5,1011) -- (476.5,1124.5)(496.5,1011) -- (496.5,1124.5)(516.5,1011) -- (516.5,1124.5)(536.5,1011) -- (536.5,1124.5)(556.5,1011) -- (556.5,1124.5)(576.5,1011) -- (576.5,1124.5)(596.5,1011) -- (596.5,1124.5) ; \draw  [dash pattern={on 0.84pt off 2.51pt}] (416.5,1011) -- (599.5,1011)(416.5,1031) -- (599.5,1031)(416.5,1051) -- (599.5,1051)(416.5,1071) -- (599.5,1071)(416.5,1091) -- (599.5,1091)(416.5,1111) -- (599.5,1111) ; \draw  [dash pattern={on 0.84pt off 2.51pt}]  ;
\draw    (417,1124.5) -- (416.51,1011.5) ;
\draw [shift={(416.5,1008.5)}, rotate = 89.75] [fill={rgb, 255:red, 0; green, 0; blue, 0 }  ][line width=0.08]  [draw opacity=0] (8.93,-4.29) -- (0,0) -- (8.93,4.29) -- cycle    ;
\draw [shift={(417,1124.5)}, rotate = 269.75] [color={rgb, 255:red, 0; green, 0; blue, 0 }  ][fill={rgb, 255:red, 0; green, 0; blue, 0 }  ][line width=0.75]      (0, 0) circle [x radius= 3.35, y radius= 3.35]   ;
\draw    (526,1124.5) -- (525.51,1011.5) ;
\draw [shift={(525.5,1008.5)}, rotate = 89.75] [fill={rgb, 255:red, 0; green, 0; blue, 0 }  ][line width=0.08]  [draw opacity=0] (8.93,-4.29) -- (0,0) -- (8.93,4.29) -- cycle    ;
\draw [shift={(526,1124.5)}, rotate = 269.75] [color={rgb, 255:red, 0; green, 0; blue, 0 }  ][fill={rgb, 255:red, 0; green, 0; blue, 0 }  ][line width=0.75]      (0, 0) circle [x radius= 3.35, y radius= 3.35]   ;
\draw    (417,1124.5) -- (605.5,1124.5) ;
\draw [shift={(608.5,1124.5)}, rotate = 180] [fill={rgb, 255:red, 0; green, 0; blue, 0 }  ][line width=0.08]  [draw opacity=0] (8.93,-4.29) -- (0,0) -- (8.93,4.29) -- cycle    ;
\draw [shift={(417,1124.5)}, rotate = 0] [color={rgb, 255:red, 0; green, 0; blue, 0 }  ][fill={rgb, 255:red, 0; green, 0; blue, 0 }  ][line width=0.75]      (0, 0) circle [x radius= 3.35, y radius= 3.35]   ;
\draw    (417,1124.5) -- (525,1071.5) ;
\draw [shift={(525,1071.5)}, rotate = 333.86] [color={rgb, 255:red, 0; green, 0; blue, 0 }  ][fill={rgb, 255:red, 0; green, 0; blue, 0 }  ][line width=0.75]      (0, 0) circle [x radius= 3.35, y radius= 3.35]   ;
\draw    (525,1071.5) -- (592.81,1037.83) ;
\draw [shift={(595.5,1036.5)}, rotate = 153.6] [fill={rgb, 255:red, 0; green, 0; blue, 0 }  ][line width=0.08]  [draw opacity=0] (8.93,-4.29) -- (0,0) -- (8.93,4.29) -- cycle    ;

\draw (481,1046.5) node [anchor=north west][inner sep=0.75pt]    {$( r,1)$};
\draw (391,1025) node [anchor=north west][inner sep=0.75pt]    {$\sigma $};
\draw (581,1133) node [anchor=north west][inner sep=0.75pt]    {$\mathbb{R}_{\geq }{}_{0}$};

\end{tikzpicture}
\caption{The subdivision giving rise to the $r$-weighted deformation to the normal cone.}\label{fig: p-weighted-deformation}
\end{figure}
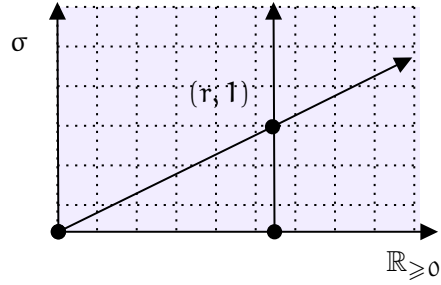

\begin{proposition}
The map $S_\Theta^{\sf rig}\to S_\Theta$ factors through the $r$-weighted deformation to the normal cone of $\mathsf M(X)$. 
\end{proposition}

\begin{proof}
Follows from the universal property of (weighted) blowing up. 
\end{proof}

We also have the following equivalence of virtual classes. This is a formal from the existence of the obstruction theories over the appropriate Artin fans. 

\begin{proposition}
The pull back of the virtual class along $S_\Theta^{\sf rig}\to S_\Theta$ coincides with the natural virtual class on $S_\Theta^{\sf rig}$. 
\end{proposition}

Finally, let us record how the evaluation spaces are related. The cone $\Theta$ in $T(Y)$ determines a cone in $\mathsf P(Y)$, namely just the cone it maps to. This in turn gives us a stratum 
\[
\mathsf Z_\Theta\subset \mathsf {Ev},
\]
which is itself a product of strata 
\[
\mathsf Z_\Theta = \prod_j \mathsf Z_j
\]
where $j$ runs over the marked legs and internal markings of the Chow $1$-complex $\gamma$. Recall that the insertions in our strata GW cycles come from cohomology classes $\upalpha$ on $\mathsf Z_\Theta$.

Next, on the rigid space, we again have an evaluation space, as discussed in~\cite{MR23}. Again we have an evaluation space
\[
\mathsf Z^{\sf rig}_\Theta = \prod_j \mathsf Z_j^{\sf rig}.
\]
where each $\mathsf Z_j^{\sf rig}\to \mathsf Z_j$ is an equivariant compactification of a torus torsor over $\mathsf Z_j$. 

\subsection{Rigidification calculus}\label{sec: rigid-calculus} We now explain how to convert a stratum invariant on $S_\Theta$ into a rigid invariant associated to the component $S_\Theta^{\sf rig}$ of the degeneration constructed above. By basic compatibilities, we have a map:
\[
p\colon S_\Theta^{\sf rig}\to S_\Theta
\]
Our goal is to find a {\it rigidifying insertion} $\eta$, naturally arising from the evaluation maps, such that
\[
p_\star\left([S_\Theta^{\sf rig}]^{\sf vir}\cap \eta \right) = [S_\Theta]^{\sf vir}. 
\]
%
%

The basic idea is to use evaluation conditions in the bundle direction of $Z_\Theta^{\sf rig}\to Z_\theta$ to cut down $S_\Theta^{\sf rig}$ until its virtual class maps to that of $S_\Theta$. This is especially straightforward when each vertex carries a a tangency $0$ marked point. We now explain how to engineer a setup where we can assume this, and what the workarounds are when we do not. 

We work vertex-by-vertex. Suppose $V$ is a vertex of $\Gamma$ such that $\beta_V$ is nonzero. We can add an internal marking at this vertex. By constraining this marked point to a divisor $E\subset X$, we obtain a new GW class. The divisor equation can be used to recover the original stratum cycle, as the two classes differ by a multiplicative factor of $E\cdot\beta$. 

For the purpose of determining GW cycles, this means we can assume that every vertex $V$ with $\beta_V$ nonzero carries an marking of tangency $0$. 

If we have a vertex $V$ with $\beta_V$ equal to $0$, the divisor axiom does not apply in the obvious way. There are two options. 

\noindent
{\bf Genus-dependent option.} For the first, we pass to genus-decorated strata, as described at the strata of the section. In other words, we will describe a rigidification procedure that takes as input a stratum {\it together} with a genus distribution. Given a vertex of this GW stratum, we can add a marked point $q$ with contact order $0$ at that vertex. This gives rise to a proper map on strata of logarithmic moduli spaces of maps:
\[
S_{\hat\Gamma\cup q}\to S_{\hat \Gamma}.
\]
We can pair with a $\psi$-class at this point $q$ and push forward to obtain a non-zero proportionality:
\[
\psi_q\cdot [S_{\hat\Gamma\cup q}]^{\sf vir} \sim_{\mathbb Q} [S_{\hat \Gamma}]^{\sf vir}.
\]
The proportionality constant is the Euler characteristic of the punctured curve corresponding to the vertex. Note that this is nonzero -- the curve class at this vertex is $0$ by hypothesis, so stability ensures this number is nonzero.  

By this construction, we can assume that {\it every} vertex carries a marked point of tangency order $0$, since we can add a point and recover the same GW cycle, either by the construction above or by the divisor equation. 

Taking this option, we can construct the rigidifying insertion as follows. 

\begin{construction}[Rigidifying insertion, first option]
    We utilize strata of the ``intermediate'' spaces $\mathscr M$ introduced in the previous section. Recall this was defined by the diagram:
\[
\begin{tikzcd}
\mathscr M\arrow{d}\arrow{r} & \mathsf{Ev} \arrow{d}\\
\mathsf a T\arrow{r} &\mathsf{a} \mathsf P.
\end{tikzcd}
\]
We then replace the strata of $\mathsf M$ with strata of $\mathscr M$, which we denote $\mathscr S$. We will construct the rigidifying insertion at the level of these intermediate spaces, and then pull them back to the full moduli spaces. 

By the discussion above, we can replace the GW strata class with an equivalent class, but now can assume that every component carries a marked point of tangency order $0$. Taking these evaluation maps together, we have a commutative diagram:
\[
\begin{tikzcd}
\mathscr S_\Theta^{\sf rig}\arrow[swap]{d}{p}\arrow{r}{{\sf e}} & {\mathsf  Z^{\sf rig}_\Theta} \arrow{d}{q} \\
\mathscr S_\Theta \arrow{r}{{\sf ev}} &{\mathsf Z_\Theta}.
\end{tikzcd}
\]
Both vertical arrows are toric compatifications of torus torsors, and the evaluation diagram is equivariant. By following the associated moves on fans, we can describe the torsors explicitly. Let $N_\Theta$ be the cocharacter lattice of the rubber torus, namely $\Theta^{\sf gp}$. The morphism $p$ is an equivariant compactification of a torsor under $N_\Theta\otimes \mathbb G_m$. Similarly, for each vertex $V$ the corresponding component of $\mathcal Y_p$ is generically a torus torsor over a stratum in $Y$. The cocharacter lattice of this torus is denoted $N_V$. The morphism $q$ is generically a torsor under a product of tori, one for each vertex $V$. 

If we pull back $q$ to $\mathscr S_\Theta$, we claim the resulting map from $\mathscr S_\Theta^{\sf rig}$ to this pullback is an embedding of torus torsors, i.e. generically over the base. Indeed, this can be checked at the level of cocharacter lattice. By the assumption that every $\beta$-nonzero vertex has an internal marking, the map on cocharacter lattices induced by $e$ is an injection. 

We can now construct the rigidifying insertion. Choose a Cartier divisor in ${\mathsf Z}_\Theta^{\sf rig}$ that is $q$-relatively ample. By appropriately scaling and taking a self-intersection, we find a class $\eta$ on $\mathscr S_\Theta^{\sf rig}$ represented by a cycle that maps onto the base birationally under $p$. Virtually pull this back to $S_\Theta^{\sf rig}$ to obtain the rigidifying insertion. 

\qed
\end{construction}

\noindent
{\bf Genus-independent option.} We record the second way, which avoids adding a marked point when we do not have access to the divisor axiom. One reason this is sometimes better is the lack of genus-independence in the construction above. We need the following observation. 

\begin{proposition}\label{prop: atleast-trivalent}
Let $\Gamma$ ba a Chow $1$-complex and let $V$ be a vertex. If $\beta_V$ is equal to $0$, then $V$ is at least trivalent and balanced. 
\end{proposition} 

\begin{proof}
We are working in a general logarithmic setting, so $\Gamma$ may not be traditionally balanced, but we have seen it can be balanced by adding a single ray. We know that this balancing ray depends only on $\beta_V$ and is trivial when $\beta_V = 0$. Therefore, we know that $V$ is already balanced. This rules out the possibility of free vertices and $1$-valent vertices. Suppose $V$ is $2$-valent. If the two flags are not contained in the interior of the same cone that contains $V$ in $\Sigma$, then $\beta_V$ cannot be $0$, as these flags correspond to intersections with proper closed strata contained in the closure of the stratum determined by $V$. The only remaining possibility is a linear $2$-valent vertex. We have removed these by hand using the stability condition on Chow $1$-complex. The result follows. 
\end{proof}

What this means for us is the following. Given a vertex $V$ of $\Gamma$, the associated component of the degeneration $\mathcal Y_r$ described above is a torus torsor over a stratum of $(Y|\partial Y)$. The edges incident to $V$ are torsors for {\it quotient} tori, for quotients by $1$-parameter subgroups corresponding to the edge directions. The trivalency and balancing conditions imply that this larger torus injects into the product of the smaller tori at the edges. 

We now construct the rigidifying insertion {\it without} adding a marked point at $\beta$-zero vertices. 

\begin{construction}[The rigidifying insertion, second option]\label{constr: rigidifying-insertion}
As in the previous construction we work with the intermediate spaces $\mathscr M$ and replace the strata of $\mathsf M$ with strata of $\mathscr M$, which we denote $\mathscr S$. 

We use ``nodal evaluations'' for the construction. Consider the Chow $1$-complex $\Gamma$ given by $p$ in $\Theta$. Fix an edge $E$, with multiplicity $m_E$. We pass to the locus the tangency data at this edge is given by an {\it ordered} partition $\mu_E$, as these appear as space-level components in the degeneration formula. The construction will apply to any choice $\mu_E$. For each entry of $\mu_E$, we have an evaluation map to the stratum $Z_E$, which is an irreducible element of the double locus in the special fiber of $\mathcal Y_p$. As in previous situations, this $Z_E$ is an equivariant compactification of a torus bundle over a stratum of $Y$. Similarly, for $S_\Theta$ itself, we an evaluation map to a stratum of $Y$ directly. 

Perform this extra labelling for every $V$ such that $\beta_V = 0$. By augmenting the evaluation maps with additional factors for every edge adjacent to a $\beta$-zero vertex, we end up with an augmented evaluation diagram:
\[
\begin{tikzcd}
\mathscr S_\Theta^{\sf rig}\arrow[swap]{d}{p}\arrow{r}{{\sf e}} &\widetilde {\mathsf  Z^{\sf rig}_\Theta} \arrow{d}{q} \\
\mathscr S_\Theta \arrow{r}{{\sf ev}} &\widetilde{\mathsf Z_\Theta}.
\end{tikzcd}
\]
At $\beta$-nonzero vertices we proceed as in the previous construction. At $\beta$-zero vertices, we use Proposition~\ref{prop: atleast-trivalent} to come to the same conclusion -- the map $\mathscr S_\Theta^{\sf rig}$ to pullback $q$ to $\mathscr S_\Theta$ is an embedding. Tracing through the toric geometry, this follows from the fact that the images of the vertex position of a $\beta$-zero vertex in the projections along the edges adjacent to it determine the position of the vertex. This implies the claim. 

The rigidifying insertion is now constructed in parallel to the previous discussion -- by taking a power of the Chern class of a relative ample divisor, and scaling appropriately. 

\qed
\end{construction}

In any event, we end up with a rigidifying insertion, either on $S_\Theta^{\sf rig}$ or on the genus-decorated space $S_{\hat\Gamma}^{\sf rig}$

\begin{proposition}\label{prop: rigidifying-insertion}
The rigidifying insertion satisfies
\[
p_\star\left([S_{\hat \Gamma}^{\sf rig}]^{\sf vir}\cap \eta \right) = [S_{\hat \Gamma}]^{\sf vir}. 
\]
The analogous statement holds replacing $S_{\hat \Gamma}$ and $S_{\hat \Gamma}^{\sf rig}$ by corresponding spaces $S_\Theta$ and $S_\Theta^{\sf rig}$, provided we use the genus-independent rigidifying insertion above. 
\end{proposition}

\begin{proof}
True by construction. 
\end{proof}

\subsection{Rubber geometry}\label{sec: rubber-calculus} To recap, we have now started off with a stratum cycle of the form
\[
\left(\mathsf{ev}^\star_{\mathsf S}(\upalpha)\cup \mathsf t^\star(\vartheta)\right)\cap [\mathsf S_\Theta]^{\sf vir}
\]
and converted it into an integral of the form
\[
\left ( \mathsf{ev}^\star_{\mathsf S}(\upalpha)\cup \mathsf t^\star(\vartheta)\cup \eta\right)\cap [\mathsf S_\Theta^{\sf rig}]^{\sf vir}.
\]
The $T$-class in the second integral is pulled back from the first. To analyze it, observe that the cone $\Theta$ determines a stratum
\[
\mathsf B_\Theta\subset \mathfrak aT(Y),
\]
which is typically a negative dimensional Artin stack. It has generic stabilizer of dimension equal to the dimension of $\Theta$. However, the composite morphism
\[
\mathsf S_\Theta^{\sf rig}\to \mathsf S_\Theta\to \mathsf B_\Theta
\]
factorizes through an Artin fan, i.e. a zero-dimensional Artin stack. Specifically, take the relative Artin fan $T(\mathcal Y_r)/\mathbb A^1$ and take the component over $0$ corresponding to $p$. Denote this Artin fan $\mathsf B_\Theta^{\sf rig}$. We have a factorization
\[
\mathsf S_\Theta^{\sf rig}\to \mathsf B_\Theta^{\sf rig} \to \mathsf B_\Theta.
\]
Now, given any class $\vartheta$ on $\mathsf B_\Theta$, its pullback to $\mathsf B_\Theta^{\sf rig}$ is necessarily supported on boundary strata (unless $\vartheta$ is the fundamental class). This is a new stratum invariant. However, it is one whose $T$-class has lower cohomological degree. 

The rigidification procedure only introduces evaluation classes, and never $T$-classes, so it follows that if we iterate the procedure, we obtain a finite number of invariants associated to components of degenerations, none of them carry a nontrivial $T$-class, and a linear combination of them yields the original one.

\subsection{A remark on the structure of the induction}\label{sec: translated-stars} We will shortly carry out an inductive proof of Theorem~\ref{thm: rigid-split} based on the ordering on tropical curves constructed earlier, as well as a genus induction. We record some preparatory comments.

Fix a target $(Y|\partial Y)$ and let $\Sigma$ denote its cone complex. The goal is to show that exotic GW cycles for $(Y|\partial Y)$ for a given discrete data are determined by non-exotic GW classes for $(Y|\partial Y)$ and the non-exotic invariants of strata, determined by stars less than or equal to a fixed star. 

The nature of the proof will produce inductively smaller problems, smaller in the ordering by stars. The high level approach is: (i) start with a star $\sf v$ and an exotic insertion associated to it; (ii) argue that the exotic GW class differs from a non-exotic one by non-exotic GW classes of strata, each attached to some Chow $1$-complex; (iii) apply the rigidification and rubber calculus to reduce to a finite collection of new stars, that are smaller than $\sf v$. 

In (iii) the insertions on the star coming from the degeneration formula are typically exotic. 

Let us flag something. Say we have started with $\sf v$ above, ended up with a non-exotic stratum invariant with $1$-complex $\Gamma$, and isolated a new star $\sf w$, whose base point lies in some cone of $\Sigma$. The next step in the induction is to reapply the procedure with $\sf w$.

When we apply the argument to $\mathsf w$, we will end up with a new Chow $1$-complex $\Gamma'$ -- or more correctly a cone of tropical curves with the same type as $\Gamma'$ -- to analyze. Going from this cone to a single Chow $1$-complex is a combinatorial choice in the rigidification above; it is the choice of point in the interior of the cone $\Theta$ in the preceding discussion. 

However, the ordering that we have discussed has to do with stars in the cone complex $\Sigma$. In order to appeal to the ordering, one needs to view the vertices of $\Gamma'$ as giving stars in here. When doing this, it is entirely possible that the star $\mathsf u$ is equivalent to the star $\mathsf w$ (recall that stars in a cone are identified if they differ by translation of the base point). 

\subsubsection{Translated stars} Let us explain this in an example. Suppose that we are at the second level in the induction -- we start with $\mathsf v$, and then chose to analyze the star $\mathsf w$ at a vertex of $\Gamma$. Suppose further that the base of $\mathsf w$ is a higher dimension cone of $\Sigma$. This means that the target determined by $\mathsf w$ after rigidification is a broken toric bundle $(E|\partial E)\to (B|\partial B)$, over a stratum $B$ of $Y$. 

The fan of $(E|\partial E)$ is a refinement
\[
\Sigma_{E|\partial E}\to \mathbb R^k\times\Sigma_{B|\partial B}
\]
where $k$ is the generic fiber dimension of $E\to B$. Now let $\sf w$ be a star based at the origin in $\Sigma_{E|\partial E}$. A {\it translated star} $\sf w'$ is obtained by translating $\sf w$ by the $\RR^k$ action on the first factor above. 

We can also identify $\Sigma_{E|\partial E}$ with the star neighborhood in $\Sigma$ of the base of the star $\mathsf w$. Under this identification, observe that the stars $\mathsf w$ and $\mathsf w'$ differ by translation only. 

The caution is that $\sf w$ and $\sf w'$ are equivalent stars in $\Sigma$, but if viewed in $\Sigma_{E|\partial E}$, we might conclude that $\sf w'$ is smaller than $\sf w$. But this latter ordering is not valid for us -- the ordering concerns only the stars the fixed cone complex $\Sigma_{Y|\partial Y}$. See Figure~\ref{fig: translated-star}.

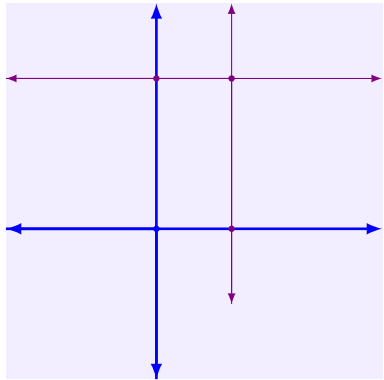
\begin{figure}[h!]

\begin{tikzpicture}
  \definecolor{softlavender}{RGB}{223,210,255}
  \colorlet{bluecol}{blue}
  \colorlet{violetcol}{violet}

  \fill[softlavender, opacity=0.4] (-2,-2) rectangle (3,3);

  \tikzset{>=latex}

  \draw[bluecol, line width=1pt, ->] (-2,0) -- (3,0);
  \draw[bluecol, line width=1pt, ->] (0,-2) -- (0,3);
  \draw[bluecol, line width=1pt, <-] (-2,0) -- (0,0);
  \draw[bluecol, line width=1pt, <-] (0,-2) -- (0,0);

  \fill[bluecol] (0,0) circle (1.2pt);
  \fill[violetcol] (0,2) circle (1.2pt);
  \fill[violetcol] (1,0) circle (1.2pt);

  \begin{scope}[shift={(1,2)}]
    \draw[violetcol,->] (-3,0) -- (2,0);
    \draw[violetcol,->] (0,-3) -- (0,1);
    \draw[violetcol,<-] (-3,0) -- (0,0);
    \draw[violetcol,<-] (0,-3) -- (0,0);
    \fill[violetcol] (0,0) circle (1.2pt);
  \end{scope}
\end{tikzpicture}

\caption{A translated star obtained by translating the star in blue. After rigidification, the translated star becomes the star that we began with.}\label{fig: translated-star}
\end{figure}

However, we note two important points for the coming proof:
\begin{enumerate}[(i)]
\item The rigidifying insertion constructed in this section can always be chosen to be a non-exotic class. This follows immediately from the construction. 

\item When the genus at the secondary vertices are equal to $0$, the degeneration formula applied to translated stars takes a simplified form.
\end{enumerate}

We record the point (ii) precise for later use. Fix a translated star $\Gamma$ as above. Let $\sf M_1$ be the moduli space associated to the main vertex and let $\sf M_2$ be the product of all moduli associated to the secondary vertices. 

As in Section~\ref{sec: degeneration-of-strata}, the choice of $\Gamma$ determines a target degeneration, and partial discrete data associated to $\Gamma$. As noted previously, we can further specify the genus decoration. 

Given such a genus distribution, we obtain a moduli space, let us call it $\mathsf M$, to which the degeneration formula applies. 

\begin{lemma}\label{lem: uncorrected-translated}
     Fix $\Gamma$ as above. Assume that every connected component of every curve corresponding to a secondary vertex is $0$. Then the evaluation map
    \[
    \mathsf M_2\to \mathsf E
    \]
    is combinatorially flat.
\end{lemma}

\begin{proof}
The condition is purely tropical. By the genus $0$ assumption and the fact that translated stars are trees, it follows that the map $\mathsf M_2\to \mathsf E$ is an {\it external} product -- it is a product of evaluation maps, one for each connected component corresponding each vertex. Therefore it suffices to consider the case where there is a single secondary vertex and a single connecting edge to the main vertex. By the balancing condition, every tropical curve corresponding to the secondary vertex is contained in a line parallel to the span of the star of this vertex.  Every such tropical curve has exactly one end corresponding to connecting edge; the tropical evaluation map is given by taking any point on this end, and looking at its image under the quotient by this line distinguished line. It follows easily that if this projection hits any point in the interior of the cone of the evaluation divisor, it can be translated to hit every such point. This is precisely the required condition.
\end{proof}

To see why this lemma simplifies the degeneration formula, observe that we usually deduce from the degeenration formula, a proportionality:
\[
[\sf M]^{\sf vir} \sim_{\mathbb Q} [\sf M^\diamond_1\times_{E^\diamond} \sf M^\diamond_2]^{\sf vir},
\]
say after pushforward to the homology of $\sf M_1\times_E \sf M_2$, for sufficiently fine blowups of $\mathsf E$. By the lemma, for translated stars, where all the genus is concentrated at the main vertex, we obtain the stronger proporitionality
\[
[\sf M]^{\sf vir} \sim_{\mathbb Q} [\sf M_1\times_{E} \sf M _2]^{\sf vir},
\]
It follows that the degeneration formula holds without needing to blowup the evaluation space. 

In particular, a rigidified stratum GW cycle with non-exotic insertions can immediately be converted into a non-exotic GW cycle on the main vertex.

\subsection{Proof of Theorem~\ref{thm: rigid-split}} We proceed by double induction, on the complexity of stars and the genus parameter. 

Fix $\Theta$ as above, and a substratum 
\[
S_{\hat \Gamma}\subset S_\Theta
\]
given by fixing a genus distrubution. 

Start with a stratum class 
\[
(\mathsf{ev}^\star_{\mathsf S}(\upalpha)\cup \mathsf t^\star(\vartheta))\cap [\mathsf S_{\hat \Gamma}]^{\sf vir}.
\]
Pass to the rigidified space $\mathsf S_{\hat \Gamma}^{\sf rig}$. By the discussion on rigidification, we can add $0$-tangency markings where needed and then choose a rigidifying insertion $\eta$ using Proposition~\ref{prop: rigidifying-insertion}. Apply the rubber calculus of Section~\ref{sec: rubber-calculus} to terms that lie over deeper strata in $T(Y)$, so by inducting on the depth of we can assume that the class $\vartheta$ is equal to $1$. Now apply the logarithmic degeneration formula~\cite[Section~8]{MR23}. 

We examine the constituent terms in the degeneration formula. The terms are indexed by stable Chow $1$-complexes with genus decorations. Every vertex in the Chow $1$-complex described above defines a GW cycle associated to an equivariant compactification of a torus bundle over a stratum of $(Y|\partial Y)$. The insertion on the target associated to the star has two parts: one coming from coming from ${\sf ev}_S^\star(\alpha)$, which remains is non-exotic, and a new set of insertions coming from the  K\"unneth splitting, which produces {\it new} exotic insertions at this vertex at marked points corresponding to gluing nodes in the degeneration. 

We now want to induct, arguing that everything that is exotic either gets smaller according to the order in Section~\ref{sec: stars-partial-ordering}, or when it stays the same, that the genus drops. Repeatedly applying the procedure, we will conclude. 

To explain this precisely, first note that, after rigidification, a vertex $V$ of a Chow $1$-complex determines a broken toric bundle
\[
E\to B
\]
where $B$ is a stratum of $(Y|\partial Y)$. Up to logarithmic birational maps, there are finitely many targets that appear as targets associated to vertices -- depending on the choice of stratum $B$. The choice of logarithmic birational model is immaterial to us. We refer to the codimension of $B$ as as the {\it depth} of $E$ over $(Y|\partial Y)$. Combinatorially, it is the dimension of the cone in $\Sigma_{Y|\partial Y}$ that contains $V$. 

Let us now spell out the induction, examining the possibilities for the vertex $V$. 

If $V$ is the origin of $\Sigma_{Y|\partial Y}$, then by the ordering in Section~\ref{sec: stars-partial-ordering} and noting in light of Section~\ref{sec: translated-stars} that there is no nontrivial translation equivalence here, the star at this vertex is smaller than the star that we started with. The contribution in the degeneration formula is an exotic GW class associated to this smaller star. 

If $V$ lies in a higher dimensional cone, i.e. if $B$ has positive codimension, then the issues discussed in Section~\ref{sec: translated-stars} appear. After rigidifying $\Gamma$, this vertex gives rise to a star for a GW class in a broken toric bundle $E\to B$, and new exotic insertions coming from the splitting. Call this star $\sf w$. 

We carry out the rigidification and splitting procedure again, this time for the exotic GW class with star $\sf w$, on the target $(E|\partial E)$. Note that the curve class on $E$ is specified by its pushforward $B$ and the intersection numbers with respect to the horizontal divisors of $E\to B$. In our case, both of these are determined by the Chow $1$-complex that we started with and the choice of $V$ on it. Call the curve class $\beta_{\sf w}$

The outcome of running the exotic-to-stratum algorithm for the exotic invariant with star $\sf w$ is that we replace $\sf w$ with a Chow $1$-complex $\Gamma'$ in $\Sigma_{E|\partial E}$, and study the stratum GW class on the corresponding stratum of the mapping space to $E$, with a {\it non-exotic} insertion. We then run the rigidification and splitting and focus our attention on another star $\sf u$. Again, all the exotic insertions at this star come from the K\"unneth splitting. 

The star $\sf u$ of this form could have higher depth over $Y$ than $\sf w$, or they lie over the origin in the map
\[
\Sigma_{E|\partial E}\to \Sigma_{B|\partial B}.
\]
We analyze these two possibilities separately. 

\noindent
{\sc Case I.} If the curve class $\beta_{\sf w}$ is not concentrated at $\sf u$, then by induction on the curve class, we can conclude. If the entire curve class lies on this single vertex $V'$ of $\Gamma'$, then after adding the $\beta_{\sf w}$-balancing ray to this vertex, the $1$-complex $\Gamma'$ becomes balanced in the traditional sense. Adding this ray to $\sf w$ gives a star that is balanced in the traditional sense. Therefore, unless the star at $V'$ of $\Gamma'$ is equivalent to $\mathsf w$, then $\mathsf u$ is strictly smaller than $\mathsf w$ and we can conclude. 

In the case where $\mathsf w$ and $\mathsf u$ are equivalent as stars in $\Sigma_{Y|\partial Y}$, we can also conclude. Indeed, by saw the genus parameter induction, we can assume that the full genus is concentrated at $\mathsf w$. In this case the K\"unneth splitting in the degeneration formula does not produce exotic insertions, as noted in the preceding discussion. 

\noindent
{\sc Case II.} In the second case, the star ${\sf u}$ in $\Gamma'$ lies over a larger cone in $\Sigma_{B|\partial B}$. We now repeat the whole procedure again, noting that the depth has increased. Since the depth eventually runs out, by induction again we conclude. Geometrically, the target eventually becomes a bundle over stratum in $(Y|\partial Y)$ that, viewed as a pair on its own, has no boundary divisor. 

The two cases exhaust the possibilities. The proof is complete. 

\qed

\section{The universal projective line}\label{sec: uni-P1-bundle}

\subsection{Status} Our overall goal is to convert GW cycles of a general fiber of $\cX\to B$ to invariants of the strata of the special fiber without any logarithmic structure involved. By Theorem~\ref{thm: primitive-in-exotic} and the degeneration formula, we can express these in terms of GW cycles of toric bundles over strata of the special fiber, i.e. components of expansions. By Theorem~\ref{thm: exotic-to-strata} and Theorem~\ref{thm: rigid-split} we turn this into ordinary logarithmic GW invariants. 

Two steps remain: convert GW cycles of these toric bundles over strata into cycles on the base and then convert logarithmic GW cycles of a pair to ordinary GW cycles of the strata. 

\subsection{The provenance of $\mathbb P^1$-bundles} The two issues steps that remain are closely related, and $\mathbb P^1$-bundles are the key to both. 
\begin{enumerate}[(i)]
    \item Any toric bundle $V\to W$ over a stratum $W$ is logarithmically birational to a fiber product of $\mathbb P^1$-bundles, where the logarithmic structures are along $0$ and $\infty$. 
    \item The way to add logarithmic structure from a pair $(Y|\partial Y)$, say along $D$, is to degenerate to the normal cone of $D$. The degeneration formula links these two logarithmic structures:
    \[
    (Y|\partial Y)\leadsto (Y|\partial Y+D) \star (\mathbb P_D|\partial \PP_D).
    \] 
     The pair $(\mathbb P_D|\partial \PP_D)$ is the projective normal cone over $D$, a $\mathbb P^1$-bundle, with logarithmic structure along one of the sections and the divisors induced along $D$ by $\partial Y$. By the degeneration formula, the theory of $\mathbb P_D$ can be viewed as a prescribing a transformation rule between the two theories. Our goal is to {\it invert} this transformation. 
\end{enumerate}

Both steps boil down to an understanding of $\mathbb P^1$-bundles over a base with logarithmic structure, where the bundle has $0$ and/or $\infty$ sections as part of its logarithmic structure. 

Our next goal is to prove that GW cycles of a projectivization of a split rank $2$ vector bundle over a base can be reconstructed from the GW cycles of the base. For brevity, we will use the phrase $\mathbb P^1$-bundle to mean projectivization of a split rank $2$ vector bundle. The base is allowed to have logarithmic structure and the bundle can also have logarithmic structure, coming from one or both of the $0$ and $\infty$ sections. 

The key difficulty in achieving this goal is the presence of logarithmic structure on the base. This removes access to relative virtual localization~\cite{GV05}, which is the key technique used to prove the analogous result when the base has trivial logarithmic structure~\cite{MP06}.

\noindent
{\bf Temporary suspension of disjointness.} In this section and the next, we drop the disjointness condition on the contact order matrix $[c_{ij}]$. We will be interested in $\mathbb P^1$-bundles over base logarithmic targets $\mathbb P\to Y$. The disjointness and $\mathbb P^1$-bundle properties are in tension with each other -- if we want disjoint contact orders we must allow $\PP\to Y$ to degenerate over the boundary, but if we want a genuine $\mathbb P^1$-bundle, then the contact order may not be disjoint. 

The key result of the next two sections is the following.

\begin{theorem}\label{thm: P1-bundle-formula}
Let $\mathbb P\to Y$ be a $\mathbb P^1$-bundle with the relative structure given by the preimage of $\partial Y$, possibly together with one or both of the $0$ and $\infty$ sections. The Gromov--Witten cycles of $\mathbb P$ lie in the span of the Gromov--Witten cycles of $(Y|\partial Y)$. 
\end{theorem}

To emphasize, when we say GW cycle of $\mathbb P$ here, the insertion space is the one discussed in Section~\ref{sec: stable-maps-to-pairs}. In particular, all insertions will can be written as pullbacks of classes on strata of $(Y|\partial Y)$ along the restriction of the $\mathbb P^1$-bundle, times the Chern class of the relative $\mathcal O(1)$. 

We would also like to treat insertions from targets of the form $\widetilde\PP\to\PP\to Y$, where the first map is a logarithmic blowup. We do this in Section~\ref{sec: compressed-evaluations} using rigidification and rubber calculus arguments, similar to the exotic/non-exotic moves used in the proofs of Theorem~\ref{thm: exotic-to-strata} and Theorem~\ref{thm: rigid-split}.  The final results are stated for broken toric bundles in Section~\ref{sec: broken-bundles}.

The proof of Theorem~\ref{thm: P1-bundle-formula} will do something more refined. The morphism $\mathbb P\to Y$ can be used to produce GW classes in the homology of the space of logarithmic stable maps to $Y$. We will produce an algorithm that expresses these as sums of tautological cohomology classes acting on the virtual class of the space of logarithmic stable maps to $Y$. 
%
\subsection{Naive strategy, its issues, and fixes} The approach is based on a naive strategy which, as we explain, does not immediately work. Nevertheless, it paints the right geometric picture. We then explain why it fails, and then the fix. 

Let $(Y|\partial Y)$ be a simple normal crossings pair and let $\mathbb P\to Y$ be a $\mathbb P^1$-bundle. By adding the $0$ and/or $\infty$ sections to the logarithmic structure, we obtain a pair which we denote by $(\mathbb P|\partial \mathbb P)$. We refer to these extra divisors as horizontal, and denote them by $\partial_{\sf hor} \mathbb P$. There is a basic diagram of logarithmic mapping spaces: 
\[
\begin{tikzcd}
\Mbar(\mathbb P|\partial \mathbb P)\arrow{r}\arrow{d} & \Mbar(\mathbb P|\partial_{\sf hor} \mathbb P)\arrow{d}\\
\Mbar(Y|\partial Y)\arrow{r} & \Mbar(Y). 
\end{tikzcd}
\]
All spaces in the diagram are equipped with virtual fundamental classes. The naive approach is to proceed as follows:
\begin{enumerate}[(i)]
\item Argue that the diagram is Cartesian at the level of spaces,
\item argue that $[\Mbar(\mathbb P|\partial \mathbb P)]^{\sf vir}$ is the pullback of $[\Mbar(\mathbb P|\partial_{\sf hor}\mathbb P]^{\sf vir}$ along the bottom arrow, and 
\item use reconstruction results for $\Mbar(\mathbb P|\partial_{\sf hor} \mathbb P)$ proved in~\cite{MP06} to deduce that the GW pushforwards of $\Mbar(\mathbb P|\partial \mathbb P)$ to $\Mbar(Y|\partial Y)$ are pullbacks of tautological classes in $\Mbar(Y)$. 
\end{enumerate}

\noindent
{\bf Every step of the plan fails.} First, the diagram is {\it not} Cartesian unless $\partial_{\sf hor}\mathbb P$ is empty, already in the case where the bundle is trivial~\cite{MR21,NR19}. It is Cartesian in the category of fine and saturated logarithmic stacks but not in the standard category. Even setting this aside, the second and third steps do not even make sense -- it is not clear that the horizontal arrow is equipped with a perfect obstruction theory that identifies virtual classes by pullback. 

A few key ideas allow us to overcome these issues. The first is to utilize the methods of~\cite{MR21}, which explains that after logarithmic blowups and taking a type of strict transform, the diagram {\it becomes} Cartesian in the usual sense. 

The second idea, which is the main input here, is that the intersection theoretic issues on the bottom horizontal arrow are related to the virtual structure on the target of that morphism. So we try to find a better pullback diagram that achieves the same outcome, but has a smooth bottom right corner. If the $\mathbb P^1$-bundle were trivial, one could analyze the GW cycles by using the product formula in GW theory, and the bottom right becomes the stack of curves. In general, we absorb the twistedness of the $\mathbb P^1$-bundle into the Picard stack. 

Observe that 
\[
[\mathbb P^1/\mathbb G_m]\to B\mathbb G_m
\] 
is the universal $\mathbb P^1$-bundle with $0$ and $\infty$ section. Therefore, the bundle $\mathbb P\to Y$ is a pullback
\[
\begin{tikzcd}
{\mathbb P} \arrow{d}\arrow{r} & {[\mathbb P^1/{\mathbb G}_m]}\arrow{d} \\
Y\arrow{r} & B\mathbb G_m.
\end{tikzcd}
\]
The main gain here is that maps to $B\mathbb G_m$, namely the Picard stack, is an unobstructed moduli problem. 

Working with maps to algebraic stacks often poses difficulties when it comes to properness. However, we identify an appropriate ``stability'' condition on the space of logarithmic maps to $[\mathbb P^1/\mathbb G_m]$ that makes it proper over the Picard stack, and this is sufficient to carry out a modified version of (ii) in the strategy. We then use the results of~\cite{MP06} to approximate $B\mathbb G_m$ and deduce that GW classes produced from $[\mathbb P^1/\mathbb G_m]$ are tautological on the Picard stack. This is the key result of this section. In Section~\ref{sec: uni-log-bundle} we adapt the methods of~\cite{MR21} to enhance this to logarithmic cohomology.

\begin{remark}
The strategy laid out is in some ways parallel to the series of papers that lead to a proof that the logarithmic GW theory of toric varieties relative to their {\it full} toric boundary is tautological~\cite{BHPSS,HMPPS,JPPZ2,MR21,RUK22}. 
\end{remark}

\subsection{Curves in the universal projective line} 

Consider the moduli $\fM_{g,n}(B\mathbb G_m)$ of prestable maps from nodal curves to $B\mathbb G_m$, otherwise known as the universal Picard stack of the universal curve over the moduli space of curves. We have universal structure maps:
\[
\textsf{Tot}(\mathcal L)\to \mathcal C\to \fM_{g,n}(B\mathbb G_m),
\]
where $\mathcal C$ is the universal curve and $\mathsf{Tot}(\mathcal L)$ is the total space of the universal line bundle over the universal curve. Pass to the projective completion to obtain:
\[
\mathcal P\to \mathcal C\to \fM_{g,n}(B\mathbb G_m). 
\]

\subsubsection{Absolute theory}  Given a stable map from $C'$ to a projective completion $\mathbb P_Y$ of some line bundle $L$ over some $Y$, we obtain a stabilization $C'\to C$ after composing with $\mathbb P_Y\to Y$. We also get a rational section of the pullback of the line bundle $L$ on every non-contracted component of $C'\to C$. We axiomatize this over the Picard stack. 

A family of pointed nodal curves $\mathcal C/S$ can be equipped with a canonical logarithmic structure -- pull back the logarithmic structure along the map from $S$ to the moduli space of curves, and similarly along the map from $\mathcal C$ to the universal curve. This makes it logarithmically smooth with nontrivial logarithmic structure along nodes and markings. Similarly, a $\mathbb P^1$-bundle over a logarithmically smooth curve is certainly also logarithmically smooth with the pullback logarithmic structure. 

Given a logarithmically smooth curve $C$ and a $\mathbb P^1$-bundle $\PP$ over it, we want to compactify the space of rational sections of $L$. To do so, we study maps
\[
C'\ ``\to"  \PP
\]
from destabilizations $C'\to C$. We have used the quotation marks because the target is singular, so one must be careful in order to have a reasonable deformation theory for maps from curves. We can equip $\PP$ with the logarithmic structure described above, and study logarithmic maps 
\[
C'\to \PP
\]
using the formalism of Abramovich--Chen--Gross--Siebert~\cite{AC11,Che10,GS13}.

\begin{definition}[Maps the universal bundle, collapsed logarithmic version]
Let $\Mbar_{g,n}(\mathcal P/B\mathbb G_m)$ denote the moduli problem on the category of logarithmic schemes, parameterizing flat families of the following data:
\[
\left(C,L,\delta\colon C'\to C, s\colon C'\to{\mathbb P}= \mathbb P(L\oplus\mathcal O)\right)
\]
where
\begin{enumerate}[(i)]
\item $C$ is a genus $g$ nodal $n$-marked curve with logarithmic structure and $L$ is a line bundle on it,
\item the morphism $[\delta\colon C'\to C]$ is a partial destabilization of curves with logarithmic structure,
\item the map $s\colon C'\to \mathbb P$ is a logarithmic morphism any contracted subcurve of $C'$ must be stable,
\item every component of $C'$ that is contracted by $\delta$ maps onto a smooth fiber of $\mathbb P$,
\item every component of $C'$ that is not contracted by $\delta$ maps isomorphically onto its image under the composition
\[
C'\to {\mathbb P}\to C
\]
with the bundle projection.
\end{enumerate}
\end{definition}

Objects in the space may be visualized as show in Figure~\ref{fig: maps-to-p1-bundle}.

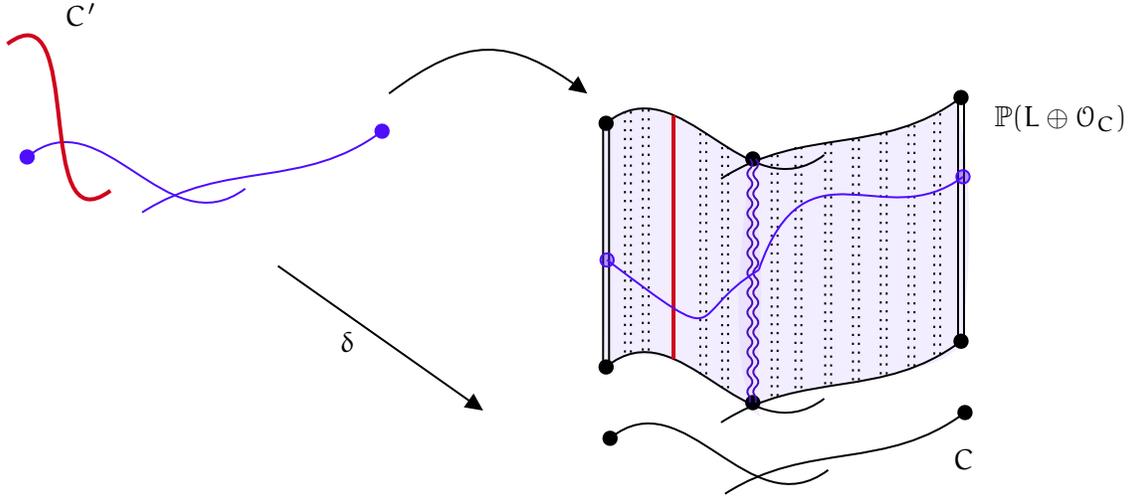
\begin{figure}[h!]

\tikzset{every picture/.style={line width=0.75pt}} 

\begin{tikzpicture}[x=0.75pt,y=0.75pt,yscale=-1,xscale=1]

\definecolor{softlavender}{RGB}{223,210,255}

\draw    (426.5,2147) .. controls (468.5,2120) and (507.5,2136) .. (547.5,2106) ;
\draw [shift={(547.5,2106)}, rotate = 323.13] [color={rgb, 255:red, 0; green, 0; blue, 0 }  ][fill={rgb, 255:red, 0; green, 0; blue, 0 }  ][line width=0.75]      (0, 0) circle [x radius= 3.35, y radius= 3.35]   ;
\draw    (368.5,2119) .. controls (406.5,2088) and (438.5,2165) .. (478.5,2135) ;
\draw [shift={(368.5,2119)}, rotate = 320.79] [color={rgb, 255:red, 0; green, 0; blue, 0 }  ][fill={rgb, 255:red, 0; green, 0; blue, 0 }  ][line width=0.75]      (0, 0) circle [x radius= 3.35, y radius= 3.35]   ;
\draw [color={rgb, 255:red, 78; green, 13; blue, 253 }  ,draw opacity=1 ]   (132.5,2005) .. controls (174.5,1978) and (213.5,1994) .. (253.5,1964) ;
\draw [shift={(253.5,1964)}, rotate = 323.13] [color={rgb, 255:red, 78; green, 13; blue, 253 }  ,draw opacity=1 ][fill={rgb, 255:red, 78; green, 13; blue, 253 }  ,fill opacity=1 ][line width=0.75]      (0, 0) circle [x radius= 3.35, y radius= 3.35]   ;
\draw [color={rgb, 255:red, 78; green, 13; blue, 253 }  ,draw opacity=1 ]   (74.5,1977) .. controls (112.5,1946) and (144.5,2023) .. (184.5,1993) ;
\draw [shift={(74.5,1977)}, rotate = 320.79] [color={rgb, 255:red, 78; green, 13; blue, 253 }  ,draw opacity=1 ][fill={rgb, 255:red, 78; green, 13; blue, 253 }  ,fill opacity=1 ][line width=0.75]      (0, 0) circle [x radius= 3.35, y radius= 3.35]   ;
\draw [color={rgb, 255:red, 208; green, 2; blue, 27 }  ,draw opacity=1 ][line width=1.5]    (64.5,1920) .. controls (104.5,1890) and (76.5,2024) .. (116.5,1994) ;
\draw    (257,1945) .. controls (296.2,1915.6) and (317.63,1915.97) .. (354.71,1943.29) ;
\draw [shift={(357,1945)}, rotate = 216.99] [fill={rgb, 255:red, 0; green, 0; blue, 0 }  ][line width=0.08]  [draw opacity=0] (8.93,-4.29) -- (0,0) -- (8.93,4.29) -- cycle    ;
\draw    (201,2032) -- (302.05,2103.27) ;
\draw [shift={(304.5,2105)}, rotate = 215.2] [fill={rgb, 255:red, 0; green, 0; blue, 0 }  ][line width=0.08]  [draw opacity=0] (8.93,-4.29) -- (0,0) -- (8.93,4.29) -- cycle    ;
\draw  [draw opacity=0][fill={rgb, 255:red, 223; green, 210; blue, 255 }  ,fill opacity=0.4 ] (440.5,1978) .. controls (453.5,1973.5) and (524.5,1966.5) .. (545.5,1947) .. controls (539.5,1946) and (557.5,1986) .. (545.5,2070) .. controls (540.5,2090.5) and (448.5,2090.5) .. (443.5,2105.5) .. controls (424.5,2088.5) and (437.5,1986) .. (440.5,1978) -- cycle ;
\draw  [draw opacity=0][fill={rgb, 255:red, 223; green, 210; blue, 255 }  ,fill opacity=0.4 ] (366.5,1960) .. controls (384.5,1937.5) and (421.5,1969.5) .. (440.5,1978) .. controls (434.5,1977) and (452.5,2017) .. (440.5,2101) .. controls (460.5,2131) and (400.5,2055) .. (366.5,2083) .. controls (364.5,2062.5) and (363.5,1968) .. (366.5,1960) -- cycle ;
\draw    (424.5,2111) .. controls (466.5,2084) and (505.5,2100) .. (545.5,2070) ;
\draw [shift={(545.5,2070)}, rotate = 323.13] [color={rgb, 255:red, 0; green, 0; blue, 0 }  ][fill={rgb, 255:red, 0; green, 0; blue, 0 }  ][line width=0.75]      (0, 0) circle [x radius= 3.35, y radius= 3.35]   ;
\draw    (366.5,2083) .. controls (404.5,2052) and (436.5,2129) .. (476.5,2099) ;
\draw [shift={(366.5,2083)}, rotate = 320.79] [color={rgb, 255:red, 0; green, 0; blue, 0 }  ][fill={rgb, 255:red, 0; green, 0; blue, 0 }  ][line width=0.75]      (0, 0) circle [x radius= 3.35, y radius= 3.35]   ;
\draw    (440.5,2101) ;
\draw [shift={(440.5,2101)}, rotate = 0] [color={rgb, 255:red, 0; green, 0; blue, 0 }  ][fill={rgb, 255:red, 0; green, 0; blue, 0 }  ][line width=0.75]      (0, 0) circle [x radius= 3.35, y radius= 3.35]   ;
\draw    (424.5,1988) .. controls (466.5,1961) and (505.5,1977) .. (545.5,1947) ;
\draw [shift={(545.5,1947)}, rotate = 323.13] [color={rgb, 255:red, 0; green, 0; blue, 0 }  ][fill={rgb, 255:red, 0; green, 0; blue, 0 }  ][line width=0.75]      (0, 0) circle [x radius= 3.35, y radius= 3.35]   ;
\draw    (366.5,1960) .. controls (404.5,1929) and (436.5,2006) .. (476.5,1976) ;
\draw [shift={(366.5,1960)}, rotate = 320.79] [color={rgb, 255:red, 0; green, 0; blue, 0 }  ][fill={rgb, 255:red, 0; green, 0; blue, 0 }  ][line width=0.75]      (0, 0) circle [x radius= 3.35, y radius= 3.35]   ;
\draw    (440.5,1978) ;
\draw [shift={(440.5,1978)}, rotate = 0] [color={rgb, 255:red, 0; green, 0; blue, 0 }  ][fill={rgb, 255:red, 0; green, 0; blue, 0 }  ][line width=0.75]      (0, 0) circle [x radius= 3.35, y radius= 3.35]   ;
\draw    (368,1960) -- (368,2083)(365,1960) -- (365,2083) ;
\draw    (547,1947) -- (547,2070)(544,1947) -- (544,2070) ;
\draw  [dash pattern={on 0.84pt off 2.51pt}]  (388,1953) -- (388,2076)(385,1953) -- (385,2076) ;
\draw [color={rgb, 255:red, 208; green, 2; blue, 27 }  ,draw opacity=1 ][line width=1.5]    (400.5,1956) -- (400.5,2079) ;
\draw  [dash pattern={on 0.84pt off 2.51pt}]  (417,1965) -- (417,2088)(414,1965) -- (414,2088) ;
\draw  [dash pattern={on 0.84pt off 2.51pt}]  (379,1954) -- (379,2077)(376,1954) -- (376,2077) ;
\draw  [dash pattern={on 0.84pt off 2.51pt}]  (535,1955) -- (535,2078)(532,1955) -- (532,2078) ;
\draw  [dash pattern={on 0.84pt off 2.51pt}]  (522,1961) -- (522,2084)(519,1961) -- (519,2084) ;
\draw  [dash pattern={on 0.84pt off 2.51pt}]  (508,1965) -- (508,2088)(505,1965) -- (505,2088) ;
\draw  [dash pattern={on 0.84pt off 2.51pt}]  (494,1968) -- (494,2091)(491,1968) -- (491,2091) ;
\draw  [dash pattern={on 0.84pt off 2.51pt}]  (480,1970) -- (480,2093)(477,1970) -- (477,2093) ;
\draw  [dash pattern={on 0.84pt off 2.51pt}]  (465,1972) -- (465,2095)(462,1972) -- (462,2095) ;
\draw  [dash pattern={on 0.84pt off 2.51pt}]  (453,1976) -- (453,2099)(450,1976) -- (450,2099) ;
\draw [color={rgb, 255:red, 74; green, 13; blue, 182 }  ,draw opacity=1 ]   (442,1978) .. controls (443.67,1979.67) and (443.67,1981.33) .. (442,1983) .. controls (440.33,1984.67) and (440.33,1986.33) .. (442,1988) .. controls (443.67,1989.67) and (443.67,1991.33) .. (442,1993) .. controls (440.33,1994.67) and (440.33,1996.33) .. (442,1998) .. controls (443.67,1999.67) and (443.67,2001.33) .. (442,2003) .. controls (440.33,2004.67) and (440.33,2006.33) .. (442,2008) .. controls (443.67,2009.67) and (443.67,2011.33) .. (442,2013) .. controls (440.33,2014.67) and (440.33,2016.33) .. (442,2018) .. controls (443.67,2019.67) and (443.67,2021.33) .. (442,2023) .. controls (440.33,2024.67) and (440.33,2026.33) .. (442,2028) .. controls (443.67,2029.67) and (443.67,2031.33) .. (442,2033) .. controls (440.33,2034.67) and (440.33,2036.33) .. (442,2038) .. controls (443.67,2039.67) and (443.67,2041.33) .. (442,2043) .. controls (440.33,2044.67) and (440.33,2046.33) .. (442,2048) .. controls (443.67,2049.67) and (443.67,2051.33) .. (442,2053) .. controls (440.33,2054.67) and (440.33,2056.33) .. (442,2058) .. controls (443.67,2059.67) and (443.67,2061.33) .. (442,2063) .. controls (440.33,2064.67) and (440.33,2066.33) .. (442,2068) .. controls (443.67,2069.67) and (443.67,2071.33) .. (442,2073) .. controls (440.33,2074.67) and (440.33,2076.33) .. (442,2078) .. controls (443.67,2079.67) and (443.67,2081.33) .. (442,2083) .. controls (440.33,2084.67) and (440.33,2086.33) .. (442,2088) .. controls (443.67,2089.67) and (443.67,2091.33) .. (442,2093) .. controls (440.33,2094.67) and (440.33,2096.33) .. (442,2098) -- (442,2101) -- (442,2101)(439,1978) .. controls (440.67,1979.67) and (440.67,1981.33) .. (439,1983) .. controls (437.33,1984.67) and (437.33,1986.33) .. (439,1988) .. controls (440.67,1989.67) and (440.67,1991.33) .. (439,1993) .. controls (437.33,1994.67) and (437.33,1996.33) .. (439,1998) .. controls (440.67,1999.67) and (440.67,2001.33) .. (439,2003) .. controls (437.33,2004.67) and (437.33,2006.33) .. (439,2008) .. controls (440.67,2009.67) and (440.67,2011.33) .. (439,2013) .. controls (437.33,2014.67) and (437.33,2016.33) .. (439,2018) .. controls (440.67,2019.67) and (440.67,2021.33) .. (439,2023) .. controls (437.33,2024.67) and (437.33,2026.33) .. (439,2028) .. controls (440.67,2029.67) and (440.67,2031.33) .. (439,2033) .. controls (437.33,2034.67) and (437.33,2036.33) .. (439,2038) .. controls (440.67,2039.67) and (440.67,2041.33) .. (439,2043) .. controls (437.33,2044.67) and (437.33,2046.33) .. (439,2048) .. controls (440.67,2049.67) and (440.67,2051.33) .. (439,2053) .. controls (437.33,2054.67) and (437.33,2056.33) .. (439,2058) .. controls (440.67,2059.67) and (440.67,2061.33) .. (439,2063) .. controls (437.33,2064.67) and (437.33,2066.33) .. (439,2068) .. controls (440.67,2069.67) and (440.67,2071.33) .. (439,2073) .. controls (437.33,2074.67) and (437.33,2076.33) .. (439,2078) .. controls (440.67,2079.67) and (440.67,2081.33) .. (439,2083) .. controls (437.33,2084.67) and (437.33,2086.33) .. (439,2088) .. controls (440.67,2089.67) and (440.67,2091.33) .. (439,2093) .. controls (437.33,2094.67) and (437.33,2096.33) .. (439,2098) -- (439,2101) -- (439,2101) ;
\draw  [dash pattern={on 0.84pt off 2.51pt}]  (428,1972) -- (428,2095)(425,1972) -- (425,2095) ;
\draw [color={rgb, 255:red, 78; green, 13; blue, 253 }  ,draw opacity=1 ][line width=0.75]    (367,2029) .. controls (430.5,2077) and (405.5,2057) .. (443.5,2034) ;
\draw [shift={(367,2029)}, rotate = 37.09] [color={rgb, 255:red, 78; green, 13; blue, 253 }  ,draw opacity=1 ][fill={rgb, 255:red, 78; green, 13; blue, 253 }  ,fill opacity=0.4 ][line width=0.75]      (0, 0) circle [x radius= 3.35, y radius= 3.35]   ;
\draw [color={rgb, 255:red, 78; green, 13; blue, 253 }  ,draw opacity=1 ][line width=0.75]    (443.5,2034) .. controls (467.5,1965) and (506.5,2017) .. (546.5,1987) ;
\draw [shift={(546.5,1987)}, rotate = 323.13] [color={rgb, 255:red, 78; green, 13; blue, 253 }  ,draw opacity=1 ][fill={rgb, 255:red, 78; green, 13; blue, 253 }  ,fill opacity=0.4][line width=0.75]      (0, 0) circle [x radius= 3.35, y radius= 3.35]   ;

\draw (561,1949) node [anchor=north west][inner sep=0.75pt]    {$\mathbb{P}( L\oplus \mathcal{O}_{C})$};
\draw (540,2123) node [anchor=north west][inner sep=0.75pt]    {$C$};
\draw (231,2064) node [anchor=north west][inner sep=0.75pt]    {$\delta $};
\draw (92,1898) node [anchor=north west][inner sep=0.75pt]    {$C'$};

\end{tikzpicture}
\caption{Stable maps to a projective bundle over a nodal curve.}\label{fig: maps-to-p1-bundle}
\end{figure}

As the space is modelled on a space of maps to projective bundles, it inherits some of the basic properties of these spaces, but relative to the Picard stack. 

\begin{proposition}
The moduli problem $\Mbar_{g,n}(\mathcal P/B\mathbb G_m)$ is represented by an Artin stack equipped with a logarithmic structure. The morphism 
\[
{\Mbar_{g,n}}(\mathcal P/B\mathbb G_m)\to {\mathfrak{M}}_{g,n}(B\mathbb G_m)
\]
is proper and relatively Deligne--Mumford, and is equipped with a relative logarithmic perfect obstruction theory given by $R^\bullet \pi_\star s^\star T^{\sf log}$ where $T^{\sf log}$ of $\mathcal P$ over the stack $\Mbar_{g,n}(\mathcal P/B\mathbb G_m)$. The parallel statements hold for $\Mbar^{\sf exp}_{g,n}(\mathcal P/B\mathbb G_m)$. 
\end{proposition}

\begin{proof}
The results are standard, so we provide a sketch. Representability of the logarithmic version can be deduced immediately from Wise's results on mapping stacks~\cite{Wis16a,Wis16b}, which recognizes the logarithmic mapping stack as a closed inside the ordinary mapping stack~\cite{AHR19}. In particular, Wise's proof shows that the logarithmic mapping stack admits minimal objects. To see the Deligne--Mumford property, it suffices to check the relative finiteness of stabilizers -- this has been rigged to be true. There are no automorphisms of the projective bundle $\mathbb P$ relative to the base, since $(C,L)$ is fixed.
	
	For the properness, note that the map is of finite type, so we only need to check the valuative criterion of properness. Take a one-parameter family of curves $\mathcal C/K$, a line bundle on the family, and a map from the general fiber to the universal projective bundle. By expressing the line bundle as a difference of amples, we can view the line bundle as a pullback of an embedding from $\mathcal C$ into a product of projective spaces, and the $\mathbb P^1$-bundle is pulled back from a $\mathbb P^1$-bundle on this space here. The existence part of the valuative criterion follows from this. Uniqueness follows from the stability condition -- if there were two limits, stable in the sense above, embedding in a projective bundle over a product of projective spaces as above, the separatedness of the stack of stable maps would be violated. So we conclude. 
	
	The virtual structure is given by the cohomology groups of the relative logarithmic tangent bundle by standard arguments~\cite{BF97,GS13}.
\end{proof}

\begin{remark}[Minimal objects and tropicalizations]
    The moduli stack $\Mbar_{g,n}(\mathcal P/B\mathbb G_m)$ is representable by an algebraic stack with logarithmic structure, and therefore defines a moduli problem on logarithmic schemes -- which was given as the definition -- as well as a moduli problem on schemes. We record the latter. If $S$ is a scheme, the $S$-valued points of $\Mbar_{g,n}(\mathcal P/B\mathbb G_m)$ is  the groupoid of families over logarithmic schemes $(S, M_S)$ that satisfy a condition called {\it minimality}. In our case, a family $(C,L,C',s)$ over a logarithmic base $(S,M_S)$ is minimal if and only if the moduli map to the logarithmic moduli space of curves given by $[C']$
    \[
    (S,M_S)\xrightarrow{[C']} (\fM_{g,n},\partial \fM_{g,n})
    \]
    is strict. 

    Before justifying this, let us spell it out; this material is essentially~\cite[Section~1]{GS13}\footnote{What we call minimality is called {\it basicness} in~\cite{GS13}. Our terminology is consistent with~\cite{AC11,Che10}.}. Minimality is a pointwise condition on $S$, and we explain what it is when $S$ is a geometric point, with monoid $P$, given by the characteristic monoid $\overline M_S$. In our case, a family $(C,L,C',s)$ as above over a logarithmic point $S = \Spec(P\to \mathbb C)$ is minimal if and only if $P$ is $\mathbb N^E$, where $E$ is the set of nodes of $C'$, and the at the level of the characteristic monoids, the logarithmic morphism
    \[
    (C,M_C)\to (S,M_S)
    \]
    is given, locally at a node $q$ by $\overline x_q+\overline y_q = \overline t_q)$, where $t_q$ is a generator for the corresponding factor in $\mathbb N^E$ and $\overline x_q$ and $\overline y_q$ are the images in the characteristic sheaf of $(C,M_C)$ of the functions cutting out the branches at $q$. 

    To deduce this the claimed description of minimial monoids, we apply the observation of~\cite[Remark~1.18 \& 1.21]{GS13} that characterizes the minimal monoids tropically. It asserts that a family over  $\Spec(P\to \mathbb C)$ if the associated family of tropical curves is identified with the dual cone of $P$. The description in~\cite[Section~1.5]{GS13} immediately gives the claim. Indeed, given $(C,L,C'/L,s)$, the associated tropical moduli space is the space of lengths enhancing the map on dual graphs $\Gamma'\to \Gamma$ given by destabilizations. But this is exactly the set of lengths of edges of the domain $\Gamma'$. The depiction in Figure~\ref{fig: tropical-universal-bundle} may be helpful. The claim follows. 
\end{remark}

\begin{figure}

\tikzset{every picture/.style={line width=0.75pt}} 

\begin{tikzpicture}[x=0.75pt,y=0.75pt,yscale=-1,xscale=1]

\definecolor{softlavender}{RGB}{223,210,255}

\def\vr{2}

\colorlet{violetcol}{violet}
\colorlet{bluecol}{blue}

\fill[softlavender, opacity=0.4] (120,2650) rectangle (460,2850);


\draw[violetcol] (348.5,2825.5) -- (422.5,2806);
\draw[violetcol, fill=violetcol] (422.5,2806) circle [x radius=\vr, y radius=\vr];
\draw[violetcol, fill=violetcol] (348.5,2825.5) circle [x radius=\vr, y radius=\vr];

\draw[violetcol] (348.5,2825.5) -- (447.5,2842);
\draw[violetcol, fill=violetcol] (447.5,2842) circle [x radius=\vr, y radius=\vr];

\draw[violetcol] (146.5,2805) -- (249.5,2827);
\draw[violetcol, fill=violetcol] (146.5,2805) circle [x radius=\vr, y radius=\vr];

\draw[violetcol] (188.5,2846) -- (249.5,2827);
\draw[violetcol, fill=violetcol] (188.5,2846) circle [x radius=\vr, y radius=\vr];

\draw[violetcol] (249.5,2827) .. controls (289.5,2797) and (329.5,2812.5) .. (349.5,2827);
\draw[violetcol, fill=violetcol] (349.5,2827);

\draw[violetcol] (249.5,2827) .. controls (306.5,2852.5) and (329.5,2832.5) .. (348.5,2825.5);

\begin{scope}[shift={(0,-100)}]

\draw[violetcol] (348.5,2825.5) -- (422.5,2806);
\draw[violetcol, fill=violetcol] (422.5,2806) circle [x radius=\vr, y radius=\vr];
\draw[violetcol, fill=violetcol] (348.5,2825.5) circle [x radius=\vr, y radius=\vr];

\draw[violetcol] (348.5,2825.5) -- (447.5,2842);
\draw[violetcol, fill=violetcol] (447.5,2842) circle [x radius=\vr, y radius=\vr];

\draw[violetcol] (146.5,2805) -- (249.5,2827);
\draw[violetcol, fill=violetcol] (146.5,2805) circle [x radius=\vr, y radius=\vr];

\draw[violetcol] (188.5,2846) -- (249.5,2827);
\draw[violetcol, fill=violetcol] (188.5,2846) circle [x radius=\vr, y radius=\vr];

\draw[bluecol] (228.5,2763) -- (233.5,2808);
\draw[bluecol] (208.5,2779) -- (249.5,2827);
\draw[bluecol] (285.5,2770) -- (289.5,2810.75);

\draw[violetcol] (249.5,2827) .. controls (289.5,2797) and (329.5,2812.5) .. (349.5,2827);
\draw[violetcol, fill=violetcol] (349.5,2827) ;

\draw[violetcol] (249.5,2827) .. controls (306.5,2852.5) and (329.5,2832.5) .. (348.5,2825.5);

\draw[bluecol, fill=bluecol] (233.5,2808) circle [x radius=\vr, y radius=\vr];
\draw[bluecol, fill=bluecol] (228.5,2763) circle [x radius=\vr, y radius=\vr];
\draw[bluecol, fill=bluecol] (289.5,2810.75) circle [x radius=\vr, y radius=\vr];
\draw[bluecol, fill=bluecol] (285.5,2770) circle [x radius=\vr, y radius=\vr];
\draw[bluecol, fill=bluecol] (249.5,2827) circle [x radius=\vr, y radius=\vr];
\draw[bluecol, fill=bluecol] (208.5,2779) circle [x radius=\vr, y radius=\vr];

\end{scope}

\end{tikzpicture}
\caption{The dual tropical map corresponding to a logarithmic map to the universal projective bundle. The tropical map is constant on the blue part of the graph corresponding to the destabilization. The corresponding tropical moduli space has has a system of coordinates given by all the edges of $\Gamma'$}\label{fig: tropical-universal-bundle}
\end{figure}
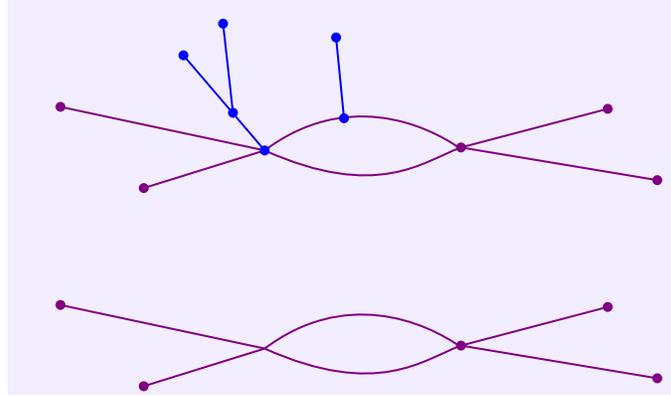

\begin{remark}\label{prop: turns-out-theyre-the-same}
In place of the logarithmic mapping space, one could use expanded degenerations to study maps to the bundle $\mathbb P\to C$. Call the space $\Mbar^{\sf exp}_{g,n}(\mathcal P/B\mathbb G_m)$. By general results of Abramovich--Marcus--Wise~\cite{AMW12}, one can compare this to the logarithmic mapping space constructed above. Going through the details of the analysis there, it turns out the moduli spaces $\Mbar_{g,n}(\mathcal P/B\mathbb G_m)$ and $\Mbar^{\sf exp}_{g,n}(\mathcal P/B\mathbb G_m)$ are naturally isomorphic as algebraic stacks. 
\end{remark}

\subsubsection{Logarithmic theory}\label{sec: univ-log-theory} We study stable maps to the universal projective bundle, relative to the $0$-section, the $\infty$-section, or both. The stack $[\mathbb P^1/\mathbb G_m]$ is equipped with a logarithmic structure, as appropriate in each of these cases.

Consider the space $\Mbar_{g,n}(\mathcal P/B\mathbb G_m)$. An object in the space is a map to a universal projective bundle and by composition, we have:
\[
C'\to \mathbb P(L\oplus\mathcal O)\to [\mathbb P^1/\mathbb G_m].
\]
A choice of logarithmic structure on $[\mathbb P^1/\mathbb G_m]$ equips the universal $\mathbb P^1$-bundle $\mathbb P(L\oplus\mathcal O)$ with a logarithmic structure. We define the space of logarithmic stable maps to the universal logarithmic $\mathbb P^1$-bundle as follows. 

Now let $\mathsf A$ be the Artin fan of $[\mathbb P^1/\mathbb G_m]$ with the logarithmic structure coming from $D$. This could either be $[\mathbb A^1/\mathbb G_m]$ or $[\mathbb P^1/\mathbb G_m]$. We fix additional discrete data -- the tangency of the section. Specifically, the non-degenerate objects in the moduli space are maps
\[
s\colon C\to \mathbb P(L\oplus\mathcal O)
\]
such that, at each marked point of $p$, along the $0$ and/or $\infty$-section of the bundle -- depending on the logarithmic structure -- we fix the order of vanishing of $s$ along this divisor at $p$. We are interested in degenerations of maps with these tangency, and the tangency condition forms discrete data. Altogether, we use $\Lambda$ to denote the genus, number of markings, and the tangency. 

Define the space of logarithmic stable maps as the following fiber product in the category of fine and saturated logarithmic stacks:
\[
\begin{tikzcd}
    \Mbar_{\Lambda}(\mathcal P|D/B\mathbb G_m)\arrow{d}\arrow{r} & \Mbar_{g,n}(\mathcal P/B\mathbb G_m)\arrow{d}\\
    \mathfrak M^{\sf log}_\Lambda(\mathsf A)\arrow{r}& \mathfrak M_{g,n}(\mathsf A),
\end{tikzcd}
\]
where $\mathfrak M_{g,n}(\mathsf A)$ is the stack of all maps from nodal $n$-marked curves of genus $g$ to the Artin stack $\mathsf A$, equipped with the standard logarithmic structure coming from $\fM_{g,n}$. Also $\mathfrak M^{\sf log}_\Lambda(\mathsf A)$ is the space of logarithmic maps to the Artin fan, constructed in~\cite{AW}.

\begin{proposition}
The moduli problem $\Mbar_{\Lambda}(\mathcal P|D/B\mathbb G_m)$ is represented by an Artin stack with logarithmic structure. The morphism 
\[
\Mbar_{\Lambda}(\mathcal P|D/B\mathbb G_m)\to \fM_{g,n}(B\mathbb G_m) 
\]
is proper and relatively Deligne--Mumford. Furthermore, there is a relative perfect obstruction theory
\[
\Mbar_{\Lambda}(\mathcal P|D/B\mathbb G_m)\to \mathsf{Log}(\fM_{g,n}(B\mathbb G_m)) 
\]
over the stack of logarithmic structures over $\fM_{g,n}(B\mathbb G_m)$. The associated obstruction bundle is given by $R^\bullet \pi_\star s^\star T^{\sf log}_{\mathcal P/\mathfrak M}$, where $T_{\mathcal P/\mathfrak M}$ is the relative logarithmic tangent bundle of the universal projective bundle. 
\end{proposition}

\begin{proof}
	The representability of the logarithmic mapping space, in this generality, follows from results of Wise~\cite{Wis16a}. The logarithmic mapping space is quasi-finite and proper over the underlying mapping space (disregarding logarithmic structures) and we have already shown relative properness in that case. The remaining statement is the obstruction theory, which again is standard. 
\end{proof}

\begin{remark}[Rubber theory]\label{rem: uni-rubber} When studying relative stable maps to a $\mathbb P^1$-bundle over a smooth projective variety, say $\mathbb P\to Y$, with relative divisor given by both the $0$ and $\infty$ sections, its often natural to consider the associated {\it rubber theory}, e.g.~\cite{FP,MP06}. In the interior, we consider maps 
\[
C\to \mathbb P
\]
meeting the sections with prescribed tangency, but consider two such maps to be the same if the differ by $\mathbb G_m$-dilation action on the fibers. The foundational theory can be treated in many ways. One way is to start with the locus of relative stable maps to $\mathbb P$ using Jun Li's expanded theory, and then delete the closed locus where the map to the main component of the bundle is stable under $\mathbb G_m$ action. We then take a quotient by the free action of $\mathbb G_m$. 

A different approach can be carried out using the logarithmic multiplicative group, see for instance~\cite{RUK22,RW19}. We start with the moduli space of logarithmic maps to the $\mathbb G_{\sf log}$-bundle associated to $\mathbb P\to Y$. Stability for such maps is the same as stability of the underlying map to $Y$. By arguments from~\cite{RUK22}, the moduli space admits a free action of $\mathbb G_{\sf log}$, and the quotient is the rubber moduli space. 

Using either of these approaches, we can define a rubber variant for the space $\Mbar_{\Lambda}(\mathcal P|D/B\mathbb G_m)$, when $D$ consists of both $0$ and $\infty$ sections. The rubber variant carries a virtual fundamental class over the Picard stack as above. It arises naturally when describing strata of the moduli space of relative maps. For a geometric picture of the discussion, see~\cite[Section~2]{GV05}.
\end{remark}

\subsection{Compatibility with pullbacks} The universal bundle theories are compatible with stable maps to projective lines over any particular target. 

To spell this out, let $
\fM_{g,n}(\mathcal P)$ denote the stack parameterizing data:
\[
(C,L,\delta\colon C'\to C,s\colon C'\to \mathbb P(L\oplus \mathcal O)),
\]
where, as before, $C$ is a prestable curve of genus $g$ with $n$ markings, the map $\delta\colon C'\to C$ is a partial destabilization, and $s\colon C'\to \mathbb P(L\oplus \mathcal O)$. 

This is a very large stack, but we only need it to formulate a clean statement, so we don't bother cutting it down any further. Standard arguments show that it is an algebraic stack, and we equip it with a logarithmic structure coming from 
\[
\fM_{g,n}(\mathcal P)\to\fM_{g,n}.
\]
This coincides with the logarithmic structure coming from viewing this as a logarithmic mapping stack, but equipping $\mathbb P(L\oplus\mathcal O)$ with the trivial logarithmic structure. 

There is certainly a morphism
\[
\Mbar_{g,n}^{\sf log}(\mathcal P/B\mathbb G_m)\to 
\fM_{g,n}(\mathcal P).
\]
It is the underlying morphism of a morphism between logarithmic algebraic stacks; this is a consequence of the fact that the morphism from any logarithmic scheme to the underlying scheme with the trivial logarithmic structure is a morphism in the logarithmic category. 

Let $Y$ be a smooth projective variety and let $\mathbb P\to Y$ be the projective completion of a line bundle $L$. Given a stable map
\[
C\to Y
\]
the pullback of $L$ to $C$ defines a morphism from the stack of all stable maps to $Y$:
\[
\Mbar_{g,n}(Y)\to \fM_{g,n}(B\mathbb G_m).
\]
Again, we view these as logarithmic mapping stacks where the target has trivial logarithmic structure, so both domain and target are strict over $\fM_{g,n}$ with these logarithmic structures, and the morphism is the underlying morphism of a logarithmic morphism with these structures. 

Given a stable map $[C'\to \mathbb P]$, we can stabilize the composition with $\mathbb P\to Y$ to obtain a partial stabilization $C$. Furthermore, given a map to $Y$, a lift to $\mathbb P$ is exactly the data of a rational section of the line bundle $L$, pulled back to the curve. Indeed, the diagram
\[
\begin{tikzcd}
    {\mathbb P}\arrow{d}\arrow{r} & {[\mathbb P^1/\mathbb G_m]}\arrow{d}\\
    Y\arrow{r} & B\mathbb G_m
\end{tikzcd}
\]
is cartesian. The relative functor of points of $[\mathbb P^1/\mathbb G_m]$ over $B\mathbb G_m$ is, over a line bundle $L$ on a scheme $S$, a pair of line bundle-section pairs $(L_1,s_1)$ and $(L_2,s_2)$ such that $L_1\otimes L_2^{-1} = L$ and $\mathbb V(s_1)\cap\mathbb V(s_2) = \emptyset$. The rational section is $\frac{s_1}{s_2}$. 

Therefore, given the map $C'\to \PP$ and recording these line bundles and sections, we obtain a morphism
\[
\Phi\colon \Mbar_{g,n}(\mathbb P,\beta)\to \fM_{g,n}(\mathcal P).
\]
of logarithmic algebraic stacks. 

We have the following basic fact:
\begin{lemma}
    The morphism $\Phi$ of logarithmic algebraic stacks above factors as
    \[
    \Mbar_{g,n}(\mathbb P,\beta)\to \Mbar_{g,n}(\mathcal P/B\mathbb G_m) \to \fM_{g,n}(\mathcal P)
    \]
\end{lemma}

\begin{proof}
    We verify that, given a stable map $[C'\to \mathbb P]$, the data obtained satisfies the conditions defining $\Mbar_{g,n}(\mathcal P/B\mathbb G_m) \to \fM_{g,n}(\mathcal P)$ as a subfunctor of the logarithmic functor of points. The stability and the condition that components of $C'\to C$ that are contracted must map vertically follows by construction. 

    The more subtle condition, at least at first glance, is that any object in $\fM_{g,n}(\mathcal P)$ that happens to arise from a stable map to $\mathbb P$ necessarily satisfies the conditions to lie in $\Mbar_{g,n}(\mathcal P/B\mathbb G_m)$. To see this, observe that the stabilization morphism
    \[
    \delta\colon C'\to C
    \]
    is logarithmic. The arrows in the following diagram are all logarithmic:
    \[
    \begin{tikzcd}
    C'\arrow{r} \arrow[swap]{dr}{\delta} & \mathbb P(L\oplus \mathcal O)\arrow{d}\arrow{r} & \mathbb P\arrow{d}\\
     & C\arrow{r} & Y.
     \end{tikzcd}
    \]
    Here $Y$ and $\mathbb P$ are given the trivial logarithmic structure, while $C'$ is given the minimal logarithmic structure; 
    the logarithmic structure on $C$ is defined by pushforward along the stabilization, see for instance~\cite[Appendix~B]{AMW12}. The square is cartesian in the logarithmic category as well as the ordinary category; the latter is evident, and the former follows because the logarithmic structure on $Y$ is trivial. The map
    \[
    C'\to\mathbb P(L\oplus \mathcal O)
    \]
    is logarithmic by universal property. 
\end{proof}

We state our comparison results. Fix a projective bundle $\mathbb P\to Y$ as before. Let $\beta$ be a curve class on $\mathbb P$ and let $\overline \beta$ be the pushforward on $Y$. 

\begin{proposition}
    The diagram
    \[
    \begin{tikzcd}
        \Mbar_{g,n}(\PP,\beta)\arrow{d}\arrow{r} & \Mbar_{g,n}(\mathcal P/B\mathbb G_m)\arrow{d}\\
        \Mbar_{g,n}(Y,\overline\beta)\arrow{r} &\fM_{g,n}(B\mathbb G_m).
    \end{tikzcd}
    \]
    is cartesian in the category of algebraic stacks. 
\end{proposition}

\begin{proof}
    Follows immediately from universal mapping properties and the previous lemma. 
\end{proof}

We have an analogous statement for the logarithmic theory. Fix a logarithmic structure $D$ on $[\mathbb P^1/\mathbb G_m]$ and discrete data $\Lambda$. The square
\[
\begin{tikzcd}
    {\PP}\arrow{d}\arrow{r} & {[\PP^1/\mathbb G_m]}\arrow{d}\\
    Y\arrow{r} & B\mathbb G_m
\end{tikzcd}
\]
determines a logarithmic structure on $\PP$. We slightly abuse notation and use $\Mbar_{\Lambda}(\PP|D,\beta)$ for the space of maps with this logarithmic structure.

\begin{proposition}
    The diagram
    \[
    \begin{tikzcd}
        \Mbar_{\Lambda}(\PP|D,\beta)\arrow{d}\arrow{r} & \Mbar_{\Lambda}(\mathcal P|D/B\mathbb G_m)\arrow{d}\\
        \Mbar_{g,n}(Y,\overline\beta)\arrow{r} &\fM_{g,n}(B\mathbb G_m).
    \end{tikzcd}
    \]
    is cartesian in the category of algebraic stacks. 
\end{proposition}

\begin{proof}
    The stack of logarithmic maps $\Mbar_{\Lambda}(\PP|D,\beta)$ is the pullback of
    \[
    \mathfrak M^{\sf log}_\Lambda(\mathsf A)\to \mathfrak M_{g,n}(\mathsf A)
    \]
    under the map from $\Mbar_{g,n}(\PP,\beta)$ to $\mathfrak M_{g,n}(\mathsf A)$ induced by composition with the map $\PP\to\mathsf A$ to the Artin fan. 
\end{proof}

\begin{remark}
    The fact that the diagrams above are cartesian in algebraic stacks, rather than only in logarithmic algebraic stacks, is a consequence of the fact that $Y$ has trivial logarithmic structure. This makes
    \[
    \Mbar_{g,n}(Y,\beta)\to \fM_{g,n}(B\mathbb G_m)
    \]
    strict. In the next step, we will add a logarithmic structure to $Y$. The same argument above will then guarantee that the analogous diagram is cartesian, but only in the logarithmic category. This complexity will be addressed in the next section. 
\end{remark}

\subsection{GW and tautological classes on the universal bundle} We show that the GW classes of the universal $\mathbb P^1$-bundle, with or without logarithmic structure, are ``tautological'' in the cohomology or Chow ring of $\mathfrak M_{g,n}(B\mathbb G_m)$. It is  useful to have a definition for the latter notion. For foundations of operational Chow cohomology in this setting of locally finite type Artin stacks, we refer to~\cite[Section~2]{BHPSS}.

Tautological classes on the Picard stack can be approached via a notion due to Bae~\cite{Bae19}. 

For the stack of stable maps $\Mbar_{g,n}(X,\beta)$, Bae defines a subring of its total Chow {\it homology} called the {\it tautological subring}. This is achieved, as in the stable curves case, by defining certain {\it decorated strata classes}. It carries a ring structure by its realization it as a subquotient of the operational Chow ring by an ideal. 

An {\it $X$-valued stable graph} is a dual graph with $H_2(X;\ZZ)^+$-decorations at vertices satisfying the standard Kontsevich stability condition. We decorate it in the following ways: (i) its legs can be given $\psi$-classes and evaluation classes, (ii) its edges with evaluation classes, (iii) its half edges with $\psi$-classes from nodes, and (iv) its vertices with {\it twisted} $\kappa$-classes obtained by  pushing forward powers of $\psi$ and evaluation classes from the universal curve over the corresponding component. These decorated strata classes, viewed as formal symbols, carry a natural multiplication induced by edge contractions, and there is a map from this algebra to Chow cohomology of the stratum of $\Mbar_{g,n}(X,\beta)$ attached to this graph. Capping this with the virtual class of the stratum and pushing forward leads to a natural collection of homology classes. 

By definition, classes of the form above span the {\it tautological subgroup}, and Bae shows there is a natural ring structure~\cite[Section~2]{Bae19}.

\subsubsection{Tautological classes on the Picard stack} We steal Bae's setup for use in the setting of the Picard stack. We let $\mathfrak M_{g,n,d}(B\mathbb G_m)$ denote mapping stack from prestable curves to $B\mathbb G_m$, where the universal line bundle has total degree $d$ on the domain curve. The morphism
\[
\mathfrak M_{g,n,d}(B\mathbb G_m)\to \fM_{g,n}
\]
is locally of finite type, of Artin type, and smooth of relative dimension $g-1$. Given a family of maps from a nodal curve
\[
\begin{tikzcd}
\mathcal C\arrow{d}\arrow{r}{f} & B\mathbb G_m\\
\fM_{g,n,d}(B\mathbb G_m)\arrow[bend left]{u}{s_i}&	
\end{tikzcd}
\]
where $s_i$ for $1\leq i\leq n$ are sections into the smooth locus, we can produce classes on the base as follows. Denote $f \circ s_i$ by $\mathsf{ev}_i$. First, we can consider {\it evaluation classes} given by $\mathsf{ev}_i^\star(t)$ where $t$ is the first Chern class of the universal bundle on $B\mathbb G_m$. Next, we have {\it twisted $\kappa$-classes}. Let $\mathfrak M_{g,n+m,d}'(B\mathbb G_m)$ be the moduli space of $m$ points on the universal curve $\mathcal C$, compactified by bubbling when two points collide and when a point falls into the node. The twisted $\kappa$-classes are defined by taking monomials in $\psi_{n+i}$ and $\mathsf{ev}_{n+i}^\star(t)$, over all $i$, and pushing forward to $\mathfrak M_{g,n,d}(B\mathbb G_m)$. 


Similarly, suppose $\mathfrak S \hookrightarrow \mathfrak M_{g,n,d}(B\mathbb G_m)$ is a closed stratum. After passing to a finite cover, we can assume that the dual graph of the generic fiber of the universal curve is identified with a fixed graph $\Gamma$. For each edge $E$ of $\Gamma$, there is an asscoiated nodal section
\[
s_E\colon \mathfrak S\to \mathcal C.
\]
We produce classes on $\mathfrak S$ by pulling back classes along $f\circ s_E$. We can also, for each half edge $H$ of an edge $E$ of $\Gamma$, associate the $\psi$-class of the normalization of $\mathcal C$ at the marked point corresponding to $H$. The earlier constructions of twisted $\kappa$-classes can also be performed for each vertex of $\Gamma$, i.e. each irreducible component of the generic fiber of $\mathcal C$, and similarly for evaluation classes. 

A {\it decorated stratum} is a closed stratum $\mathfrak S\hookrightarrow \mathfrak M_{g,n,d}(B\mathbb G_m)$ together with {\it vertex, edge, and leg} cohomology decorations, as described above. The pushforward of the associated class to $\fM_{g,n,d}(B\mathbb G_m)$ is a {\it decorated strata class}. 

\begin{definition}
A {\it tautological Chow cohomology class} on $\fM_{g,n,d}(B\mathbb G_m)$ is a linear combination of decorated strata classes. 	
\end{definition}

\subsubsection{GW classes from the universal bundle} The Gromov--Witten-type moduli spaces discussed earlier have proper morphisms
\[
\Mbar_{g,n,d}(\mathcal P/B\mathbb G_m)\to \fM_{g,n,d}(B\mathbb G_m)
\]
that give rise to natural classes on $\fM_{g,n,d}(B\mathbb G_m)$ by the usual pull/push procedure. Since $[\mathbb P^1/\mathbb G_m]$ is the universal $\mathbb P^1$-bundle, there are evaluations from the stack of maps:
\[
\mathsf{ev}_i\colon \Mbar_{g,n,d}(\mathcal P/B\mathbb G_m)\to [\PP^1/\mathbb G_m]
\]
for each marked point. We can pull back cohomology classes from $[\mathbb P^1/\mathbb G_m]$ to obtain evaluation classes. Taking a product of these, capping with $[\Mbar_{g,n,d}(\mathcal P/B\mathbb G_m)]^{\sf vir}$, and pushing forward to $\fM_{g,n,d}(B\mathbb G_m)$ produces {\it GW classes associated to the universal $\mathbb P^1$-bundle}. 

Using the logarithmic variants, $\Mbar_{\Lambda}(\mathcal P|D/B\mathbb G_m)$ of maps to the universal $\mathbb P^1$-bundle with logarithmic structure coming from the sections, we obtain {\it logarithmic} GW classes.

\begin{remark}[Abel--Jacobi theory on the Picard stack]
If we take the universal $\mathbb P^1$-bundle with logarithmic structure along both $0$ and $\infty$, and use the rubber theory outlined in Remark~\ref{rem: uni-rubber}, one obtains classes on the universal Picard stack by pushing forward the virtual fundamental class. One can check in this case that, in fact, the morphism
\[
\Mbar_{\Lambda}(\mathcal P|D/B\mathbb G_m)\to \fM_{g,n,d}(B\mathbb G_m)
\]
is a closed immersion of an irreducible stack -- the stack is actually just the stack of logarithmic maps to $[\mathbb P^1/\mathbb G_m]$ with its toric logarithmic structure, and this follows from~\cite[Proposition~1.6.1]{AW}. The class of the pushforward of the fundamental class of this immersion is essentially the calculation performed in~\cite{BHPSS} from the perspective of Abel--Jacobi theory. The Gromov--Witten perspective is more useful to us here.
\end{remark}

The key theorem about the construction is the following. 
\begin{theorem}\label{thm: universal-bundle}
The logarithmic GW classes on $\fM_{g,n,d}(B\mathbb G_m)$ associated to the universal $\mathbb P^1$-bundle are tautological. 
\end{theorem}

For a $\mathbb P^1$-bundle over a projective manifold $\mathbb P/X$, this means the GW classes of $\PP$, pushed down to $X$, lie in Bae's tautological ring~\cite{Bae19}. This is known by~\cite{MP06}. We will not give a new proof of this result -- we use it to deduce the universal statement. 

We begin with a precise statement of this older result. 

\begin{theorem}\label{thm: tvgw-projective-bundle}
Let $Y$ be a projective manifold and $L$ a line bundle on $Y$. Let $\mathbb P$ be the bundle $\mathbb P(\mathcal O_Y+L)$. The logarithmic GW classes of $\mathbb P$, possibly with logarithmic structure along the $0$ and/or $\infty$ sections, are tautological. 	
\end{theorem}

\begin{proof}
The statement follows from~\cite[Theorem~1]{MP06}. The result there is stated for invariants, but the proof can be followed step-by-step for GW classes as well. To summarize, apply relative virtual localization with respect to the fiber torus of the bundle. The fixed locus is described in terms of stable maps to $X$, viewed as the $0$ of $\infty$ section, with Hodge insertions, or rubber stable maps to a $\mathbb P^1$-bundle, which can either be handled by the inductive method of~\cite{MP06} or using~\cite{JPPZ2} directly. The GW classes are seen to be tautological. 
\end{proof}

\subsubsection{Proof of Theorem~\ref{thm: universal-bundle}}

We fix a particular Gromov--Witten class on the universal bundle, obtained by capping with evaluation classes from $[\mathbb P^1/\mathbb G_m]$. The proof strategy to control this GW class is to reduce the theorem to the statement of Theorem~\ref{thm: tvgw-projective-bundle}. 

The question is insensitive to whether we have logarithmic structure along $\mathcal P$, so we treat the case where $\mathcal P$ has trivial logarithmic structure and leave the cosmetic changes to the reader. The key point is that whether there is logarithmic structure or not, the reduction to the known case stated in Theorem~\ref{thm: tvgw-projective-bundle} is formally identical, at which point we just plug in the appropriate version of the result. 

\noindent
{\sc Step I. Families over an irreducible base.} We rephrase the problem in terms of families of curves. For any family
\[
\begin{tikzcd}
(\mathcal C,\mathcal L)\arrow{d}{\varpi}\\
B	
\end{tikzcd}
\]
of curves and line bundles, we can consider the associated $\mathbb P^1$-bundle, and the mapping space $\Mbar(\varpi)$ of sections of the $\mathbb P^1$-bundle, compactified as previously, by allowing destabilizations of the curve. We have the basic fiber square
\[
\begin{tikzcd}
\Mbar(\varpi)\arrow{d}\arrow{r}& \Mbar(\mathcal P/B\mathbb G_m)\arrow{d}\\
B\arrow{r} & \mathfrak M(B\mathbb G_m). 	
\end{tikzcd}
\]
If we denote these evaluation classes by 
\[
\eta\in \mathsf{CH}^\star(\Mbar(\varpi);\QQ),
\]
we obtain classes on $B$ by pushing forward $\eta\cap [\Mbar(\varpi)]^{\sf vir}$. 

It suffices to show, for every irreducible $B$, that this push forward is given by the pullback of a tautological Chow cohomology class on $\fM(B\mathbb G_m)$ applied to the fundamental class of $B$. 

\noindent
{\sc Step II. Producing a family with lots of sections.} To carry this out, we steal a trick from~\cite{BHPSS}. First, we can find a proper, surjective, generically finite morphism
\[
a\colon B'\to B
\]
and a destabilization $\mathcal C'$ of the pullback family of curves
\[
\mathcal C'\to a^{-1}\mathcal C
\]
together with sections $s_1,\ldots,s_m$ of $\mathcal C'$ satisfying the following properties:
\begin{enumerate}[(i)]
	\item The images of the $s_i$ are disjoint and lie in the smooth locus, and makes the family $\mathcal C'$ into a family of $m$-pointed stable curves. 
	\item The union of the components which do not carry markings form a disjoint union of trees of rational curves, contracted by the destabilization $\mathcal C'\to a^{-1}\mathcal C$.
\end{enumerate}

The existence of such a modified family is guaranteed by~\cite[Lemma~43]{BHPSS}. It is proved by starting with multi-sections over the generic point of $B$ satisfying analogous properties, base changing to make them section, resolving the moduli map from the resulting family to $\Mbar_{g,m}$, and analyzing the resulting destabilization. 

For us, something a little less fine will be easier to work with: we ask for the $s_i$ to lie in the smooth locus, but do not insist on them being disjoint. This can be obtained, for example, by using Hassett spaces, assigning weight $\epsilon$ to the $m$ points and pulling back the universal family along Hassett's weight reduction morphism $\Mbar_{g,m}\to\Mbar_{g,m\cdot\varepsilon}$, see~\cite{Has03}. Fiberwise, the map $\mathcal C'\to a^{-1}\mathcal C$ contracts semistable chains of rational curves, necessarily carrying marked points. 

Since $B'\to B$ is an alteration, it follows from formal properties of the mapping spaces and their virtual classes that in order to prove that the pushforwards of $\eta\cap [\Mbar(\varpi)]^{\sf vir}$ are tautological, it suffices to show the analogous statement for the family
\[
a^{-1}\mathcal C\to B',
\]
equipped with the pullback line bundle. 

\noindent
{\sc Step III. Virtual compatibility.} We further claim that it suffices to prove that the classes obtained from the family $\mathcal C'\to B'$, equipped with the pullback line bundle $\mathcal L'$, lie in the tautological ring of $B'$.  

We have families
\[
\begin{tikzcd}
\mathcal C'\arrow{rr} \arrow[swap]{dr}{q}& & a^{-1}\mathcal C\arrow{dl}{p}\\
&B'.&	
\end{tikzcd}
\]
By the conditions on the contraction, in every geometric fiber over $B'$, the pullback line bundle is trivial along all components of curves in $\mathcal C'$ that are contracted by the destabilization. In other words, the $\mathbb P^1$-bundle is pulled back. 

The mapping spaces associated to $\mathcal C'$ is obtained by taking a (further) destabilization of a fiber of $\mathcal C'\to B'$ and mapping to the projective completion of $\mathcal L'$ over this fiber, such that the newly destabilized components map onto fibers of the $\mathbb P^1$-bundle. 

There are moduli of maps associated to the family of curves $q$ and the family of curves $p$, denoted $\Mbar(q)$ and $\Mbar(p)$ respectively. These both come equipped with virtual classes, and we have a morphism
\[
h\colon \Mbar(q)\to\Mbar(p). 
\]
We claim that pushforward along $h$ identifies virtual fundamental classes. 

This follows immediately from virtual birational invariance in logarithmic GW theory~\cite{AW}. Indeed, the morphism $\mathcal C'\to a^{-1}\mathcal C$ is, by construction, a logarithmic modification. The family of $\mathbb P^1$-bundles are therefore also related by logarithmic modification and so the map
\[
h\colon \Mbar(q)\to\Mbar(p)
\]
is a logarithmic modification compatible with virtual classes. 

We conclude that it suffices to prove that the invariants constructed from maps out of the family $q$ are tautological.

\noindent
{\sc Step IV. Moving to a mapping space.}  By the previous step, it suffices to prove that the classes obtained from the family $\mathcal C'\to B'$ equipped with the line bundle $\mathcal L'$, are tautological. There exist sections passing through every irreducible component of every geometric fiber. If we let $D$ be the sum of all of these sections, we know that on each geometric fiber $C'_b$, for large $n$, the line bundle $\mathcal L'(nD)$ restricts to one with vanishing higher cohomology groups. Indeed, this follows from Riemann--Roch for nodal curves. Of course, similarly $\mathcal O_{C'_b}(nD)$ itself also has vanishing  higher cohomology for $n$ large. 

Let us now explain the use of these sections. On each geometric fiber of the family $\mathcal C'\to B'$, there is a map to $\mathbb P^\ell$, such that the line bundle $\mathcal L(nD)$ is the pullback of $\mathcal O(1)$. Similarly, a map to $\mathbb P^m$ produces $\mathcal O_{C'_b}(nD)$. Putting these together, we can realize the bundle $\mathcal L$ itself by pulling back $\mathcal O(1,-1)$ along a map to $\mathbb P^\ell\times\mathbb P^m$. Furthermore, since the higher cohomology groups vanish, it follows immediately from the Euler sequence that these maps are have unobstructed deformations. The idea is to now deduce the statement about GW classes being tautological on $B'$ from analogous statements for the $\mathbb P^1$-bundle over $\mathbb P^\ell\times\mathbb P^m$ associated to $\mathcal O(1,-1)$. 

We now provide the technical details. Recalling that $q\colon \mathcal C'\to B'$ denotes the family, we can construct the vector bundle
\[
Rq_\star \mathcal L(nD)^{\oplus \ell+1}\oplus Rq_\star \mathcal O_{\mathcal C'}(nD)^{\oplus m+1}
\]
over $B'$. This is a vector bundle over $B'$. We can pass to the locus of sections that give a morphism to $\mathbb P^\ell\times\mathbb P^m$. Note that this is an open subset defined by the condition that the sections of each of these two line bundles, independently, are basepoint free. The basepoint free condition for $\ell+1$ sections is codimension $\ell$, so by forcing {\it both} $\ell$ and $m$ to be large, we can assume that
\[
U_{B'}(\mathbb P^\ell\times\mathbb P^m)\subset Rq_\star \mathcal L(nD)^{\oplus \ell+1}\oplus Rq_\star \mathcal O_{\mathcal C'}(nD)^{\oplus m+1}
\]
is an open complement of a high codimension subset. In particular, by forcing the codimension to be larger than the dimension of $B$, we can ensure the morphism
\[
U_{B'}(\mathbb P^\ell\times\mathbb P^m)\to B'
\]
induces an {\it injective} pullback on Chow homology groups. 

\noindent
{\sc Step V. Conclusion via mapping spaces.} Since the Chow pullback map induced by the projection $U_{B'}(\mathbb P^\ell\times\mathbb P^m)\to B'$ is injective, it suffices to show the tautological property of classes after replacing (i) the family $B'$ by $U_{B'}(\mathbb P^\ell\times\mathbb P^m)$, (ii) replacing the curve family $\mathcal C'$ with its pullback under $U_{B'}(\mathbb P^\ell\times\mathbb P^m)\to B'$, and (iii) replacing $\mathcal L'$ with the pullback of $\mathcal O (1,-1)$ under the universal map to $\mathbb P^\ell\times\mathbb P^m$. 

We can now appeal to the older results of~\cite{MP06}, as stated in Theorem~\ref{thm: tvgw-projective-bundle} to conclude. Precisely, we have a moduli map
\[
U_{B'}(\mathbb P^\ell\times\mathbb P^m)\to \Mbar_{g,n}(\mathbb P^\ell\times\mathbb P^m,\beta).
\]
Since the degrees of the pullbacks of the two tautological bundles have, by construction, vanishing higher cohomology, this map factors through the locus in $\Mbar_{g,n}(\mathbb P^\ell\times\mathbb P^m,\beta)$ parameterizing unobstructed maps, i.e. the open locus on which the obstruction sheaf is trivial and the fundamental class agrees with the virtual class. In particular, by compatibility of the virtual class with open immersions, this means that in order to show that the GW class associated to the family $\mathcal C'\to B'$ and the line bundle $\mathcal L'$ are tautological, it suffices to show that the corresponding GW cycle in $\mathsf{CH}_\star(\Mbar_{g,n}(\mathbb P^\ell\times\mathbb P^m,\beta);\QQ)$ associated to the $\mathbb P^1$-bundle
\[
\mathbb P(\mathcal O(1,-1)\oplus\mathcal O)\to \mathbb P^\ell\times\mathbb P^m.
\]
Indeed, restricting this GW cycle to the open locus of unobstructed maps, and applying Poincar\'e duality, pulling back to the irreducible space $U_{B'}(\mathbb P^\ell\times\mathbb P^m)$ and capping with the fundamental class, we recover the original GW cycle on $U_{B'}(\mathbb P^\ell\times\mathbb P^m)$ that we are trying to control. 

The result now follows from the fact that GW classes associated to $\mathbb P^1$-bundles are tautological, as stated in Theorem~\ref{thm: tvgw-projective-bundle}.

\qed

\section{Logarithmic geometry and the universal projective line}\label{sec: uni-log-bundle} 

\subsection{Goal of the section} Let $(Y|\partial Y)$ be an snc pair and let $(\mathbb P|\partial \PP)$ be a $\mathbb P$-bundle over $Y$ with logarithmic structure given by pulling back $\partial Y$ and adding one or both of the $0/\infty$-sections. The goal of this section is to show that the (logarithmic) pushforwards of GW classes under
\[
\Mbar_\Lambda(\mathbb P|\partial \PP)\to \Mbar_\Lambda(Y|\partial Y)
\]
can be expressed in terms of tautological operators, acting on the virtual class of $\Mbar_\Lambda(Y|\partial Y)$. This will prove Theorem~\ref{thm: tvgw-projective-bundle}. 

The cohomology of $\mathbb P$ satisfies a Leray--Hirsch formula, both in Chow and in cohomology. This means that any insertion in a GW pushforward can be written in terms of a pullback class on $Y$, times a power of the Chern class of the relative hyperplane bundle on $\mathbb P\to Y$. Since we are pushing forward to $\Mbar_\Lambda(Y|\partial Y)$, we can actually assume all the evaluation classes are Chern classes of the relative hyperplane bundle. As noted earlier, we will study more general insertions from broken $\mathbb P^1$-bundles in Section~\ref{sec: compressed-evaluations}.

Both $\mathbb P\to Y$ and the insertions, can therefore be assumed to be pulled back from the universal situation 
\[
[\mathbb P^1/\mathbb G_m]\to B\mathbb G_m.
\]
In particular, given any GW pushforward class with this class of insertions, there is a ``precursor'' problem on the space $\Mbar_\Lambda(\mathcal P|D/B\mathbb G_m)$, and we have control over this class by Theorem~\ref{thm: universal-bundle}.

The goal of the section is to compare the class on $\Mbar_\Lambda(Y|\partial Y)$ with this universal class. This relation is less straightforward than when $\partial Y$ is empty, because the natural diagram of logarithmic mapping spaces
\[
\begin{tikzcd}
    \Mbar_\Lambda(\mathbb P|\partial \PP)\arrow{r}\arrow{d} & \Mbar_\Lambda(\mathcal P|D/B\mathbb G_m)\arrow{d} \\
    \Mbar_\Lambda(Y|\partial Y)\arrow{r} & \fM_{g,n}(B\mathbb G_m). 
\end{tikzcd}
\]
is not a cartesian diagram of algebraic stacks, but only in the logarithmic category. Therefore one cannot just pull back the relevant universal classes from $\fM_{g,n}(B\mathbb G_m)$ to  $\Mbar_\Lambda(Y|\partial Y)$. 

We will show the analogous statement {\it does} hold in the logarithmic Chow ring, or equivalently, after a sufficient blowup. Logarithmic enhancements of such classes have been previously studied in nearby contexts~\cite{HMPPS,HS21,MR21}. The general framework outlined in~\cite[Section~3]{MR21} is particularly relevant for us.

\subsection{Logarithmic Chow lifts} We construct a lift of the universal bundle GW classes to logarithmic Chow and then prove they are tautological. 

\subsubsection{Setting up the problem} We return to the space $\Mbar_\Lambda(\mathcal P|D/B\mathbb G_m)$ of maps to the universal $\mathbb P^1$-bundle over $\fM_{g,n,d}(B\mathbb G_m)$. The logarithmic structure $D$ will either or both of the $0$ or $\infty$ sections\footnote{When $\mathcal P$ is given the trivial logarithmic structure, the logarithmic Chow upgrade is the trivial one.}.

We have an evaluation map
\[
\mathsf{ev}\colon \Mbar_\Lambda(\mathcal P|D/B\mathbb G_m)\to \mathcal P^k\times B\mathbb G_m^{n-k}
\]
where $k$ among the $n$ markings have tangency order $0$ and the remaining have positive tangency with one of the divisors. For the whole discussion, we fix a cohomology class 
$$
\eta\in \mathsf{CH}^\star(\mathcal P^k\times B\mathbb G_m^{n-k})
$$
and refer to it as the {\it evaluation class}. We will define a logarithmic enhancement of the Gromov--Witten class
\[
\mathsf{ev}^\star(\eta)\cap [\Mbar_\Lambda(\mathcal P|D/B\mathbb G_m)]^{\sf vir}
\]
and its pushforward to the Picard stack $\fM_{g,n,d}(B\mathbb G_m)$. 

\subsubsection{Defining the logarithmic enhancement} In order to deal with the lack of finiteness of the moduli space, we slightly modify the Picard stack and the mapping space as follows. 

\noindent
{\bf Curve class semigroup.} Let $\mathbb H$ be a finitely generated, torsion free commutative semigroup with the following two properties: (i) the element $0$ cannot be written non-trivially as a sum of two elements, and (ii) any element in $\mathbb H$ can only be written as a sum of two elements in $\mathbb H$ in finitely many ways. We refer to this as a {\it curve class semigroup}. The relevant example is the semigroup of {\it effective} curve classes on a projective variety $X$, where these properties are easily verified.

\noindent
{\bf Chern class functional.} Given $\mathbb H$, we fix a $\mathbb Z$-valued homomorphism
\[
c \colon \mathbb H\to \mathbb Z
\]
We refer to it as the {\it Chern class functional}. A typical example is given by taking $\mathbb H$ to be the semigroup of effective curve classes on a projective variety $Y$ and $c$ is defined by taking intersection product with the first Chern class of some fixed line bundle on $Y$. 

Given a map from a nodal curve
\[
C\to Y
\]
to $Y$ and a line bundle $L$ on $Y$, we obtain curve class decorations on the components by looking at the curve class, e.g. in $H_2(Y;\ZZ)$. The Chern class functional is given by $c_1(L)\cdot \beta_V$. 

We make the following definition, inspired by Costello~\cite{Cos06}. See also~\cite{BS22}. 

\begin{definition}
	Fix a monoid $\mathbb H$ and Chern class function $c$ on it, as above. The {\it $(\mathbb H,c)$-decorated Picard stack} is the stack
	\[
	\fM_{g,n,d}(B\mathbb G_m,\mathbb H,c)\to \fM_{g,n,d}(B\mathbb G_m)
	\]
	parameterizing pairs $(C,L)$ of curves with a line bundle, together with a decoration of each irreducible component by an element of $\mathbb H$, such that additionally, at each component $C_V$ of $V$ with decoration $\beta_V$, we have an equality
	\[
	c(\beta_V) = \deg(L|_{C_V}).
	\]
	The {\it total curve class} of a decorated curve is the sum of all the component decorations. We say a decorated curve is {\it $\mathbb H$-stable} if every genus $0$ component with zero degree decoration has at least three distinguished points. 
	
	With $(\mathbb H,c)$ fixed from context, we let $\Mbar_{g,n,d}(B\mathbb G_m,\beta)$ denote the locus of $\mathbb H$-stable curves with degree $d$ line bundles having total curve class $\beta$. 
\end{definition}

Note that despite the use of the word {\it stable} in $\mathbb H$-stable, we are not making any claims about properness of this space.

We now fix $\mathbb H$, the function $c$, and the total curve class $\beta$ for the rest of the discussion in the section. The following proposition is standard. 

\begin{proposition}
	The stack $\Mbar_{g,n,d}(B\mathbb G_m,\beta)$ is of finite type and the morphism
	\[
	\Mbar_{g,n,d}(B\mathbb G_m,\beta)\to \fM_{g,n,d}(B\mathbb G_m)
	\]
	is \'etale.
\end{proposition}

The morphism 
\[
\mathcal P\to B\mathbb G_m
\] 
admits a {\it relative Artin fan}, given by
\[
\mathcal P\to \mathsf A\times B\mathbb G_m\to B\mathbb G_m,
\]
where $\mathsf A$ is either $[\mathbb A^1/\mathbb G_m]$ or $[\mathbb P^1/\mathbb G_m]$, depending on whether $D$ has one or two components, respectively. We equip these toric stacks with their toric logarithmic structure. With this understood, the first arrow above is {\it strict and smooth of relative dimension $1$}, while the second arrow is logarithmically smooth of relative dimension $0$, but not strict. 

\begin{remark}We warn the reader that when there are two divisors, this copy of $[\mathbb P^1/\mathbb G_m]$ in the relative Artin fan should not be confused with $\mathcal P$, which we view as positive dimensional over $B\mathbb G_m$. \end{remark}

We can pull back the sequence of logarithmic mapping 
\[
\Mbar_\Lambda(\mathcal P|D/B\mathbb G_m)\to \fM_{\Lambda}(\mathsf A\times B\mathbb G_m)\to \fM_{g,n,d}(B\mathbb G_m). 
\]
along this \'etale map to obtain
\[
\Mbar_\Lambda(\mathcal P|D,\beta)\to \Mbar^{\sf ub}_\Lambda (\mathsf A\times B\mathbb G_m,\beta)\to \Mbar_{g,n,d}(B\mathbb G_m,\beta).
\]

The composite of the two maps is proper, and in particular, of finite type, but the second arrow alone, at this stage, is not. This can be fixed by imposing the {\it balancing condition}. Concretely, an object in $\Mbar^{\sf ub}_\Lambda (\mathsf A\times B\mathbb G_m,\beta)$ consists of the following data:
\begin{enumerate}[(i)]
	\item nodal curve $C$ with dual graph $G_C$,
	\item a line bundle $L$ on $C$ of degree $d$,
	\item an $\mathbb H$-decoration on the vertices of $\Gamma_C$ of total degree $d$, whose $c$-value at each vertex is the degree of $L$,
	\item a combinatorial type of piecewise linear maps $\Gamma\to\mathsf A^{\sf trop}$ such that $\Gamma$ has underlying graph $G_C$. Note that $\mathsf A^{\sf trop}$ is either $\mathbb R_{\geq 0}$ or $\RR$, depending on the Artin fan. 
\end{enumerate}

Note that the item (iv) comes from the logarithmic map $C\to\mathsf A$ implicit in a point of $\Mbar^{\sf ub}_\Lambda (\mathsf A\times B\mathbb G_m,\beta)$. Indeed, a logarithmic map to an Artin fan is a piecewise linear map between the respective tropicalizations, see for instance~\cite[Section~2]{ACGS15}. 

\begin{definition}[Universal balancing]
	With the notation above, a point of $\Mbar^{\sf ub}_\Lambda (\mathsf A\times B\mathbb G_m,\beta)$ is {\it balanced} if at every vertex $V$ of the tropical map $\Gamma\to\mathsf A^{\sf trop}$, the sum of the outgoing slopes is equal to the degree of $L|_{C_V}$. The open substack of balanced maps will be denoted $\Mbar_\Lambda (\mathsf A\times B\mathbb G_m,\beta)$. 
\end{definition}

By similar arguments to~\cite{AW} and the boundedness results in~\cite{GS13}, we make the following conclusions about the triple of spaces 
\[
\Mbar_\Lambda(\mathcal P|D,\beta)\to \Mbar_\Lambda (\mathsf A\times B\mathbb G_m,\beta)\to \Mbar_{g,n,d}(B\mathbb G_m,\beta).
\]
\begin{enumerate}[(i)]
	\item The space $\Mbar_\Lambda (\mathsf A\times B\mathbb G_m,\beta)$ is of finite type,
	\item the first arrow is now {\it strict and admits a perfect relative obstruction theory} with obstruction bundle given by $R^\bullet \pi_\star s^\star T^{\sf log}_{\mathcal P/\mathfrak M}$, as in Section~\ref{sec: univ-log-theory}. 
	\item The second arrow is logarithmically \'etale with smooth codomain. 
\end{enumerate}

With this preparation we can construct the logarithmic enhancements of the Gromov--Witten virtual classes. 

Endow $\Mbar_{g,n,d}(B\mathbb G_m,\beta)$ with the logarithmic structure coming from the components of the divisor of singular curves. The components are indexed by graphs with a single edge, together with a $\mathbb H$-decoration. In other words, we take the pullback of the standard boundary structure on $\frak M_{g,n}$, and take the logarithmic structure given by its irreducible components. 

We consider a logarithmic blowup
\[
\Mbar^\diamond_{g,n,d}(B\mathbb G_m,\beta)\to \Mbar_{g,n,d}(B\mathbb G_m,\beta).
\]
For ease of exposition, and because we lose no generality, we assume this is a sequence of blowups along smooth centers. 

We now perform two constructions: we can either pullback $\Mbar_\Lambda (\mathsf A\times B\mathbb G_m,\beta)\to \Mbar_{g,n,d}(B\mathbb G_m,\beta)$ along this blowup in the category of algebraic stacks, or alternatively we can pull back in the fine and saturated logarithmic category. The latter is the {\it strict} transform. Pulling back the mapping space $\Mbar_\Lambda(\mathcal P|D,\beta)$ as well, we end up with the following diagram:
\[
\begin{tikzcd}
\Mbar^\diamond_\Lambda(\mathcal P|D,\beta)\arrow{r}\arrow{d} & \mathsf F_\Lambda(\mathcal P|D,\beta) \arrow{r}\arrow{d} & \Mbar_\Lambda(\mathcal P|D,\beta)\arrow{d}\\
\Mbar^\diamond_\Lambda (\mathsf A\times B\mathbb G_m,\beta) \arrow{r} & \mathsf F_\Lambda (\mathsf A\times B\mathbb G_m,\beta) \arrow{r}\arrow{d} & \Mbar_\Lambda (\mathsf A\times B\mathbb G_m,\beta)\arrow{d}\\
 & \Mbar^\diamond_{g,n,d}(B\mathbb G_m,\beta)\arrow{r}& \Mbar_{g,n,d}(B\mathbb G_m,\beta).
\end{tikzcd}
\]
We make a few notes about the diagram:
\begin{enumerate}[(i)]
	\item All three squares are Cartesian. 
	\item The vertical arrows in the first row are equipped with compatible perfect obstruction theories.
	\item All horizontal arrows are proper. The horizontals in the left square are finite\footnote{In fact, if we work in the category of fine rather than fine and saturated logarithmic schemes, which is essentially harmless, they are closed immersions.}.
	\item The space $\mathsf F_\Lambda (\mathsf A\times B\mathbb G_m,\beta)$ is the pullback of an irreducible space along a birational map; it is typically reducible. The space $\Mbar^\diamond_\Lambda (\mathsf A\times B\mathbb G_m,\beta)$ is its strict transform, or ``main component''. 
	\item The composition of the two vertical arrows in the second and third columns are proper. 
\end{enumerate}

We view $\Mbar^\diamond_\Lambda(\mathcal P|D,\beta)$, or more properly the virtual cycle given by the pullback of the fundamental class of $\Mbar^\diamond_\Lambda(\mathsf A\times B\mathbb G_m,\beta)$ as a ``virtual strict transform'' of the virtual cycle $[\Mbar_\Lambda(\mathcal P|D,\beta)]^{\sf vir}$.

We record names for the compositions:
\[
p\colon \Mbar_\Lambda(\mathcal P|D,\beta)\to \Mbar_{g,n,d}(B\mathbb G_m,\beta) \ \ \ \ p^\diamond\colon \Mbar^\diamond_\Lambda(\mathcal P|D,\beta)\to \Mbar^\diamond_{g,n,d}(B\mathbb G_m,\beta).
\]
We are finally ready to name the logarithmic lifts. 

\begin{definition}[Logarithmic lifts]
	Given a blowup $\Mbar^\diamond_{g,n,d}(B\mathbb G_m,\beta)\to \Mbar_{g,n,d}(B\mathbb G_m,\beta)$ and a Gromov--Witten cycle class given by
	\[
	p_\star(\mathsf{ev}^\star(\eta)\cap [\Mbar_\Lambda(\mathcal P|D,\beta)]^{\sf vir})
	\]
	its {\it logarithmic enhancement} is the system of classes in the Chow groups of the system of all logarithmic blowups, whose value on the Chow group $\Mbar^\diamond_{g,n,d}(B\mathbb G_m,\beta)$ is given by
	\[
	p^\diamond_\star(\mathsf{ev}^\star(\eta)\cap [\Mbar^\diamond_\Lambda(\mathcal P|D,\beta)]^{\sf vir}).
	\]
\end{definition}

A priori this merely defines a class on each blowup $\Mbar^\diamond_{g,n,d}(B\mathbb G_m,\beta)$. By compatibility of the obstruction theories in the big diagram above, these classes are identified under pushforward~\cite[Section~3]{R19}. They define a class in the inverse limit of Chow homology groups, with transition maps given by proper pushforward. More is true.

\begin{proposition}
	The logarithmic enhancements of the GW class $p_\star(\mathsf{ev}^\star(\eta)\cap [\Mbar_\Lambda(\mathcal P|D,\beta)]^{\sf vir})$ are eventually compatible under pullback. In other words, there exists a blowup $\Mbar_{g,n,d}^\vartriangle(B\mathbb G_m,\beta)$ such that for all further blowups, the GW classes are related by cohomological pullback. 
\end{proposition}

\begin{proof}
The proof is essentially identical to~\cite[Section~3.2]{RUK22}, and has been promised in a forthcoming survey by Herr, Molcho, Pandharipande, and Wise. To adapt the argument from~\cite{RUK22}, replace the moduli space of stable curves there with $\Mbar_{g,n,d}(B\mathbb G_m,\beta)$. Once a sufficient blowup of $\Mbar_{g,n,d}(B\mathbb G_m,\beta)$ is made, replacing $\Mbar_{g,n,d}(B\mathbb G_m,\beta)$ in the big diagram above by this blowup, and considering further blowups, we can assume the right-vertical arrow in the bottom square is flat. This in turn makes the horizontal arrows in the top left square isomorphisms. By flat base change for the cotangent complex, the two possible different lci pullbacks are identified. An elementary diagram chase yields the result. 
\end{proof}

\subsection{Tautological property for enhanced classes} The next goal is to control ths logarithmic enhancement, and show it is ``logarithmically tautological''. 

We have distinguished a collection of tautological classes in the Chow cohomology of the stack $\fM_{g,n,d}(B\mathbb G_m)$. Replacing the dual graphs in that construction with the degree decorated graphs of $\Mbar_{g,n,d}(B\mathbb G_m,\beta)$ we obtain an analogous tautological ring\footnote{This is not quite the same as the pullback of the tautological ring from the undecorated space -- the latter may contain classes that can be written as a sum of tautological classes in the sense we have just outlined, without the summands being tautological.}. By definition, the strata of this space themselves have a notion of tautological class coming from their description as a fiber product.

For each blowup 
\[
\Mbar_{g,n,d}(B\mathbb G_m,\beta)^\diamond\to \Mbar_{g,n,d}(B\mathbb G_m,\beta)
\]
along a sequence of smooth strata, we have a pullback map in Chow. The logarithmic Chow ring is the direct limit:
\[
\mathsf{CH}^\star_{\sf log}(\Mbar_{g,n,d}(B\mathbb G_m,\beta)) = \varinjlim \mathsf{CH}^\star_{\sf log}(\Mbar_{g,n,d}(B\mathbb G_m,\beta)^\diamond)
\]
with transitions given by pullback. 

We would like to lift the tautological ring of $\Mbar_{g,n,d}(B\mathbb G_m,\beta)$ to this direct limit in a sensible way. This can now be done formally. In~\cite{PRSS}, a formalism is presented that takes as input a smooth and logarithmically smooth pair $(X,D)$, together with a system of subrings of Chow rings of all its strata closed under pushforward -- a ``Chow subsystem''. These can then be upgraded to a subring of logarithmic Chow rings.

The output is a {\it logarithmic tautological ring} for $\Mbar_{g,n,d}(B\mathbb G_m,\beta)$. We only need a few properties about this ring, which follow from the construction. First, it contains the usual tautological ring. Second, the ring is closed under the following iterative reciple: start with a tautological class on a stratum, pull back to the stratum of some blowup of the ambient that maps to this stratum, and push forward.

\begin{theorem}\label{thm: log-enhancement}
	The logarithmic enhancement of the GW class
	\[
	p_\star(\mathsf{ev}^\star(\eta)\cap [\Mbar_\Lambda(\mathcal P|D,\beta)]^{\sf vir})
	\]
	lies in the logarithmic tautological ring of $\Mbar_{g,n,d}(B\mathbb G_m,\beta)$.
\end{theorem}

The structure of the proof is similar (and in fact, a little simpler) than the exotic/non-exotic manoeuvre in Section~\ref{sec: exotic-nonexotic-move}.  We apply Fulton's blowup formula, use Aluffi's formula for Segre classes, and induct on the strata complexity~\cite{Alu16,Ful98}. 

We require two preliminaries, after which the proof is formal: the data type indexing the stratification of the boundary of the universal mapping space and the decomposition of each stratum and its virtual class into ``inductively smaller'' pieces. 

\subsubsection{Strata, blowups, and normal bundles} We explain the stratification of $\Mbar_\Lambda(\mathcal P|D,\beta)$ and its subdivisions. For expository simplicity, we suppose $D$ is the universal $0$-section in $\mathcal P$. The case where we take both sections is similar. 

We begin with $\Mbar_\Lambda(\mathcal P|D,\beta)$. It is convenient to work with an ``expanded'' version of $\Mbar_\Lambda(\mathcal P|D,\beta)$, which can be obtained by the methods of~\cite{R19}. A point of $\Mbar_\Lambda(\mathcal P|D,\beta)$ determines, among other things, a logarithmic morphism 
$$
C\to [\mathbb P^1/\mathbb G_m].
$$ 
By adding in the map to the Artin fan of $[\mathbb P^1/\mathbb G_m]$ this gives:
\[
C\to [\mathbb P^1/\mathbb G_m]\to\mathsf A.
\]
Consider the universal $\mathcal C$ curve over $\Mbar_\Lambda(\mathcal P|D,\beta)$ and project to $\mathsf A$. We can apply the combinatorial and virtual semistable reduction of~\cite[Section~2 {\it \&} 3]{R19}: blowup the source and target of the map
\[
\mathcal C\to \mathsf A\times \Mbar_\Lambda(\mathcal P|D,\beta)
\]
so it meets strata in the expected dimension. The target is now a relative Artin fan over $\Mbar_\Lambda(\mathcal P|D,\beta)$. Flatten the resulting maps to $\Mbar_\Lambda(\mathcal P|D,\beta)$. As noted in~\cite[Section~2.8]{R19}, there is a canonical such modification in this case. 

This blows up the moduli space to 
\[
\Mbar^{\sf exp}_\Lambda(\mathcal P|D,\beta)\to \Mbar_\Lambda(\mathcal P|D,\beta).
\]
It now carries, as universal families, a modified universal curve, an expanded family of $[\mathbb P^1/\mathbb G_m]$'s, say $[\mathbb P^1/\mathbb G_m]^\sim$. The equidimensionality condition means that the map to this expansion of $[\mathbb P^1/\mathbb G_m]$ is dimensionally transverse along strata.

We have, by pullback, a universal $\mathbb P^1$-bundle over the universal curve, say $\mathcal P$, and by pulling back $[\mathbb P^1/\mathbb G_m]^\sim\to [\mathbb P^1/\mathbb G_m]$ to obtain, over every point in this blowup, a map to an expansion of it:
\[
C\to \widetilde{\mathcal P}\to \mathcal P.
\]
Note that we have only made $C\to\widetilde {\mathcal P}$ partially transverse -- points of $C$ can still map to the corners of the reducible surface $\widetilde {\mathcal P}$\footnote{This type of geometry is not uncommon but hides in plain sight. For example, given an admissible cover or relative stable map to a curve, the map to the graph has exactly this form. Indeed, if the projective line bundle here is trivial, we exactly recover this situation.}.

The space is stratified by the same data structure as the combinatorial types as relative stable maps, recorded by the following data associated to a map $[C\to{\widetilde{\mathcal P}}\to\mathcal P]$ to an expansion:
\begin{enumerate}[(i)]
	\item The dual graph $\Gamma$ of the marked domain curve $C$,
	\item the dual graph $T$ of the expanded Artin fan $[\PP^1/\mathbb G_m]^\sim$, inducing the expansion of the target ${\widetilde{\mathcal P}}$ along the horizontal boundary; this a line graph, ordered from left to right, with a single unbounded edge on the right, 
	\item a morphism of graphs $\Gamma\to T$, induced by recording, for each vertex corresponding to an irreducible component of $C$ the vertex corresponding to the component of the target to which its generic point maps,
	\item for each edge of $\Gamma$ mapping to an edge of $T$, the ramification order at the corresponding node. If an edge maps to a vertex, we set the integer to be $0$. This integer is the {\it expansion factor}
\end{enumerate}

Each combinatorial type $\Theta = [\Gamma\to T]$ determines a cone $\sigma_\Theta$, corresponding to the moduli of possible edge length assignments to $\Gamma$ and $T$, such that the induced map of graphs with the given slopes is continuous. See Figure~\ref{fig: combinatorial-type} for a visualization.

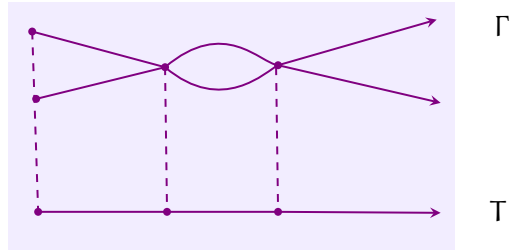
\begin{figure}
    
\begin{tikzpicture}[x=0.75pt,y=0.75pt,yscale=-1,xscale=1]

  \definecolor{softlavender}{RGB}{223,210,255}

  \fill[softlavender, opacity=0.4] (170,260) rectangle (395,385);

  \tikzset{mylines/.style={violet, line width=0.75pt}}

  \draw[mylines] (185,366) -- (250,366);
  \draw[mylines] (182,275) -- (249,293);
  \draw[mylines] (184,309) -- (249,293);
  \draw[mylines] (250,366) -- (306,366);

  \draw[mylines] (249,293) .. controls (279,267.5) and (297,291) .. (306,292);
  \draw[mylines] (249,293) .. controls (278,317.5) and (300,295) .. (306,292);

  \draw[mylines, ->, >=stealth] (306,292) -- (386,269);
  \draw[mylines, ->, >=stealth] (306,292) -- (388,311);
  \draw[mylines, ->, >=stealth] (306,366) -- (388,366.5);

  \draw[mylines, dashed] (182,275) -- (185,366);
  \draw[mylines, dashed] (249,293) -- (250,366);
  \draw[mylines, dashed] (305,293) -- (306,366);

  \foreach \pt in {
    (185,366), (250,366), (182,275), (184,309),
    (249,293), (306,292), (306,366)
  } {
    \fill[violet] \pt circle (1.5pt);
  }

  \node at (419,272) {\(\Gamma\)};
  \node at (417,365.5) {\(T\)};  

\end{tikzpicture}
\caption{A visualization of the combinatorial type of a point in $\Mbar_\Lambda^{\sf exp}(\mathcal P|D,\beta)$. This corresponds to a codimension $2$ stratum in the moduli space.}\label{fig: combinatorial-type}
\end{figure}

Ranging over all combinatorial types, we get a collection of cones. These glue to form a cone stack, which we denote $T^{\sf exp}_\Lambda(\mathcal P|D,\beta)$. We refer the reader to~\cite{CCUW,R19} for similar constructions and additional details. 

\begin{remark}[An ordering on stars]\label{rem: ordering}
	There is a natural ordering on the combinatorial data above, obtained from the specialization relation. We do this similarly to Section~\ref{sec: stars-partial-ordering}. The only difference is that we keep track of the genus. 
	
	We define a {\it star} to be a dual graph $[\Gamma\to T]$ where $\Gamma$ has a single vertex, and $T$ is, as a consequence of stability, a single vertex with either one two two rays attached. Given a combinatorial type $[\Gamma\to T]$ as above, each vertex of $\Gamma$ determines a ``star''. We then define the order as before -- one star $\mathsf v$ is smaller than $\mathsf w$ if $\mathsf v$ appears as the star of a vertex of a degeneration of $\mathsf w$.
\end{remark}

We will be interested in logarithmic blowups of $\Mbar^\diamond_\Lambda(\mathcal P|D,\beta)\to\Mbar^{\sf exp}_\Lambda(\mathcal P|D,\beta)$ coming from subdivisions of $T^{\sf exp}_\Lambda(\mathcal P|D,\beta)$. A stratum is determined by a cone in a subdivision
\[
T^\diamond_\Lambda(\mathcal P|D,\beta)\to T^{\sf exp}_\Lambda(\mathcal P|D,\beta).
\]
Choose such a cone $\sigma$ and let $\overline \sigma$ be the cone to which it maps, under the subdivision above. Let $\Mbar(\overline\sigma)$ be the closed stratum of $\Mbar_\Lambda(\mathcal P|D,\beta)$ corresponding to $\overline \sigma$, and $\Mbar^\diamond(\sigma)$ the stratum on the blowup.

\begin{lemma}\label{lem: normal-1}
The natural morphism
\[
\Mbar^\diamond(\sigma)\to \Mbar(\overline\sigma)
\]	
is an equivariant compactification of a torus bundle.
\end{lemma}

\begin{proof}
The morphism $\Mbar^\diamond_\Lambda(\mathcal P|D,\beta)\to \Mbar^{\sf exp}_\Lambda(\mathcal P|D,\beta)$ is a subdivision. Such morphisms are, locally on the target, pullbacks of proper birational maps of toric varieties. The claim follows from the basic structure theory of toric morphisms.
\end{proof}

We need some control over various normal bundles to these strata. Each stratum of $\Mbar^\diamond_\Lambda(\mathcal P|D,\beta)$ is the pullback of a stratum on the Artin fan corresponding to $T^\diamond_\Lambda(\mathcal P|D,\beta)$. This Artin fan is a smooth and logarithmically \'etale stack, and so this stratum has a well-defined normal bundle. Its pullback is called the {\it virtual normal bundle} of the stratum. 

\begin{lemma}\label{lem: normal-2}
	The line bundle factors of the virtual normal bundle to $\Mbar^\diamond(\sigma)$ in $\Mbar^\diamond_\Lambda(\mathcal P|D,\beta)$ is contained in the subgroup of the Picard group of $\Mbar^\diamond(\sigma)$ generated by (i) cotangent line bundles associated to nodes/markings of the domain curve, and (ii) classes of boundary divisors. 
\end{lemma}

\begin{proof}
First consider the stratum $\Mbar(\overline\sigma)$ on the non-blown up space. It has a combinatorial type $\Theta = [\Gamma\to T]$. The normal bundle of the stratum is the direct sum of line bundles corresponding to (i) smoothing parameters of the nodes of the target expansion, corresponding to each edge of $T$, and (ii) smoothing parameters for the nodes of $C$ corresponding to edges in $\Gamma$ whose expansion factor is $0$. Since the virtual normal bundle is pulled back from the Artin fan, this again follows from toric computations. 

Now, we examine the smoothing parameter for a target edge. Working on the Artin fan, this corresponds to a piecewise linear function on  $T^\diamond_\Lambda(\mathcal P|D,\beta)$ that, on the cone $\sigma$, is equal to the length of one of the target edges. It is therefore, up to a multiple, equal to the length of any edge of the domain tropical curve that maps to this edge. Choose such an edge. Away from higher dimensional cones containing $\overline \sigma$, this piecewise linear function is equal to the pullback of the piecewise linear function on the moduli stack of all tropical curves that measures the length of this edge. Therefore, up to twisting by boundary strata, this target smoothing parameter is equal to the smoothing parameter for the domain. These are easily seen to be of the required form. 

For strata on blowups, the statement follows by induction on the number of blowups. For each blowup, a stratum is either a blowup of a stratum on a smaller blowup, or a projective bundle over a stratum. In the former case, the normal bundle of the new stratum is the pullback of the old normal bundle, twisted by strata. In the latter case, the normal bundle is the relative $\mathcal O(-1)$ on the projective bundle. The Chern class of this line bundle can be represented by strata, so we conclude. 
\end{proof}

\subsubsection{Inductive structure of the boundary} We recall the description of the boundary strata of the space of relative stable maps. These facts are well known for $\mathbb P^1$-bundles over schemes, and we closely follow~\cite{GV05}. 

We first work on the space $\Mbar^{\sf exp}_\Lambda(\mathcal P|D,\beta)$ and choose a combinatorial type $\Theta = [\Gamma\to T]$, or equivalently, a cone $\overline \sigma$ of the tropical space $T^{\sf exp}_\Lambda(\mathcal P|D,\beta)$. This determines a stratum $\Mbar(\overline\sigma)$. The domain graph $\Gamma$, together with its natural decorations, determines a stratum $\Mbar_\Gamma(B\mathbb G_m)$ of the decorated Picard stack, and the natural morphism to the Picard stack factorizes through
\[
\Mbar(\overline\sigma)\to \Mbar_\Gamma(B\mathbb G_m).
\]
We will be interested in the pushforward of the virtual class to this substack, and related classes. We first describe how to reconstruct the space $\Mbar(\overline\sigma)$ from smaller spaces, and then describe how this related to the corresponding decomposition of $\Mbar_\Gamma(B\mathbb G_m)$ into smaller spaces.

Let $V$ be a vertex of the target graph $T$ and let $V_1,\ldots, V_m$ be the vertices of $\Gamma$ mapping to it. Each vertex $V_i$ determines a discrete data for a stack of maps. If $V$ is a vertex corresponds to a bubble component, then this determines a stack of maps of the form
\[
C\to [\mathbb P^1/\mathbb G_m]
\]
relative to both $0$ and $\infty$. For each $V$, we distinguish, arbitrarily, a vertex over it, which we may as well denote $V_1$. For such $V$, corresponding to bubble components, we obtain {\it a stack associated to $V$} is defined to be 
\[
\Mbar(V):=\Mbar(V_1)^{\sf rub}\times\prod_{i=2}^m\Mbar(V_i).
\]
The first factor $\Mbar(V_1)^{\sf rub}$ is the stack of to an unparametrized, or ``rubber'' $\mathbb P^1$-bundle, see Remark~\ref{rem: uni-rubber}.

Similarly, if $V$ is the distinguished vertex of the target graph $T$ corresponding to the main component, we take the analogous product, but where all spaces are parameterized. 

Say an edge $E$ of $\Gamma$ is {\it horizontal} if expansion factor is nonzero, and vertical otherwise. At each horizontal edge we $E$ incident to $V_i$ have evaluation maps
\[
\mathsf{ev}_E\colon \Mbar(V_i)\to B\mathbb G_m.
\]
For vertical edges, there is also always this evaluation to $B\mathbb G_m$. However, if $E$ is incident to $V_i$ and $V_i$ is not parameterized\footnote{This could either be because it is attached to a vertex mapping to the main vertex, or because $i\neq 1$, i.e. $V_i$ is not the distinguished vertex.} there is a more refined evaluation map to $\mathcal P$.

There is a morphism
\[
\Mbar(\overline\sigma)\to \prod_V \Mbar(V)
\]
where $V$ ranges over the vertices of the target graph $T$ associated to the combinatorial type $\overline \sigma$. Note that both domain and target have natural virtual classes.

Note that unlike in situations where we work with schematic targets, this morphism is not a closed immersion, or even finite. Nevertheless, we can relate the virtual classes on the two sides. 

Let $\mathfrak E$ denote the consolidated evaluation space, i.e. the product taken over all pairs of $(V_i,E)$, where $V_i$ is a vertex of $\Gamma$ and $E$ is an edge incident to it, of the evaluation maps $\mathsf{ev}_E$ described above. 

\begin{proposition}\label{prop: diagonal}
	The space $\Mbar(\overline\sigma)$ maps to the fiber product below:
	\[
	\begin{tikzcd}
	\Mbar(\overline\sigma)\arrow{r}{h}&\mathsf G(\overline \sigma)\arrow{d}\arrow{r}&	\prod_V \Mbar(V)\arrow{d}\\
	&\mathfrak E\arrow[swap]{r}{\Delta}&\mathfrak E^2.
	\end{tikzcd}
	\]
	The space $\mathsf G(\overline\sigma)$ is endowed with a virtual class by diagonal pullback. The pushforward of the virtual class along $h$ identifies virtual classes up to a positive rational multiple. 
\end{proposition}

\begin{proof}
The proof is identical to analogous results for schematic targets~\cite{GV05,Li01}.
\end{proof}

Recall that our goal is to study the pushforward of GW classes from $\Mbar(\overline \sigma)$ to the stratum $\Mbar_\Gamma(B\mathbb G_m)$ of the (decorated) Picard stack. To that end, we note that the stack $\Mbar_\Gamma(B\mathbb G_m)$ has an analogous decomposition over the vertices. For each vertex $W$ of $\Gamma$, we have a Picard stack stack $\Mbar_W(B\mathbb G_m)$, and
\[
\Mbar_W(B\mathbb G_m) \to B\mathbb G_m^{E(W)},
\] 
where $E(W)$ is the set of edges incident to $W$ in $\Gamma$. Standard properties of mapping stacks then guarantee the following:

\begin{proposition}\label{prop: pic-strata}
	There is an identification of the stratum $\Mbar_\Gamma(B\mathbb G_m)$ as a fiber product:
\[
\begin{tikzcd}
\Mbar_\Gamma(B\mathbb G_m)\arrow{r}\arrow{d} & \prod_W \Mbar_W(B\mathbb G_m)\arrow{d}\\
\prod_E B\mathbb G_m \arrow{r} & \prod_W	 B\mathbb G_m^{E(W)}.
\end{tikzcd}
\]
The product in the top right is over the vertices of $\Gamma$, the bottom left is taken over the edges of $\Gamma$, and the bottom right is the product over all vertices of $\Gamma$ of one copy of $B\mathbb G_m$ for each edge incident to $W$.

Both horizontal arrows are lci morphisms with compatible cotangent complexes. In particular, pullback along identifies fundamental classes in the top row.
\end{proposition}

We can use these basic propositions to control the virtual class attached to $\overline\sigma$.

\begin{corollary}
	If $[\Gamma\to T]$ has no vertical edges, then the pushforward of the virtual class along
	\[
	\Mbar(\overline\sigma)\to \Mbar_\Gamma(B\mathbb G_m)
	\]
	is a rational multiple of the pullback along
	$$
	\Mbar_\Gamma(B\mathbb G_m)\to \prod_W \Mbar_W(B\mathbb G_m)
	$$
	of the pushforward of the class $[\prod_V \Mbar(V)]^{\sf vir}$. If $[\Gamma\to T]$ contains vertical edges, the pushforward of $[\Mbar(\overline\sigma)]^{\sf vir}$ is the same pullback, times a tautological class on $\Mbar_\Gamma(B\mathbb G_m)$.
\end{corollary}

\begin{proof}
When there are no vertical edges, this follows from first identifying $\mathfrak E\to \mathfrak E^2$ with the bottom arrow in the square of Proposition~\ref{prop: pic-strata}.	If there are vertical edges, the proof must be modified slightly, as the evaluations at the nodes have only been required to glue after projection to $B\mathbb G_m$, whereas we need to impose the condition that they glue in $\mathcal P$. There are two cases. If the edges connect two rigid vertices, then for each edge there is an additional gluing condition on the preimages of the nodes associated to each such edge. The condition is given by pulling back under the fiberwise diagonal embedding
\[
\mathcal P\to\mathcal P\times_{B\mathbb G_m}\mathcal P.
\]
If one of the two vertices is a rubber vertex, there is no additional condition. In both cases we impose evaluation classes at these half edges, which gives the statement.
\end{proof}

\subsubsection{Proof of Theorem~\ref{thm: log-enhancement}} We complete the proof of the theorem that logarithmic enhancements of Gromov--Witten cycles in the Picard stack are tautological. 

We prove the result by nested induction. We will show that, for a fixed GW cycle associated to the universal target $(\mathcal P|D)$, if 
\[
\Mbar^\diamond_{g,n,d}(B\mathbb G_m,\beta)\to \Mbar_{g,n,d}(B\mathbb G_m,\beta)
\]
is a composition of $k$ blowups along smooth strata, the logarithmic enhancement of the GW cycle lies in the tautological ring of $\Mbar^\diamond_{g,n,d}(B\mathbb G_m,\beta)$. The plan for the induction is as follows:
\begin{enumerate}[(i)]
	\item We first prove the full result for the case of $\mathcal P$ relative to both $0$ and $\infty$, and then treat the one-sided case. 
	\item We then induct on the number of blowups along smooth strata in the morphism above.
	\item We finally induct on the order on stars, as adapted in Remark~\ref{rem: ordering}.
\end{enumerate}

\noindent
{\sc Step I. Setup and notation.} To ease the burden of notation during the course of the proof, we use $B$ to denote $\Mbar_{g,n,d}(B\mathbb G_m,\beta)$ and we use $\Mbar$ to denote the space $\Mbar^{\sf exp}_\Lambda(\mathcal P|D,\beta)$. We denote the space $\Mbar_\Lambda(\mathsf A\times B\mathbb G_m,\beta)$ by $\mathfrak a \Mbar$. As discussed above, we have maps
\[
\Mbar\to\mathfrak a\Mbar\to B,
\]
where first arrow is virtually smooth and strict, while the second is logarithmically \'etale. Finally, we let 
\[
B(k)\to B
\]
be a sequence of $k$ blowups at smooth centers. At the inductive step, we have the following standard diagram consisting of three Cartesian squares:
\[
\begin{tikzcd}
\Mbar(k+1)\arrow{d}\arrow{r} & F(k+1)\arrow{d}\arrow{r} & \Mbar(k)\arrow{d}\\
\mathfrak a\Mbar(k+1)\arrow{r} & \mathfrak a F(k+1)\arrow{r}\arrow{d} & \mathfrak a \Mbar(k)\arrow{d} \\
&B(k+1)\arrow{r} & B(k),
\end{tikzcd}
\]
of the form that we have also seen in Section~\ref{sec: exotic-nonexotic-move}. We let $W$ be the center of the blowup 
\[
B(k+1)\to B(k)
\]
and 
\[
\mathbb P_W\to W
\]
denote the exceptional divisor. We let
\[
\mathfrak a E(k+1) \ \ \textnormal{and } \ E(k+1)
\]
be the pullbacks of the exceptional divisor to the middle and top row respectively. 

\noindent
{\sc Step II. Exceptional terms.} As noted previously, the three vertical arrows in the top row are all equipped with compatible perfect obstruction theories. We work with the middle row first, control the exceptional terms, pullback to the top row and then push down, applying the standard blowup yoga from~\cite{MR21}, already used several times. 

By Fulton's refined blowup formula~\cite[Section~6.7]{Ful98} the difference
\[
[\mathfrak a F(k+1)]^{\sf exc}-[\mathfrak a \Mbar(k+1)]
\]
between the excess class on the total transform and the strict transform is given by the pushforward of a class supported on $\mathfrak a E(k+1)$. This class, in turn, is obtained from the Segre class
\[
s(W\times_{B(k)} \mathfrak a\Mbar(k),\mathfrak a\Mbar(k))
\]
by pulling back to $\mathfrak a E(k+1)$ and applying tautological Chern class operators. This fiber product is a monomial substack of $\mathfrak a\Mbar(k)$. Since are free to replace $\mathfrak a\Mbar(k)$ by a further logarithmic modification, we can assume it is a simple normal crossings pair. We can therefore apply Aluffi's formula~\cite{Alu16}. 

Now Aluffi's formula involves strata of $\mathfrak a\Mbar(k)$ and normal bundles of the hypersurfaces that meet these strata. By Lemmas~\ref{lem: normal-1} and~\ref{lem: normal-2}, these normal bundles can be represented, on any given stratum, by further boundary divisors and pullbacks of classes on the moduli stack of curves. The map to the latter factors through $B(k)$. By an application of the projection formula, we can assume the Segre class term is a strata class on $\mathfrak a\Mbar(k)$. 

\noindent
{\sc Step III. Reduction to strata contributions.} By the discussion above, we are left to analyze a stratum of $\mathfrak a\Mbar(k)$. This stratum is a compactification of a torus bundle over a stratum of $\Mbar$ itself. Specifically, it can be obtained as a stratum of the form $\Mbar^\diamond(\sigma)$ associated to a cone $\sigma$, and the compactified torus bundle structure is given by the morphism
\[
\Mbar^\diamond(\sigma)\to \Mbar(\overline\sigma)
\] 
discussed previously. We denote the stratum $Z$ until we need its precise form above.

In order to calculate the excess class above, pull back the exceptional divisor $\mathbb P_W\to W$ along $Z\to W$. This pullback is again a compactified torus bundle, now typically of higher dimension, over the same base $\Mbar(\overline\sigma)$. Fulton's formula asserts that we should apply some polynomial in the Chern classes of the excess bundle of $\mathbb P_W\to W$. This is a polynomial in the Chern classes of $W\hookrightarrow B(k)$ and in the relative hyperplane bundle of $\mathbb P_W$. However, since our goal is to eventually push forward to $B(k+1)$, another application of the projection formula means that we can simply study classes on
\[
\mathbb P_W\times_W Z
\]
given by powers of the first Chern class of the relative hyperplane bundle. At this stage, we perform virtual pullback along the top square and push back down. The virtual pullback for the top row is bivariant, and so compatible with Chern classes. Moreover, since the Chern class of the bundle $\mathbb P_W$ is represented by a relative stratum, we conclude that the correction terms are all of the following form: take a stratum of $\Mbar(k)$, push forward to a stratum contained in the center, through which it factors, pull back to the exceptional divisor $\mathbb P_W$ of $B(k+1)$ over $B(k)$, apply a tautological class, and push forward to $B(k+1)$. 

\noindent
{\sc Step IV. Concluding via boundary splitting.} It now suffices to show that terms of the form described in the paragraph above are inductively tautological, in other words, we need to understand the pushforward of the virtual class $[\Mbar^\diamond(\sigma)]^{\sf vir}$ to $B(k)$. The basic idea is that $B(k)$ maps to $B$ as an equivariant compactification of a torus bundle. The space $B$ receives a map a stratum of $\Mbar(\overline\sigma)$. By a formal diagram chase, the pushforward class that we're aiming to understand is obtained by taking the logarithmic enhancement of the pushforward of the virtual class under 
\[
\Mbar(\overline\sigma)\to B,
\]
then pulling back to a broken toric bundle that dominates $B(k)$ and then pushing forward. 

We now give the details. We have fixed $\overline \sigma$, equivalently $[\Gamma\to T]$. The map $B(k)\to B$ is a broken toric bundle. If $\Gamma$ is the graph associated to the stratum $B$, we can use the strata formula in the previous section to write
\[
\begin{tikzcd}
B\arrow{d}\arrow{r} & \prod_{V\in \Gamma} B_V\arrow{d}\\
B\mathbb G_m^E\arrow{r} & B\mathbb G_m^{2E},
\end{tikzcd}
\]
where the spaces $B_V$ are the decorated Picard stacks associated to the vertices of $V$ in $\Gamma$, the set $E$ is the edge set of $\Gamma$. 

We now observe that the horizontal morphisms above, importantly the top one, is strict. We also note that the stratum $\Mbar(\overline\sigma)$ can be chosen to be indexed by a map $[\Gamma\to T]$ of dual graphs without vertical edges.

We now note that there is a map to the space $\prod B_V$ above:
\[
\prod_V \Mbar_V\to \prod_V B_V
\]
from the space of maps to the universal bundle determined by the start at $V$, taken over all vertices of $\Gamma$. This map is an external product, i.e. given factorwise when viewed as a product over the vertices $V$ of $\Gamma$. We therefore can pass -- factorwise-- to the blowups of $B_V$ and the corresponding strict transforms of $\Mbar_V$ to produce the logarithmic enhancements:
\[
\prod_V \Mbar^\diamond_V\to \prod_V B^\diamond_V.
\]
By the inductive hypothesis, the pushforward of this virtual class is tautological in the logarithmic sense. 

Note also that since the map $B\to\prod B_V$ is strict, and so this factorwise subdivision determines a subdivision of $B^\diamond$, which we denote $B$. Now consider the fiber product:
\[
B^\diamond\times_{\prod_V B^\diamond_V} \prod_V \Mbar^\diamond_V
\]
Since $[\Gamma\to T]$ has no vertical edges, inspecting the gluing diagrams in Propositions~\ref{prop: diagonal} and~\ref{prop: pic-strata}, this fiber product is a subdivision of $\Mbar(\overline\sigma)$, and furthermore, the pushforward of the virtual class of this subdivision computes the logarithmic enhancement of $[\Mbar(\overline\sigma)]^{\sf vir}$ in the logarithmic Chow ring of $B$.\footnote{The logarithmic structure of $B$ is the divisorial logarithmic structure induced by its intersection with the other divisors of $\Mbar_{g,n,d}(B\mathbb G_m,\beta)$.}

Since the logarithmic enhancement of $[\Mbar(\overline\sigma)]^{\sf vir}$ has been shown, inductively, to be tautological, we can now conclude. The space $B(k)$ is generically a torus bundle over $B^\diamond$, and by blowing up $B(k)$, we can obtain a diagram
\[
\begin{tikzcd}
& \widetilde B(k)\arrow{dr}\arrow{dl} &\\
B(k) & & B^\diamond.	
\end{tikzcd}
\]
Furthermore, a straightforward diagram chase shows that the pushforward of $[\Mbar^\diamond(\overline)]^{\sf vir}$ to $B(k)$ is obtained by pull/push from the logarithmic enhancement on $B^\diamond$ constructed above. We conclude the result. \qed

\subsubsection{Pullback to geometric targets and Theorem~\ref{thm: P1-bundle-formula}} The result that logarithmic enhancements of GW classes of the universal $\mathbb P^1$-bundle is tautological implies the result for logarithmic targets, as stated in Theorem~\ref{thm: P1-bundle-formula}. This follows from a virtual pullback argument. We provide the details.

To reset notation, recall $(Y|\partial Y)$ is an snc pair and $\mathbb P\to Y$ is the projective completion of a line bundle. The logarithmic structure on $\mathbb P$ is given by the pullback $\partial Y$, possibly together with one or both of the $0$ and $\infty$ sections. Recall also that $[\PP^1/\mathbb G_m]\to B\mathbb G_m$ is the universal $\mathbb P^1$-bundle. As a result, there is a Cartesian square
\[
\begin{tikzcd}
\mathbb P\arrow{d}\arrow{r} & {[\PP^1/\mathbb G_m]}	\arrow{d}\\
Y\arrow{r} & B\mathbb G_m. 
\end{tikzcd}
\]
The square is Cartesian in the category of logarithmic schemes as well, with the logarithmic structures above for $Y$ and $\mathbb P$, the trivial logarithmic structure on $B\mathbb G_m$, and the logarithmic structure on $\mathcal P$ given by a subset of $0$ and $\infty$, as appropriate based on the choice of $\mathbb P$. We let $\mathsf A$ denote the Artin fan of $[\mathbb P^1/\mathbb G_m]$ -- it is either $[\mathbb A^1/\mathbb G_m]$ or $[\mathbb P^1/\mathbb G_m]$.

Passing to mapping spaces, with diamonds denoting that we have introduced appropriate subdivisions, we have:
\[
\begin{tikzcd}
\Mbar^\diamond_\Lambda(\PP|\partial \PP,\beta)\arrow{r}\arrow{d}	& \Mbar^\diamond_\Lambda(\mathcal P|D,\beta)\arrow{d}\\
\Mbar^\diamond_\Lambda(Y\times\mathsf A|\partial Y+D,\beta)\arrow{r}\arrow{d}	& \Mbar^\diamond_\Lambda(B\mathbb G_m\times\mathsf A,\beta)\arrow{d}\\
\Mbar^\diamond_\Lambda(Y|\partial Y)\arrow{r} & \Mbar^\diamond_{g,n,d}(B\mathbb G_m,\beta).
\end{tikzcd}
\]
The subdivisions are chosen such that the right vertical in the bottom square is flat with reduced fibers. The morphism is toroidal to begin with, so this is achieved by toroidal semistable reduction. 

Analyzing the diagram, in the top square, both vertical arrows are equipped with compatible perfect obstruction theories given by the logarithmic tangent bundle. This follows from cosmetic changes to~\cite[Section~6]{AW}. The bottom square is Cartesian in the category of schemes and logarithmic schemes, because the right vertical in the bottom square has been made flat by the choice of subdivision. The lowest horizontal arrow is equipped with a perfect logarithmic obstruction theory, i.e. a perfect obstruction theory over the relative stack of logarithmic structures, e.g. by the argument of~\cite[Section~5.2]{ACW}. The flatness in the bottom row means that the middle row is also equipped with a virtual pullback.

By commutativity of virtual pullbacks with each other~\cite{Mano12}, a diagram chase and the application of the universal result in the previous section shows that the pushforward of a GW class along
\[
\Mbar^\diamond_\Lambda(\mathbb P|D,\beta)\to \Mbar_\Lambda^\diamond(Y|\partial Y)
\]
is equal to the virtual pullback of a tautological class from $\Mbar^\diamond_{g,n,d}(B\mathbb G_m,\beta)$. In other words, the GW pushforward on $\Mbar_\Lambda^\diamond(Y,\beta)$ is obtained by taking tautological cohomology class on some blowup of the universal Picard-type stack $ \Mbar^\diamond_{g,n,d}(B\mathbb G_m,\beta)$, pulling back to the mapping space, and capping with the virtual class. Pushing forward this expression to the moduli space of curves, we obtain Theorem~\ref{thm: P1-bundle-formula}.

\subsection{Compressed evaluation spaces}\label{sec: compressed-evaluations} Let $(Y|\partial Y)$ be a pair as above, and let $\mathbb P\to Y$ be a $\mathbb P^1$-bundle. Equip it with logarithmic structure along $\partial Y$ as well as either one or two sections. We refer to these cases as the {\it one sided} and {\it two sided} $\mathbb P^1$-bundles. 

Fix discrete data $\Lambda$ for a logarithmic stable maps problem to $\PP$, noting again that we have suspended the disjointness assumption. Let $\sf v$ be the corresponding star attached to $\Lambda$. 

There is a distinguished projection
\[
\Sigma_{\PP|\partial \PP}\to \Sigma_{Y|\partial Y}
\]
which, ignoring fan structures, has fibers equal to $\RR_{\geq 0}$ in the one sided case and $\RR$ in the two sided case. Let $\sf w$ be a star in $\Sigma_{\PP|\partial \PP}$. The distinguished projection allows us to identify a distinguished $\mathbb P^1$-bundle direction in the target that is determined by $\sf w$. 

In more detail, if we project $\sf w$ onto $\Sigma_{Y|\partial Y}$ we obtain a star $\overline w$ in the target, and (non-canonically or up to subdivision) a component of an expansion of $(Y|\partial Y)$, where $\overline w$ is disjoint. This is a broken toric bundle over a stratum. 

We can therefore view a target determined by $\sf w$ as a blowup of a $\mathbb P^1$-bundle (either one or two sided). We call this the {\it associated $\mathbb P^1$-bundle at $\sf w$}. In the context of mapping spaces, evaluations pulled back from the associated $\mathbb P^1$-bundle will be called {\it compressed insertions}. 

Let $\widetilde \PP\to \PP$ be a blowup such that the canonical lift of $\Lambda$ is disjoint. There are therefore two evaluation spaces associated to $\Lambda$, depending on whether we interpret it on $\PP$ or $\widetilde \PP$ via the canonical lift. See for instance Figure~\ref{fig: blowup-of-bundle}.
\begin{figure}
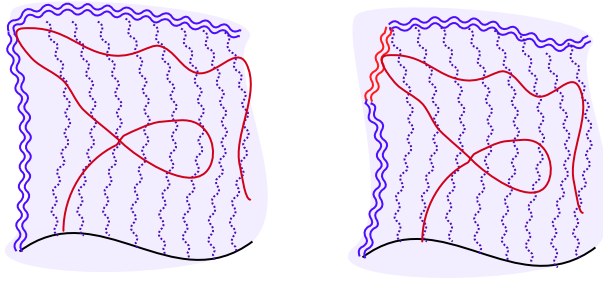


\tikzset{every picture/.style={line width=0.75pt}} 


\caption{A $\mathbb P^1$-bundle over a logarithmic base, and a blowup of it. The curve goes into the corner on the left. On the left we only have access to compressed impression but on the right we have a full dimensional evaluation space.}\label{fig: blowup-of-bundle}
\end{figure}

There is also an obvious map
\[
\mathsf{Ev}\to \overline{\sf Ev}
\]
between the total evaluation spaces spaces, as a product over all marked points. The key issue, which makes this discussion slightly different than the exotic/non-exotic comparison, is that this map typically has large, positive dimensional fibers. We refer to $\overline{\sf Ev}$ as the compressed evaluation space.  

Since any star has a distinguished $\mathbb P^1$-bundle direction, any star has an analogous compressed evaluation space, noting as usual that it only makes sense to say this after choosing a fan structure completing the image $\overline {\sf w}$ of $\sf w$ under $\Sigma_{\PP|\partial \PP}\to \Sigma_{Y|\partial Y}$.

We maintain the notation and terminology. 

\begin{theorem}\label{thm: compression-of-insertions}
The GW cycle pushforwards associated to 
\[
\Mbar_\Lambda(\widetilde\PP|\partial \widetilde\PP)\to \Mbar_\Lambda(Y|\partial Y)
\]
with arbitrary insertions can be expressed in terms of GW pushforwards with compressed insertions associated to $\Lambda$ in $\PP|\partial \PP$, and associated to all stars in $\Sigma_{\PP|\partial\PP}$ that are strictly smaller in the partial ordering by stars. 
\end{theorem}

The proof has a structure similar to that of Theorem~\ref{thm: rigid-split}, so the reader may wish to review the discussion of Section~\ref{sec: translated-stars}.

\begin{proof}
Consider the morphism $\mathsf{Ev}\to \overline{\sf Ev}$. For each marking $p_i$ we obtain the corresponding factor ${\sf Ev}_i$ from $\overline{\sf Ev}_i$ by adding the ray corresponding to its contact order to the fan of $\PP$, followed by some further subdivisions. By blowing up $\widetilde \PP$ further if required, we can assume there is a factorization:
\[
\mathsf{Ev}_i\to {\sf F_i}\to \overline{\sf Ev}_i
\]
where ${\sf F_i}\to \overline{\sf Ev}_i$ is a projective space bundle. We have corresponding maps, consolidated over the different markings:
\[
\mathsf{Ev}\to {\sf F}\to \overline{\sf Ev}.
\]
We can assume that the first arrow factorizes as a sequence of blowups along smooth centers. After the last step, we have to compose with ${\sf F}\to \overline{\sf Ev}$ which is not a blowup. We start with an insertion $\alpha$ on ${\sf Ev}$ and, for each step in the blowup factorization, we know that $\alpha$ differs from a class pulled back along the blowup by an error term. The error term is an insertion pushed forward from the exceptional divisor. By the exact same calculus as in Section~\ref{sec: exotic-evaluations}, we recognize this invariant as coming from the center of the blowup, together with additional decorations coming from the Artin fan associated to the space of Chow $1$-complexes. Descending down the tower of blowups, we end up with a strata GW classes, including with $T$-class, as in Definition~\ref{def: stratum-GW-class}. However, all the insertions now come from $\sf F$. 

The last step is not a blowup, but we can still apply the formula for the cohomology of a projective bundle factorwise. It means we can write any insertion on $\sf F$ as a pullback of a cohomology class on the compressed evaluation space $\overline{\sf Ev}$, times a stratum of $\sf F$. By the same argument above, this again turns into a combination of strata GW classes, with further decorations, and now all the insertions come from ${\sf Ev}$. 

This puts us in a position to induct, as in the final step of the proof of Theorem~\ref{thm: rigid-split}. The induction proceeds exactly as in that case, with the following two observations. First, we can always add a point of tangency order $0$ to rigidify, and by its construction in {sec: rigid-calculus}, after splitting the rigidifying class always translates into a compressed insertion on the star. Second, for dealing with the case of translated stars, we use the discussion in Section~\ref{sec: translated-stars}. We can further induct on the genus and tangency order here, to reduce to the case in the remark, and observe that the rigidification/splitting of translated stars do not lead to uncompressed insertions. This completes the proof. 
\end{proof}

\subsection{Conclusions for broken bundles}\label{sec: broken-bundles} 

We record two corollaries, which arise from combining the compressed evaluation spaces above with the theorems on collapsing $\PP^1$-bundles. 

\begin{corollary}\label{cor: collapsing-log-toric-bundles}
Let $(E|\partial E)\to (Y|\partial Y)$ be a broken toric bundle. The logarithmic GW cycles in $\Mbar_{g,n}$ of $(E|\partial E)$ lie in the span of logarithmic GW cycles of the pairs $(W|\partial W)$ where $W$ ranges over strata of $Y$ and $\partial W$ is the induced snc boundary. 
\end{corollary}

\begin{proof}
We start with a logarithmic GW cycle on $(E|\partial E)$ with a possibly exotic insertion, and some discrete data. Apply the exotic/non-exotic algorithm to reduce to non-exotic GW cycles $(E|\partial E)$ which is itself a broken toric bundle, as well as broken toric bundles over strata of $(Y|\partial Y)$. For each such target, we can express it as a $\mathbb P^1$-bundle over a smaller broken toric bundle. By Theorem~\ref{thm: compression-of-insertions} we can assume the insertions are compressed with respect to this choice of $\PP^1$-bundle, with the correction terms being analogous compressed invariants over the deeper strata of $(Y|\partial Y)$. We are now free to apply Theorem~\ref{thm: P1-bundle-formula}, which collapses us to the base of this $\mathbb P^1$-bundle. Now proceed by induction until every branch of the induction reaches a stratum $(W|\partial W)$ of the claimed form. 
\end{proof}

The analogous corollary holds for one-sided bundles. Let  $(\PP|\partial \PP)\to (Y|\partial Y)$ be a one-sided $\PP^1$-bundle. Let $E \to \PP\to Y$ be logarithmic blowup. 

\begin{corollary}\label{cor: collapsing-one-sided-bundles}
Let $(E|\partial E)\to (Y|\partial Y)$ be as above. The logarithmic GW cycles in $\Mbar_{g,n}$ of $(E|\partial E)$ lie in the span of logarithmic GW cycles of the pairs $(W|\partial W)$ where $W$ ranges over strata of $Y$ and $\partial W$ is the induced snc boundary. 
\end{corollary}

\begin{proof}
The induction is parallel; the details are left to the reader. 
\end{proof}

The results also prove a conjecture of Urundolil Kumaran and the second author, see~\cite[Conjecture~C]{RUK22}.

\begin{corollary}
Let $(Y|\partial Y)$ be a toric variety equipped with any torus invariant boundary divisor. The logarithmic Gromov--Witten theory of $(Y|\partial Y)$ is tautological.
\end{corollary}

\begin{proof}
Let $(Y|\partial Y)$ be as given. We induct on the number of torus invariant divisors that are {\it not} contained in $\partial Y$ and on the dimension. The base case is~\cite[Theorem~A]{RUK22}. In order to perform the induction, choose a boundary divisor $E\subset Y$ that is contained in the toric boundary but not in $\partial Y$. Take deformation to the normal cone and apply the logarithmic degeneration formula~\cite{MR23}. This reduces the problem inductively	to a broken one-sided $\mathbb P^1$-bundle over the toric divisor $E$ with logarithmic structure and the inductive hypothesis.  Since $E$ has smaller dimension, we conclude by induction on the dimension. 
\end{proof}


\section{The logarithmic/absolute correspondence}\label{sec: log-absolute}

\subsection{Goal of the section}

Given a degeneration
\[
\cX\to B
\]
our previous results express any Gromov--Witten (GW) cycle on the general fiber in terms of the logarithmic GW cycles of strata of the special fiber. A key consequence of Sections~\ref{sec: uni-P1-bundle} and~\ref{sec: uni-log-bundle} is that we may work directly with these strata and their natural snc divisors, without introducing additional bundles.

The goal of this section is to remove the logarithmic structure, and study non-exotic GW cycles for an snc pair $(Y|\partial Y)$. Our main result is:

\begin{theorem}\label{thm: invertibility}
The non-exotic Gromov--Witten cycles of $(Y|\partial Y)$ in the cohomology ring of $\Mbar_{g,n}$ lie in the span of the ordinary Gromov--Witten cycles of $Y$ and all strata of $\partial Y$.
\end{theorem}

By the exotic/non-exotic correspondence from Theorems~\ref{thm: exotic-to-strata} and~\ref{thm: rigid-split}, this implies Theorem~\ref{thm: exotic-non-exotic}.

This statement holds in cohomology in full generality; the Chow-theoretic version holds provided the strata of $(Y|\partial Y)$ admit Chow--K\"nneth decompositions. In our applications, this follows from the linearity of the locally closed strata in the sense of Totaro~\cite{Tot14}.

The main work is to prove an invertibility result for the degeneration formula in the setting of deformation to the normal cone of a divisor. Without logarithmic structure, this invertibility was shown in~\cite{MP06}. The strategy here is similar -- approximate tangency by descendents and control the error terms —- but the logarithmic setting introduces significant additional complexity. We tame these using Corollaries~\ref{cor: collapsing-log-toric-bundles} and~\ref{cor: collapsing-one-sided-bundles}. 

\subsection{The degeneration and its tropicalization} Let $E$ be a component of $\partial Y$ and write 
\[
\partial Y = D+E.
\]
Let
\[
\mathcal Y\to \mathbb A^1
\]
be the degeneration to the normal cone of $E$. We equip the degeneration with the divisorial logarithmic structure given by the closure of 
\[
D\times\mathbb G_m\subset\mathcal Y
\]
together with the special fiber. A general fiber is isomorphic to the pair $(Y,D)$. The special fiber consists of two components
\[
(Y|\partial Y)\cup (\mathbb P_E|\mathbb P_{D\cap E}),
\]
where $\mathbb P_E$ is the projective completion of the normal bundle of $E$ in $Y$, and $\mathbb P_{D\cap E}$ is defined analogously. See Figure~\ref{fig: expansion}.

\begin{figure}[h!]
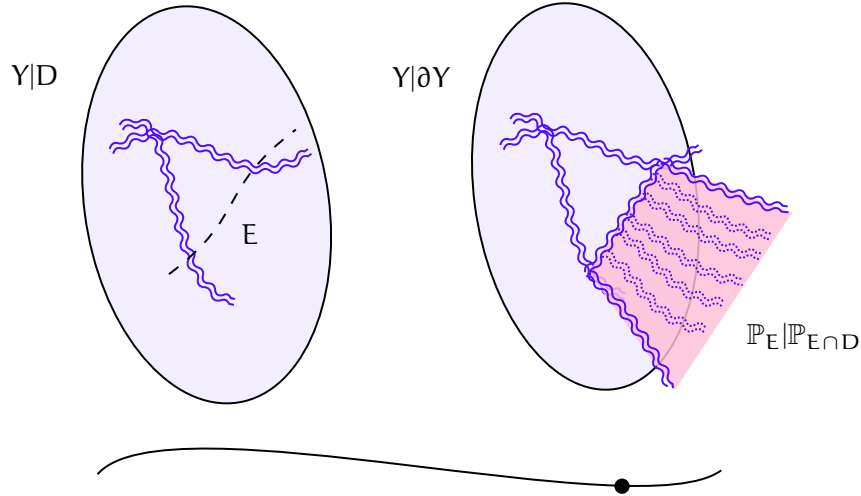


\tikzset{every picture/.style={line width=0.75pt}} 


\caption{The degeneration to the normal cone that adds logarithmic structure along $E$ to the pair $(Y|D)$ by degenerating to the normal cone. The wavy lines depict divisors in the logarithmic structure, while the dashed line $E$ is not part of the logarithmic structure.}\label{fig: expansion}
\end{figure}

The tropicalization of $\mathcal Y\to \mathbb A^1$ induces a map
\[
\Sigma_{\mathcal Y}\to \mathbb R_{\geq 0}.
\]
The fiber over the point $1$ is the dual complex of the special fiber $\mathcal Y_0$ and is denoted $\Delta(\mathcal Y_0)$. Recall that we assume that the components of $\partial Y$, and therefore of $D$, have connected intsersections. As a result, labeling the components of $D$ from $1$ to $k$, we obtain a natural embedding
\[
\Sigma_{Y|D}\hookrightarrow \mathbb R_{\geq 0}^{k}.
\]
After fixing this embedding, we obtain induced embeddings
\[
\Delta(\mathcal Y_0)\hookrightarrow \mathbb R_{\geq 0}^k\times [0,1].
\]
See Figure~\ref{fig: various-tropical-curves} below in case $k = 1$. It will be harmless to assume that this embedding is actually an isomorphism, essentially pretending that all strata intersections are nonempty. 

There is the projection
\[
\Delta(\mathcal Y_0)\to [0,1].
\]
We refer to the preimage of $0$ as the {\it $Y$-face} and the preimage of $0$ as the {\it $\PP$-face}. The complement of the $\mathbb P$-face is a cell complex. Extending the natural face $[0,1)$ to $[0,\infty)$ in the obvious way, we obtain the fan of $(Y|\partial Y)$. Similarly, the complement of the $Y$-face gives the fan of $(\mathbb P_E|\PP_{D\cap E})$. The fiber over any point in the interior of $[0,1]$ is naturally identified with the fan of $E$ with the logarithmic structure given by the divisorial components of the intersection of $E$ with the components of $D$.  

\subsection{Approximation by descendants} Our goal is to calculate a non-exotic GW cycle obtained from
\[
\mathsf{ev}\colon \Mbar_\Lambda(Y|\partial Y)\to \mathsf{Ev}_\Lambda(Y|\partial Y).
\]
This will be done by finding a descendant invariant on the pair $(Y|D)$ that, when calculated via the degeneration formula above, has this GW cycle as a leading term. 

We reminder the reader that the discrete data $\Lambda$ includes the data of the genus $g$, number of markings $n$, the degree $\beta$, and the tangency matrix $[c_{ij}]$. The tangency matrix records the tangency of maps to $D_i$ along a point $p_j$. In particular, if $E$ is a component of $\partial Y$, we have a vector of tangency orders $(c_{E,1},\ldots,c_{E,n})$. 

The discrete data $\Lambda$ naturally induces discrete data for a map to the pair $(Y|D)$ obtained by dropping $E$ from $(Y|\partial Y)$ -- the genus, number of markings, and curve class are unchanged. We delete the appropriate column of $[c_{ij}]$.  We use the symbol $\overline\Lambda$ for this new discrete data. 

Let $S\subset [k]$ be the subset of markings that have positive tangency order with $E$. If we choose $i$ in $S$, the marked point $p_i$ has two evaluation spaces -- one for $(Y|\partial Y)$ and one for $(Y|D)$, respectively denote $W_i$ and $\overline W_i$. These strata are the intersection, in the appropriate pair, of all divisors with which $p_i$ has positive tangency order. There is a natural inclusion
\[
W_i \hookrightarrow \overline W_i.
\]
Consolidating by taking products over all $i$ in $S$ we have an inclusion
\[
\iota\colon \mathsf{Ev}_S\hookrightarrow \overline{\mathsf{Ev}}_S.
\]
We can reorder the markings to write
\[
\mathsf{Ev}_\Lambda(Y|\partial Y) = \mathsf{Ev}_S\times\mathsf{Ev}_S' \ \ \textsf{and} \ \ \ \mathsf{Ev}_{\overline\Lambda}(Y|D) = \overline{\mathsf{Ev}}_S\times\mathsf{Ev}_S'.
\]

Consider an evaluation class
\[
\delta \otimes \eta\in H^\star(\mathsf{Ev}_S)\otimes H^\star(\mathsf{Ev}_S'),
\]
and fix the GW class obtained by pushing forward
\[
\mathsf{ev}^\star(\delta\otimes\eta)\cap [\Mbar_\Lambda(Y|\partial Y)]^{\sf vir}.
\]
to the moduli space of curves. 

\begin{definition}
	The {\it descendant approximation along $E$} of the invariant 
	\[
	\mathsf{ev}^\star(\delta\otimes\eta)\cap [\Mbar_\Lambda(Y|\partial Y)]^{\sf vir}.
	\]
	is
	\[
	\left(\mathsf{ev}^\star(\iota_\star(\delta)\otimes\eta)\cup \prod_{i=1}^n \psi_i^{c_{E,i}}\right)\cap [\Mbar_{\overline \Lambda}(Y|D)]^{\sf vir}
	\]
\end{definition}

Let us emphasize that this approximation makes sense, and has good properties, even if we {\it do not} need to assume disjointness of the contact order. However, since we can always blowup strata and then collapse the bundle by~\cite{Fan21,HLR08}, it is harmless to do so.

\subsection{Setup for controlling the approxmation} Fix the GW cycle of $(Y|\partial Y)$ that is to be computed. We now calculate the descendant approximation along $E$ by applying the degeneration formula to 
\[
\mathcal Y\to \mathbb A^1
\]
constructed above. 

The insertions $\iota_\star(\delta)$ and $\eta$ must be lifted to the degeneration. We do this as follows. Each stratum $W$ in $(Y|D)$ produces a stratum degeneration by taking the closure of $W\times\mathbb G_m$. The strata therefore have maps
\[
{\mathcal W}\to {W}.
\]
We can take a fiber product of this strata degeneration over all points in $[k]\setminus S$ to produce a degeneration of $\mathsf{Ev}_S'$ with a projection
\[
\mathcal{E}v_S'\to \mathsf{Ev}_S'.
\]
\begin{itemize}
	\item Lift $\eta$ to $\mathcal{E}v_S'$ by pulling back along the morphism $\mathcal{E}v_S'\to \mathsf{Ev}_S'$.
\end{itemize}

Given a stratum $\overline W$ of $(Y|D)$ we can consider the intersection $\overline W\cap E$, and taking a product over the points we have
\[
\mathsf{Ev}_S \hookrightarrow\overline{\mathsf{Ev}}_S,
\]
which on each factor is the section coming from intersection with $S$. There is a natural degeneration
\[
\overline{\mathcal{E}v}_S\to \overline{\mathsf{Ev}}_S\times\mathbb A^1.
\]
We can take the strict transform of $\mathsf{Ev}_S\times\mathbb A^1$, which maps to
\[
\mathcal{E}v^\sim_S\to \mathsf{Ev}_S.
\]
\begin{itemize}
	\item Lift $\iota_\star(\delta)$ by pulling back $\delta$ from $\mathsf{Ev}_S$ to $\mathcal{E}v_S^\sim$ and pushing forward to the total space $\overline{\mathcal{E}v}_S$ of the degeneration of the evaluation space. 
\end{itemize}

\subsection{Computation of the approximation} We now apply the degeneration package~\cite[Section~8.3]{MR23} to calculate the descendant approximation. The contributions to the degeneration formula are indexed as:
\[
\left(\mathsf{ev}^\star(\iota_\star(\delta)\otimes\eta)\cup \prod_{i=1}^n \psi_i^{c_{E,i}}\right)\cap [\Mbar_{\overline \Lambda}(Y|D)]^{\sf vir} = \sum_\gamma \left(\prod_{i=1}^n \psi_i^{c_{E,i}}\cup \xi \right)\cap  [\Mbar_\gamma]^{\sf vir}. 
\]
where $\xi$ is notation for the consolidation of all the evaluation classes noted above, where $\gamma$ is a rigid tropical stable map:
\[
\gamma\colon \Gamma\to \Delta(\mathcal Y_0). 
\]
The tropical curve $\Gamma$ has vertex decorations by genus and curve class, and the map must satisfy the modified balancing conditions, as explained in Section~\ref{sec: stars-partial-ordering}. 

\subsubsection{The principal term in the formula} There is a {\it principal term} in the degeneration formula is given by a tropical map
\[
\gamma_\circ\colon \Gamma_\circ\to\Delta(\mathcal Y_0)
\]
with the following features:
\begin{enumerate}[(i)]
	\item The graph $\Gamma_\circ$ is a star graph, consisting of a central vertex $V_Y$ and vertices $V_1,\ldots, V_m$ attached to $V_Y$ by a single edge but not to each other. 
	\item The vertex $V_Y$ maps to the $Y$-vertex of $\Delta(\mathcal Y_0)$, while all other $V_i$'s map to the $\mathbb P$-face of $\Delta(\mathcal Y_0)$. 
	\item The genus label at $V_Y$ is $g$ and the genus label of all other $V_i$'s is $0$. 
	\item The curve class label at $V_Y$ is $\beta$. The curve class labels at each $V_i$ is a class $\beta_i$ that pushes forward to $0$ under the map to $Y$. 
	\item After identifying a neighborhood of the $Y$-vertex of $\Delta(\mathcal Y)$ with a neighborhood of $\Sigma_{Y|\partial Y}$, the edge direction at the vertex $E_i$ joining $V_Y$ and $V_i$, and its slope, are as given by the starting discrete data $\Lambda$ for the corresponding marked point. 
\end{enumerate}

In other words, the star of the vertex $V_Y$ in $\gamma_\circ$ is exactly $
\Lambda$. See the leftmost part of Figure~\ref{fig: various-tropical-curves} and notice that the leftmost vertex locally looks like the star of a generic curve in $(Y|\partial Y)$. 

\subsubsection{Higher order terms in the formula} Now consider the other rigid tropical curves 
\[
\gamma\colon \Gamma\to\Delta(\mathcal Y_0). 
\]
We make some observations about such terms that are different from $\gamma_\circ$. 

\begin{lemma}\label{prop: rigid-tropical-curves}
	Any tropical curve $\gamma$ as above is either $\gamma_\circ$ or satisfies one of the following properties
\begin{enumerate}[(i)]
	\item The curve class decorations of any vertices of $\gamma$ that map to the $Y$-vertex of $\Delta(\mathcal Y_0)$ are smaller than $\beta$.
	\item The genus decoration of any vertices of $\gamma$ that map to the $Y$-vertex of $\Delta(\mathcal Y_0)$ are smaller than $g$.
	\item After identifying the complement of the $\mathbb P$-face with $\Sigma_{Y|\partial Y}$, the restriction of $\gamma$ to this complement is a map from a tree that satisfies the  balancing condition away from the $Y$-vertex of $\Delta(\mathcal Y_0)$. 
\end{enumerate}
\end{lemma}

\begin{proof}
	If we assume that $\gamma$ is not $\gamma_\circ$, and all the genus and curve class is at the $Y$ vertex, that implies that all other vertices are genus $0$ with trivial curve class decoration. In particular, because the curve class decoration is $0$, all vertices in $\Delta(\mathcal Y_0)$ in the complement of the $\mathbb P$-face are balanced in the traditional sense, i.e. without the curve class correction described earlier. As a result, the complement of the $\mathbb P$-face produces a balanced tropical curve in $\Sigma_{Y|\partial Y}$ as claimed. 

    Examples showing different types of tropical curves appear in Figure~\ref{fig: various-tropical-curves}.

    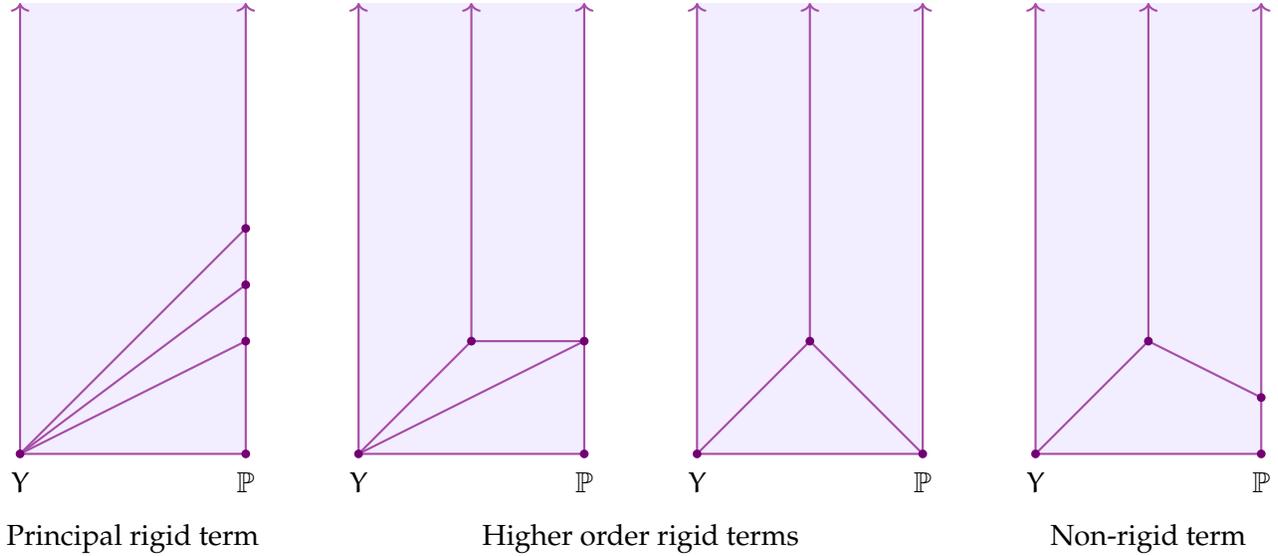
\begin{figure}
        \begin{tikzpicture}[scale=1.5]

  \definecolor{softlavender}{RGB}{223,210,255}


  \coordinate (A) at (0,0);
  \coordinate (B) at (2,0);
  \coordinate (C) at (0,4);
  \coordinate (D) at (2,4);
  
  \filldraw[fill=softlavender, draw=none,opacity=0.4] (D) -- (B) -- (A) -- (C)-- cycle;

\draw[->,violet!70, line width=0.8pt] (A)--(C);
\draw[->,violet!70, line width=0.8pt] (B)--(D);
\draw[violet!70, line width=0.8pt] (B)--(A);

\draw[-,violet!70, line width=0.8pt] (A)--(2,1);
\draw[-,violet!70, line width=0.8pt] (A)--(2,2);
\draw[-,violet!70, line width=0.8pt] (A)--(2,1.5);

\node at (0,-0.25) {$Y$};
\node at (2,-0.25) {$\mathbb P$};

\node at (1,-0.75) {Principal rigid term};

  \foreach \pt in {A,B} {
\fill[violet!90!black] (\pt) circle (1.1pt);
}
\fill[violet!90!black] (2,1) circle (1.1pt);
\fill[violet!90!black] (2,2) circle (1.1pt);
\fill[violet!90!black] (2,1.5) circle (1.1pt);

\begin{scope}[shift={(3,0)}]
  \coordinate (A) at (0,0);
  \coordinate (B) at (2,0);
  \coordinate (C) at (0,4);
  \coordinate (D) at (2,4);
  
  \node at (0,-0.25) {$Y$};
\node at (2,-0.25) {$\mathbb P$};

\node at (2.5,-0.75) {Higher order rigid terms};

  \filldraw[fill=softlavender, draw=none,opacity=0.4] (D) -- (B) -- (A) -- (C)-- cycle;

\draw[->,violet!70, line width=0.8pt] (A)--(C);
\draw[->,violet!70, line width=0.8pt] (B)--(D);
\draw[->,violet!70, line width=0.8pt] (1,1)--(1,4);
\draw[violet!70, line width=0.8pt] (B)--(A);

\draw[-,violet!70, line width=0.8pt] (A)--(2,1);
\draw[-,violet!70, line width=0.8pt] (A)--(1,1);
\draw[-,violet!70, line width=0.8pt] (2,1)--(1,1);

  \foreach \pt in {A,B} {
\fill[violet!90!black] (\pt) circle (1.1pt);
}
\fill[violet!90!black] (2,1) circle (1.1pt);
\fill[violet!90!black] (1,1) circle (1.1pt);

\end{scope}

\begin{scope}[shift={(6,0)}]
  \coordinate (A) at (0,0);
  \coordinate (B) at (2,0);
  \coordinate (C) at (0,4);
  \coordinate (D) at (2,4);
  
  \node at (0,-0.25) {$Y$};
\node at (2,-0.25) {$\mathbb P$};

  \filldraw[fill=softlavender, draw=none,opacity=0.4] (D) -- (B) -- (A) -- (C)-- cycle;

\draw[->,violet!70, line width=0.8pt] (A)--(C);
\draw[->,violet!70, line width=0.8pt] (B)--(D);
\draw[->,violet!70, line width=0.8pt] (1,1)--(1,4);
\draw[violet!70, line width=0.8pt] (B)--(A);

\draw[-,violet!70, line width=0.8pt] (B)--(1,1);
\draw[-,violet!70, line width=0.8pt] (A)--(1,1);

  \foreach \pt in {A,B} {
\fill[violet!90!black] (\pt) circle (1.1pt);
}
\fill[violet!90!black] (1,1) circle (1.1pt);

\end{scope}

\begin{scope}[shift={(9,0)}]
  \coordinate (A) at (0,0);
  \coordinate (B) at (2,0);
  \coordinate (C) at (0,4);
  \coordinate (D) at (2,4);
  
  \node at (0,-0.25) {$Y$};
\node at (2,-0.25) {$\mathbb P$};

  \filldraw[fill=softlavender, draw=none,opacity=0.4] (D) -- (B) -- (A) -- (C)-- cycle;

\draw[->,violet!70, line width=0.8pt] (A)--(C);
\draw[->,violet!70, line width=0.8pt] (B)--(D);
\draw[->,violet!70, line width=0.8pt] (1,1)--(1,4);
\draw[violet!70, line width=0.8pt] (B)--(A);

\draw[-,violet!70, line width=0.8pt] (2,0.5)--(1,1);
\draw[-,violet!70, line width=0.8pt] (A)--(1,1);

\node at (1,-0.75) {Non-rigid term};
  
  \foreach \pt in {A,B} {
\fill[violet!90!black] (\pt) circle (1.1pt);
}
\fill[violet!90!black] (1,1) circle (1.1pt);
\fill[violet!90!black] (2,0.5) circle (1.1pt);

\end{scope}

\end{tikzpicture}
\caption{The first three figures depict the combinatorics of different possible rigid tropical stable maps that contribute to the degeneration formula. The last one is non-rigid.}\label{fig: various-tropical-curves}
    \end{figure}

\end{proof}

We refer to rigid tropical maps $\gamma$, different from $\gamma_\circ$, that satisfy only condition (iii) above, but not (i) or (ii) as {\it secondary terms}. All other terms will be called {\it higher order terms}. We now make some observations about the secondary terms. 

\begin{lemma}
	Let $\gamma\colon \Gamma\to \Delta(\mathcal Y_0)$ be a secondary term. Then $\Gamma$ is a star graph with a central vertex $V_Y$ carrying all the genus and curve class. Furthermore, the remaining vertices $V_1,\ldots, V_\ell$ are attached only to this vertex and they map to the $\mathbb P$-face. 
\end{lemma}

Note that we are not claiming that the secondary term coincides with $\gamma_\circ$. Indeed, these terms are exactly further errors terms that crop up in the descendant approximation. 

\begin{proof}
	Consider a vertex $V$ of $\Gamma$ that maps to the $Y$-face of $\Delta(\mathcal Y_0)$. We claim that no edge incident to $V$ can leave the $Y$-face. Indeed, if there is such an edge, by the balancing condition, the curve class at this vertex cannot be trivial when pushed forward to $Y$, which is part of the definition of being a secondary term. Since $\gamma$ must be rigid it follows that there is a single vertex in the $Y$-face and it must map to the $Y$-vertex of $\Delta(\mathcal Y_0)$. 
	
	Next, consider vertices in the $\mathbb P$-face of $\Delta(\mathcal Y_0)$. A component dual to this vertex is naturally a broken toric bundle over a stratum in $\mathbb P_E$. The curve class at such a vertex must push forward to a fiber class in $\mathbb P_E$. Now consider a vertex in the $\mathbb P$-face. By the balancing condition, if a vertex has two distinct outgoing directions leaving the $\mathbb P$-face, the curve class cannot push forward to a fiber, again by the balancing condition. Similarly, if there is an edge connecting two $\mathbb P$-face vertices that map to different points in the face, the curve class cannot be a fiber on both. Therefore, examining the image of the map $\mathsf{im}(\gamma)$, every vertex in the $\mathbb P$-face has a single edge leaving that face. 
	
	Now the question of rigidity of this map is the same as rigidity of the map to the open cell obtained by deleting the $\mathbb P$-face and restricting the map from $\gamma$, which will now have unbounded edges. This is easily seen to be rigid only under the circumstances of the lemma, so we conclude. 
	
\end{proof}

\subsection{Conclusion via the degeneration formula} We prove the main result by induction on the discrete data, parallel to the structure of Proposition~\ref{prop: rigid-tropical-curves}. 

Consider the difference between a GW cycle of $(Y|\partial Y)$ and its descendant approximation along $E$. By the decomposition structure, this difference is given by GW cycles associated to rigid tropical curves. Consider a rigid tropical curve different from $\gamma_\circ$, which we analyze separately. The degeneration formula expresses the GW cycle associated to $\gamma$ in terms of {\it exotic} GW cycles associated to the vertices. Each vertex gives rise to a GW cycle of the following kind:

\begin{enumerate}[(i)]
	\item A GW cycle associated to a blowup of the target $(Y|\partial Y)$ with smaller genus, curve class. 
	\item A GW cycle associated to a broken toric bundle over a stratum of $(Y|\partial Y)$. 
	\item A GW cycle of a blowup of a $\mathbb P^1$-bundle over a stratum of $E$, relative to one or both of the sections, and pullbacks of divisors on $E$. 
	\item A vertex in a principal or secondary term. 
\end{enumerate}

For higher-order terms, first note that by the degeneration formula, the GW cycle is determined by exotic GW classes at the vertices. By applying the exotic/non-exotic move of Section~\ref{sec: exotic-nonexotic-move}, these are controlled by the induction. Specifically, induct on the dimension of the target, the curve class, and then the genus. By the $\mathbb{P}^1$-bundle and broken toric bundle results of Section~\ref{sec: broken-bundles}, all terms away from the $Y$-vertex lie in the span of GW cycles of corresponding strata. By induction on the curve class and genus, the $Y$-vertex term is also known to be tautological.

It remains to control the secondary terms in the formula. For this, we use the same geometric argument as is used in the smooth pair case~\cite{MP06}. We briefly describe the argument. 

A secondary term has a well-defined tangency order {\it along $E$}. Given $\gamma$, this is determined by the slopes of the rays in the $E$-direction (i.e., the $[0,1]$-direction) in $\Delta(\mathcal Y_0)$. For each $\mathbb{P}$-face vertex of $\gamma$, a subset of the descendant conditions associated to the points in the set $S$ is assigned to that vertex. One particular assignment corresponds to the term $\gamma_\circ$. The secondary terms arise from varying the distribution of descendants. The geometric possibilities are described in detail in~\cite[Section~1.4, Case I]{MP06}.

For partitions $\nu$ and $\nu'$, we say that
\[
\nu > \nu'
\]
if, after arranging the parts in decreasing order, the first entry at which they differ is larger for $\nu$ than for $\nu'$. In particular, the case of tangency $1$ at all points corresponds to the smallest term.

Now proceed by induction, also with respect to the tangency along $E$. To conclude, the principal term in the descendant approximation in this calculation is a nonzero multiple of the cycle being approximated. To see this, we apply the logarithmic degeneration formula to the principal term $\gamma_\circ$. The splitting formula of~\cite[Section~8]{MR23} applies directly. Note that because the $\mathbb{P}$-face vertices are fiber classes that meet the double locus in a single, maximally tangent point, no blowups are required in the degeneration formula. Indeed, the tropical evaluation maps from these vertices are easily seen to be flat; see also Section~\ref{sec: translated-stars}. The result now follows from an explicit evaluation of the genus $0$ fiber integrals. The calculation is identical to~\cite[Section~1.4, Case I]{MP06}.\qed

\section{Putting the pieces together}\label{sec: final-proofs}

We now record how the different results are combined to prove the main theorems about GW cycles being tautological, as stated in the introduction. 

\subsection{Proof of Theorem~\ref{thm: tautological-chow}} Let $\mathbb P$ be a product of $k$ projective spaces and let
\[
Y = X_1\cap\cdots\cap X_m
\]
be a generic complete intersection, with $X_i$ having multidegree $(d_{i1},\ldots, d_{ik})$. We take the ambient insertions to be strata classes on $\mathbb P$. Our goal is to show that all GW cycles of $Y$ with ambient insertions are tautological. 

\noindent
{\sc Step I. A degeneration via Newton polygons.} We will do so by producing a degeneration of $Y$, and in turn by degenerating $\mathbb P$. This is a standard construction in toric and tropical geometry, and goes back to the early mirror symmetry literature. We refer the reader to~\cite{Gub13,NO22} for some recent expositions. 

We first focus on a single hypersurface $X_i$. To it, we associate a Newton polytope. The Newton polytope of projective space of dimension $r$ polarized by $\mathcal O(d)$ is the $d$-times dilated $r$-simplex. In this way, each tuple $(d_{i1},\ldots, d_{ik})$ determines a product of simplices, whose dimensions match the factors of $\mathbb P$ and whose dilations match the $d_{ij}'s$. Let $\Delta_1,\ldots, \Delta_m$ be the Newton polytopes associated with these hypersurfaces. 

Choose a unimodular triangulation $\widetilde \Delta_i$ of each $\Delta_i$. Such triangulations exist. For example, one can construct a unimodular triangulation of a dilated simplex by~\cite[Theorem~4.8]{HPPS21} in each factor and then take the staircase triangulation of the product~\cite[Proposition~2.11]{HPPS21}. 

For each $i$, the dual tropical hypersurface of $\widetilde\Delta_i$ is a polyhedral decomposition of the cocharacter lattice $N_\RR$ of the dense torus of $\mathbb P$. The cone over this polyhedral decomposition produces a degeneration, often called the {\it Mumford degeneration}, of $\mathbb P$ into a union of toric varieties:
\[
\mathbb P(\widetilde \Delta_i)\to\mathbb A^1.
\]
The degeneration is typically not simple normal crossings, but it is, by construction, toric. Since the triangulation $\widetilde \Delta_i$ is into unimodular simplices, each component of the special fiber is a projective space of dimension equal to $\mathbb P$. 

We can now choose a hypersurface in the degeneration
\[
\mathbb P(\widetilde \Delta_i)\to\mathbb A^1
\]
whose general fiber is hypersurface of multidegree $(d_{ij})_j$, which we may as well identify with $X_i$, and whose special fiber is simple normal crossings. See for instance~\cite[Example~3.6]{Gr11}. By standard toric geometry, the intersection of the hypersurface with each of the projective spaces in the general fiber has degree $1$, each component of the degeneration of $X_i$ is a projective space. 

Similarly, we can construct such a polyhedral decomposition for each $\Delta_i$ and take their common refinement to obtained a polyhedral decomposition of $N_\RR$. This produces a degeneration
\[
\widetilde \PP\to \A^1
\]
that dominates each $\mathbb P(\widetilde \Delta_i)$. The morphism $\mathbb P\to \mathbb P(\Delta_i)$ is a toric modification centered in the special fiber, and each hypersurface is transverse to the strata. It follows that the intersection of their pullbacks produces a degeneration $\mathcal Y\to \A^1$ of $Y$, embedded in
\[
\widetilde \PP\to \A^1.
\]

We would like to now apply the degeneration formula to $\mathcal Y$, however since the general fiber $Y$ has not been equipped with any logarithmic structure, the logarithmic structure for the degeneration $\mathcal Y\to \A^1$ must be taken to be the special fiber. 

\noindent
{\sc Step II. Resolving the degeneration.} We now perform toric semistable reduction~\cite{KKMSD} to the morphism $\widetilde \PP\to \A^1$ and so produce a simple normal crossings degeneration
\[
\PP^\dagger\to\A^1.
\]
The semistable reduction procedure involves first passing to a ramified base change and then perform a toric modification of the pullback. The base change does not change the special fiber, but is done so the blowup produces a reduced special fiber. We can arrange for the general fiber to still be $\mathbb P$. 

The new scheme theoretic closure of $\mathcal Y_\eta$ in this degeneration to obtain a new degeneration
\[
\mathcal Y^\dagger\to\A^1,
\]
We will calculate the GW cycles of $Y$ by using this degeneration. Note the embedding $\mathcal Y^\dagger \hookrightarrow \PP^\dagger$, and so the ambient insertions from $\PP$ can be specialized by taking closures of strata. 

We can describe the special fiber concretely by performing ``the same'' base change and blowups to $\mathcal Y$ as we have to $\mathbb P$ -- for each blowup of a stratum of $\mathbb P$ we blowup the intersection of the center with $\mathcal Y$.

\noindent
{\sc Step III. Strata of the degenerate fiber.} By the semistable reduction procedure, each stratum of $\mathcal Y^\dagger_0$ is obtained by taking a stratum of $\mathcal Y_0$, and iteratively passing to toric bundle and then blowups again, and so on. As noted, each component of $\widetilde \PP$ itself is a projective space. By the unimodularity of the subdivision, each stratum of $\mathcal Y_0$ is therefore a linear space, and the natural divisor is a transverse union of hyperplanes. 

From the paragraph above, we see that all strata of the special fiber of $\mathcal Y^\dagger_0$ fall into the class of {\it linear varieties} in Totaro's sense~\cite{Tot14} and in particular, satisfy the Chow--K\"unneth property. We can therefore apply the degeneration formula in Chow~\cite{MR23}. 

We are left to determine the GW cycles of the strata of the special fiber of $\mathcal Y_0^\dagger$ with insertions either coming from the decomposition of the diagonal, which exists in Chow as we have noted, or by strata in $\PP^\dagger$. The strata are obtained from projective spaces relative to a union of hyperplanes by the sequence of projective bundle and blowup moves described above. The centers of the blowups are always smaller strata of this type. We claim the result is always tautological. 

To make the conclusion, we use the logarithmic/absolute correspondence to reduce to a statement about absolute Gromov--Witten theory of the same varieties. For projective space itself, this follows from torus localization~\cite{GP99}. For the blowup and projective bundle procedure, we can use the argument in \cite{MP06}. Each center is an intersection of divisorial strata, so the blowup algorithm in that paper applies.  Since the divisors all satisfy the Chow--K\"unneth property, the argument can be run in Chow, and preserves the tautological property. \qed

The only thing we have used about the specific geometry of products of projective space is that the Newton polygons are {\it breakable} in the sense defined in the introduction. 

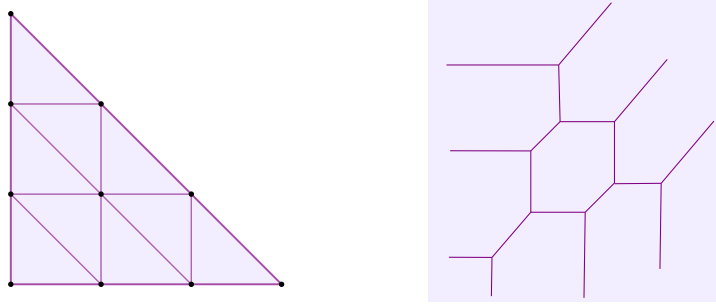
\begin{figure}
\begin{minipage}{0.45\textwidth}
\centering
\begin{tikzpicture}[scale=1.2]

  \definecolor{softlavender}{RGB}{223,210,255}

  \foreach \x in {0,1,2} {
    \foreach \y in {0,1,2} {
      \pgfmathtruncatemacro{\s}{\x+\y}
      \ifnum\s<3
        \fill[softlavender, opacity=0.4] (\x,\y) -- (\x+1,\y) -- (\x,\y+1) -- cycle;
        \draw[violet, line width=0.5pt, opacity=0.4] (\x,\y) -- (\x+1,\y) -- (\x,\y+1) -- cycle;
      \fi
      \pgfmathtruncatemacro{\sright}{\x+\y+2}
      \ifnum\sright<4
        \fill[softlavender, opacity=0.4] (\x+1,\y) -- (\x+1,\y+1) -- (\x,\y+1) -- cycle;
        \draw[violet, line width=0.5pt, opacity=0.4] (\x+1,\y) -- (\x+1,\y+1) -- (\x,\y+1) -- cycle;
      \fi
    }
  }

  \draw[violet, thick, opacity=0.5] (0,0) -- (3,0) -- (0,3) -- cycle;

  \foreach \x in {0,1,2,3} {
    \foreach \y in {0,1,2,3} {
      \pgfmathtruncatemacro{\s}{\x+\y}
      \ifnum\s<4
        \fill[black] (\x,\y) circle (0.3mm);
      \fi
    }
  }

\end{tikzpicture}
\end{minipage}%
\hspace{-2cm}%
\begin{minipage}{0.45\textwidth}
\centering
\begin{tikzpicture}[scale=2.25, rotate=90]

  \definecolor{softlavender}{RGB}{223,210,255}

  \fill[softlavender, opacity=0.4] (3.8,16.0) rectangle (5.6,17.7);

  \draw[violet]   (4.4991,16.5976) -- (4.8617,16.5976) -- (4.8617,16.9189) -- (4.6887,17.092) -- (4.3261,17.092) -- (4.3261,16.7706) -- cycle;
  \draw[violet]    (4.8617,16.5976) -- (5.2298,16.2859);
  \draw[violet]    (4.8617,16.9189) -- (5.1969,16.9272);
  \draw[violet]    (4.4971,16.3217) -- (4.4948,16.5976);
  \draw[violet]    (4.4971,16.3217) -- (4.8652,16.01);
  \draw[violet]    (4.4971,16.3217) -- (3.991,16.3289);
  \draw[violet]    (5.1969,16.9272) -- (5.1969,17.59);
  \draw[violet]    (5.1969,16.9272) -- (5.565,16.6155);
  \draw[violet]    (4.6897,17.5685) -- (4.6887,17.092);
  \draw[violet]    (4.3261,16.7706) -- (3.82,16.7778);
  \draw[violet]    (4.0589,17.3213) -- (4.3261,17.092);
  \draw[violet]    (4.0589,17.3213) -- (3.8299,17.3249);
  \draw[violet]    (4.0599,17.5757) -- (4.0589,17.3213);

\end{tikzpicture}

\end{minipage}
\caption{A unimodular subdivision giving rise to a degeneration of a smooth elliptic curve in $\mathbb P^2$. The $\mathbb P^2$ breaks into $9$ $\mathbb P^2$'s glued as the triangles are glued above. The elliptic curve breaks into $9$ copies of $\mathbb P^1$. The cone over a hexagon causes a toric threefold singularity in the total space.}
\end{figure}

\subsection{Proof of Theorem~\ref{thm: tautological-gw}} Let $\mathcal X\to B$ be a degeneration as stated in the theorem. We apply the degeneration formula of Theorem~\ref{thm: degeneration-formula} and the vanishing cohomology results of Theorem~\ref{thm: primitive-in-exotic}. This converts the GW cycles into a collection of exotic logarithmic GW cycles on broken toric bundles over the strata of a singular fiber $\cX_0$.  Now use the exotic/non-exotic move of Theorem~\ref{thm: exotic-non-exotic} to remove the exotic insertions and reduce to the ordinary insertions on broken toric bundles over strata in the special fiber $\mathcal X_0$. We can then appeal to the bundle reconstruction results of Section~\ref{sec: broken-bundles}, in particular to Corollary~\ref{cor: collapsing-log-toric-bundles}, to reduce to ordinary insertions on the spaces of stable maps to the strata. Finally, switch off the logarithmic structure by Theorem~\ref{thm: log-absolute}. The proof is complete \qed

\subsection{Proof of Corollary~\ref{cor: breakable}} The proof of Theorem~\ref{thm: tautological-chow} recorded above may be followed step-by-step to establish Corollary~\ref{cor: breakable} from Theorem~\ref{thm: tautological-gw}. \qed

\bibliographystyle{siam} 
\bibliography{TautologicalGW}

\end{document}